\documentclass[a4, 12pt]{amsart}
\usepackage{amstext,amscd,amsthm,amsfonts,mathdots,mathrsfs}
\usepackage{amssymb, upgreek}
\usepackage{bm}
\usepackage{latexsym,todonotes}
\usepackage{pdflscape}
\usepackage[all]{xypic}
\usepackage{graphicx}
\usepackage{pdflscape}

\usepackage{setspace}
\usepackage{enumerate}
\usepackage{color}
\usepackage{tikz}
\numberwithin{equation}{section}
\newtheorem{theorem}{Theorem}[section]
\newtheorem{lemma}[theorem]{Lemma}
\newtheorem{proposition}[theorem]{Proposition}
\newtheorem{definition}[theorem]{Definition}

\newtheorem{corollary}[theorem]{Corollary}

\newtheorem{remark}[theorem]{Remark}
\newcommand{\clr}{rgb:black,2;blue,2;red,0}

\tikzset{anchorbase/.style={baseline={([yshift=-0.5ex]current bounding box.center)}}}
\tikzset{wipe/.style={white,line width=4pt}}

\DeclareMathOperator{\End}{End}

\DeclareMathOperator{\Hom}{Hom}
\DeclareMathOperator{\id}{id}

\def\g{\mathfrak{g}}
\def\CC{\mathbb{C}}

\def\ZZ{\mathbb{Z}}
\def\cA{\mathcal{A}}
\def\R{\mathscr{R}}
\def\U{{\rm U}_q(\mathfrak{g})}
\def\UG{{{\rm U}_q}}
\def\AA{{\mathbb A}}
\def\UD{{{\rm U}_q^0}}
\newcommand{\ids}
{\begin{tikzpicture}[baseline=4pt,scale=1,color=\clr]
\draw[-,line width=1pt](0,0)--(0,0.5);
\end{tikzpicture}}

\newcommand{\cups}
{\begin{tikzpicture}[baseline =3pt,scale=1,color=\clr]
	\draw[-,line width=1pt](0,0.5)parabola bend (0.3,0) (0.6,0.5);
\end{tikzpicture}}

\newcommand{\caps}
{\begin{tikzpicture}[baseline=3pt,scale=1,color=\clr]
\draw[-,line width=1pt,color=\clr](0,0)parabola bend (0.3,0.5) (0.6,0);
\end{tikzpicture}}

\newcommand{\rectangle}
{\begin{tikzpicture}[baseline=4pt,scale=0.5,color=\clr]
		\draw[-,line width=1pt](0,0)rectangle(2,1);	
		\node at(1,0.5){$R_{n}$};
\end{tikzpicture}}

\newcommand{\horiz}
{\begin{tikzpicture}[baseline=2pt,scale=0.5,color=\clr]
		\draw[-,line width=1pt](8.3,1)--(8.8,0.5);
		\draw[-,line width=1pt](8.3,0)--(8.8,0.5);
		\draw[-,line width=1pt](8.8,0.5)--(9.2,0.5);
		\draw[-,line width=1pt](9.2,0.5)--(9.7,1);
		\draw[-,line width=1pt](9.2,0.5)--(9.7,0);
\end{tikzpicture}}

\newcommand{\cross}
{\begin{tikzpicture}[baseline=18pt,scale=0.5,color=\clr]
	\draw[-,line width=1pt] (0.5,2)--(1.5,1);
	\draw [-,line width=1pt](0.5,1)--(0.9,1.4);
	\draw [-,line width=1pt](1.1,1.6)--(1.5,2);
\end{tikzpicture}}

\newcommand{\Icross}
{\begin{tikzpicture}[baseline=18pt,scale=0.5,color=\clr]
	\draw [-,line width=1pt](5.5,1)--(6.5,2);
	\draw[-,line width=1pt] (6.5,1)--(6.1,1.4);
	\draw [-,line width=1pt](5.9,1.6)--(5.5,2);
\end{tikzpicture}}

\newcommand{\merges}
{\begin{tikzpicture}[baseline = -.5mm,scale=0.8,color=\clr]
		\draw[-,line width=1pt] (0.28,-.3) to (0.08,0.04);
		\draw[-,line width=1pt] (-0.12,-.3) to (0.08,0.04);
		\draw[-,line width=1pt] (0.08,.4) to (0.08,0);
\end{tikzpicture} }

\newcommand{\splits}
{\begin{tikzpicture}[baseline = -.5mm,scale=0.8,color=\clr]
		\draw[-,line width=1pt] (0.08,-.3) to (0.08,0.04);
		\draw[-,line width=1pt] (0.28,.4) to (0.08,0);
		\draw[-,line width=1pt] (-0.12,.4) to (0.08,0);
\end{tikzpicture}}

\newcommand{\upgaprimeform}{\begin{tikzpicture}[baseline=-10pt,scale=1.1,color=\clr]
		\draw[-,line width=1pt](0,0)parabola bend (0.3,-0.4) (0.6,0);\draw[-,line width=1pt](0.3,-0.4)--(0.3,-0.2);
		\draw[-,line width=1pt](0.3,-0.2)--(0.1,0);\draw[-,line width=1pt](0.3,-0.2)--(0.5,0);
\end{tikzpicture}}

\newcommand{\ele}
{\begin{tikzpicture}[baseline=48pt,scale=0.6,color=\clr]
\draw[-,line width=1pt](1,2.6)rectangle(3,3.4);\node at (2,3){$inv$}; 
\end{tikzpicture} }

\newcommand{\beq}{\begin{eqnarray}}
	\newcommand{\eeq}{\end{eqnarray}}

\newcommand{\baln}{\begin{aligned}}
	\newcommand{\ealn}{\end{aligned}}

\newcommand{\ot}{\otimes}
\newcommand{\lra}{\longrightarrow}

\newcommand{\ol}{\overline}

\newcommand{\CT}{{\mathcal T}}
\newcommand{\CH}{{\mathcal H}}

\newcommand{\XX}{{\mathbb X}}

\newcommand{\Acycl}{{\rm Acycl}}
\bibliography{ref}
\begin{document}

\title[Invariant theory of quantum $G_2$]
{Braided symmetric algebras and a 
first fundamental theorem of invariant theory for ${\rm U}_q(G_2)$} 
\author{Hongmei Hu}
\address[H.M. Hu]{Department of Mathematics, Shanghai Maritime University, Shanghai, China.}
\email{hmhu@shmtu.edu.cn}
\author{Ruibin Zhang}
\address[R.B. Zhang]{School of Mathematics and Statistics, University of Sydney, Sydney, Australia.}
\email{ruibin.zhang@sydney.edu.au}

\subjclass[2020]{Primary 17B37, 16T20, 17B10}

\keywords{quantum groups, braided symmetric algebras, invariant theory, tangle diagrams.}
	
	\begin{abstract} 
	We develop invariant theory for the quantum group ${\rm U}_q$ of $G_2$ at generic $q$ 
	by studying its actions on a class of non-commutative module algebras.  
	Let ${\mathcal A}_m$ be the braided symmetric algebra over $m$-copies of the $7$-dimensional 
	simple ${\rm U}_q$-module.  
	A set of ${\rm U}_q$-invariants in ${\mathcal A}_m$ attached to certain acyclic trivalent graphs 
	is obtained, which spans the subalgebra ${\mathcal A}_m^{{\rm U}_q}$ of invariants 
	as vector space.   
	A finite set of homogeneous elements is constructed explicitly, 
	which generates ${\mathcal A}_m^{{\rm U}_q}$ as algebra. 
	Commutation relations among the algebraic generators are determined. 
	These results amount to a non-commutative first fundamental theorem of 
	invariant theory for ${\rm U}_q$. 	
	The algebra ${\mathcal A}_m$ is a non-flat quantisation of the coordinate ring of 
	${\mathbb C}^7\otimes{\mathbb C}^m$.  As ${\rm U}_q$-module, 
	${\mathcal A}_m={\mathcal A}_1^{\otimes m}$ and we decompose ${\mathcal A}_1$ 
	into simple submodules. The affine scheme associated to the
	classical limit of ${\mathcal A}_m$ is described. 
	This is a rare case where the structure of a non-flat quantisation is understood. 

	\end{abstract}
	\maketitle
	\setcounter{tocdepth}{2}
	\tableofcontents

\section{Introduction}{\ }

\noindent{1.1.}
We develop invariant theory for the quantum group ${\rm U}_q(G_2)$ by studying its actions on a class of 
non-commutative module algebras known as braided symmetric algebras. 
A first fundamental theorem (FFT) of invariant theory will be established in this context, 
and an explicit description of the non-commutative algebra of invariants is given.  

Significant new characteristics of the present case render the invariant theory of ${\rm U}_q(G_2)$ 
materially different from those of the quantum (super)groups studied so far \cite{CW-queer, lzz:ft, LZZ-2, Zy} (also see \cite{LSS}). 
%
The braided symmetric algebras for ${\rm U}_q(G_2)$ investigated here are 
non-flat in the sense of \cite{bz:bs} (in fact there exist no flat braided symmetric algebras with non-trivial ${\rm U}_q(G_2)$-actions), and it is a very interesting and highly non-trivial matter to understand their structures.
%
The algebra of invariants relevant to the FFT of ${\rm U}_q(G_2)$ is 
generated by less generators than for classical $G_2$ \cite{HZ, S} (also see  \cite{ZS}), where
the additional generators of the classical invariants,  morally, 
come from commutators of quantum generators in the semi-classical limit. 
These are indications of the intricacy of the ${\rm U}_q(G_2)$ case, and indeed 
the invariant theory in this case is much more challenging to study.

%
%

\medskip
\noindent{1.2.}
Invariant theory lies at the foundation of symmetry in physics.   
The elemental situation is that the Hamiltonian of a quantum system 
is invariant under the action of some algebraic structure $G$, such as a Lie group,  
quantum group, or a super analogue of them.  In this case, 
the quantum system is said to possess a $G$-symmetry,  
the representation theory of which controls key properties of the system. 
An example is the ${\rm SU}(3)\times {\rm SU}(2)\times {\rm U}(1)$ 
symmetry of the standard model of particle physics. Its importance 
to our understanding of nature can not be overstated.

One can develop invariant theory in several generic settings. 
A traditional setting for the classical invariant theory \cite{W-cl} of Lie groups is 
that of commutative algebras.
Let $M$ be a module for a Lie group $G$, 
and let $S(M)$ be the symmetric algebra over $M$. 
One describes the generators of the subalgebra of $G$-invariants of $S(M)$, 
and determines the relations obeyed by the generators. 
Such results for $M$ being multiple copies of the natural module for  the orthogonal or symplectic group, 
and of $V\oplus V^*$ for the general linear group ${\rm GL}(V)$, are pillars of invariant theory, 
which were coined the first and second fundamental theorems (FFT \& SFT)
of invariant theory for the classical groups by Weyl  loc. cit.. 

\medskip
\noindent{1.3.}
It is highly desirable to develop an analogue of this commutative algebra version of classical 
invariant theory for quantum (super)groups, from the point of view of applications of quantum (super)groups to physics and to other areas of mathematics, e.g., representation theory and low dimensional topology.  
This was done in \cite{lzz:ft} by extending the entire conceptual framework, where 
$S(M)$ was replaced by a certain 
$\U$-module (super)algebra $\cA(M)$, 
called a braided symmetric algebra over the $\U$-module $M$. 
This algebra is necessarily non-commutative unless the action on $M$ is trivial,
due to the non-co-commutative nature of $\U$ as Hopf (super)algebra.  
Thus the arsenal of commutative algebra techniques 
for classical invariant theory no longer apply to the quantum case. 
Furthermore,  the subalgebra $\cA(M)^{\U}$
of invariants also becomes non-commutative in the quantum case. 
This, on one hand, makes it much harder to describe all relations among the generators 
of $\cA(M)^{\U}$ to establish a quantum analogue of SFT, on the other, makes the study of 
the structure of $\cA(M)^{\U}$ a much more interesting problem.   

Quantum FFTs in the non-commutative algebra setting described above 
were established, first for the quantum groups $\U$ associated with the classical Lie algebras 
$\mathfrak{gl}_m$, $\mathfrak{so}_m$, and
$\mathfrak{sp}_{2n}$ \cite{lzz:ft}, and then generalised to the quantum supergroups 
associated with the general linear superalgebra $\mathfrak{gl}_{m|n}$ \cite{Zy}
and queer superalgebra $\mathfrak{q}_n$ \cite{CW-queer}.  
Commutation relations among generators of the subalgebra of invariants 
were also obtained. 

We point out that the non-commutative invariant theories for $\U$ associated with 
$\mathfrak{gl}_{m}$, $\mathfrak{gl}_{m|n}$ and $\mathfrak{q}_n$ 
relied on another aspect of invariant theory, namely, 
quantum Howe dualities developed in \cite{Lai-Z, WZ-hd, z:hd}. 
Very recently, the authors of \cite{LSS} generalised the work on 
${\rm U}_q(\mathfrak{gl}_m)$ in \cite{lzz:ft} 
 to develop a theory of equivariant difference operators, and applied it to refine the 
$({\rm U}_q(\mathfrak{gl}_m), {\rm U}_q(\mathfrak{gl}_n))$ Howe duality of \cite{z:hd}.

We remark that the construction and analysis of braided symmetric algebras $\cA(M)$ 
over $\U$-modules are an interesting problem \cite{bz:bs, z:pa} in its own right.
Such algebras are a natural class of objects to study for non-commutative geometries 
admitting quantum group actions.   
However, when $M$ is non-flat in the sense of \cite{bz:bs}, 
hardly anything is known about the structure of such $\cA(M)$. 

\medskip
\noindent{1.4.}
The main objective of the present paper is to develop a non-commutative algebraic version of the invariant theory for the quantum group ${\rm U}_q(G_2)$.
%
Another aim is to understand examples of $\cA(M)$ which are not flat. 

%

We denote  by $\UG$ the quantum group ${\rm U}_q(G_2)$ for simplicity.  
Let $V$ be the $7$-dimensional simple $\UG$-module. There is a $\UG$-submodule $\wedge_q^2 V$ in $V\ot V$, which is the analogue of the skew symmetric tensor.  It generates a two-sided ideal $\langle\wedge_q^2 V\rangle$ in the tensor algebra  $T(V)$ over $V$, 
and the quotient algebra $S_q(V)=T(V)/\langle\wedge_q^2 V\rangle$ is a $\UG$-module algebra. 
The $m$-th tensor power $\cA_m(V)$ of $S_q(V)$, for any $m\ge 0$,  
has a natural $\UG$-module algebra structure
with a braided multiplication defined using the universal $R$-matrix. 
It will be referred to as a braided symmetric algebra. 

Our investigation on the invariants of $\cA_m(V)$ yields the following results. 
\begin{enumerate}[(i)]
\item \label{2} 	We obtain a set of invariants labelled by certain acyclic trivalent graphs, 
		which is shown in Theorem \ref{thm:span} to span the underlying vector space 
		of the subalgebra $\cA_m(V)^{\UG}$ of $\UG$-invariants. 
		We construct explicitly a finite set of elements  
		$\Phi^{(i, j)}$ and	$\Psi^{(r,s,t)}$, 
		where $1\le i\leq j\le m$, $1\le r<s<t\le m$, 
		and show  in Theorem \ref{FFT} that they generate 
		$\cA_m(V)^{\UG}$ as an algebra.

\item \label{3}  We determine the commutation relations among the generators 
		$\Phi^{(p,q)}$ and	$\Psi^{(r,s,t)}$ of $\cA_m(V)^{\UG}$ in Section \ref{sect:struct}. 
		An intriguing fact, which is very different from the classical case \cite{HZ, S},
		is that $q$-commutators of 
		the generators give rise to higher degree invariants associated 
		with certain acyclic trivalent graphs, e.g.,  
		 $q$-commutators of $\Phi^{(p,q)}$'s lead to the degree $4$ invariants
		$\Upsilon^{(i, j, k, \ell)}$ introduced in Section \ref{sect:elmt}. 
		Ordered monomials of $\Phi^{(p,q)}$, $\Psi^{(r,s,t)}$, $\Upsilon^{(i, j, k, \ell)}$ 
		and such higher degree invariants span $\cA_m(V)^{\UG}$
		(see Section \ref{comments} for further discussion).  
		
\end{enumerate}

These results in particular enable one to describe the subspace $\left(V^{\ot m}\right)^{\UG}$ 
 of invariant tensors in the $m$-th tensor power of $V$ for any $m$ 
(see Corollary \ref{corr:tensor-FFT}).  
They amount to a first fundamental theorem of invariant theory for $\UG$. 	
Results (ii) represent a crucial step toward the complete
understanding of the structure of the non-commutative algebras $\cA_m(V)^{\UG}$. 

A new feature, which is not present in the study of the invariant theory for quantum groups 
associated with classical Lie algebras \cite{lzz:ft},  is that 
$S_q(V)$ is ``smaller" than the symmetric algebra over the classical limit 
$\ol{V}=\CC^7$ of $V$ (i.e., specialisation of $V$ at $q\to 1$). 
More precisely, the degree $n$ homogeneous component $S_q(V)_n$ 
satisfies $\dim_{\CC(q)} S_q(V)_n < \dim_\CC S^n(\ol{V})$ for all $n\ge 3$.   
Thus the algebra $\cA_m(V)$ is a ``non-flat deformation" of the coordinate ring of $\ol{V}^{\oplus m}$ in the terminology of \cite{bz:bs}. 
In fact, $S_q(V)$ and hence also $\cA_m(V)$ have nilpotent elements. If 
the algebra $\cA_m(V)$ is regarded as defining some non-commutative space, 
a natural question is the corresponding ``space" in the classical limit. 

Another result of the paper is the following. 

\begin{enumerate}[(i)]
\setcounter{enumi}{2}
\item \label{1} We decompose $S_q(V)$ into a direct sum of simple $\UG$-submodules 
	and determine their multiplicities in Theorem \ref{homogeneous}. It is shown that 
	$S_q(V)_n$ is the quantum analogue of harmonic space at degree $n$ for all 
	$n\ne 2$;  and $S_q(q)_2$ is the direct sum of the degree $2$ quantum 
	of harmonic space with a $1$-dimensional subspace.  
 	Using Theorem \ref{homogeneous}, we construct the affine scheme associated 
	with the classical limit of $\cA_m(V)$.   
\end{enumerate}
This is the only example beside the case of the $4$-dimensional simple module for ${\rm U}_q(\mathfrak{sl}_2)$ \cite{r:ar}, 
where the decomposition of $S_q(V)$ and 
the classical limit of $\cA_m(V)$ are completely understood.

\medskip
\noindent{1.5.}  
Now we briefly explain the key ideas involved in the proofs of the main results. 


The result (\ref{1}) is proved by developing an alternative construction of $S_q(V)$ 
with a geometric flavour, which can be described as follows.
As a $\ZZ_+$-graded $\UG$-algebra,  $S_q(V)\simeq \Gamma\oplus \CC(q)\theta$, 
where $\Gamma$ is a $\ZZ_+$-graded subalgebra, 
and $\CC(q)\theta$ is a two-sided ideal which is homogeneous of degree $2$. 
The homogeneous components of $\Gamma$ are spaces of global sections of 
certain quantum line bundles on the quantum flag manifold  \cite{GZ} of $\UG$.
They are constructed using the quantum Borel-Weil theorem \cite{APW, GZ}, 
and can also be interpreted as quantisations of the subspaces of harmonics 
in $S(\ol{V})$ (see \cite[\S 4]{Lai-Z} for the general case). 

Results (\ref{2}) and (\ref{3}) are proved by investigating certain natural surjective linear maps 
from invariant tensors $\left(V^{\ot r}\right)^{\UG}$ 
to homogeneous subspaces of $\cA_m(V)^{\UG}$ with respect to a 
natural $\ZZ_+^m$-grading of $\cA_m(V)$.  
This makes essential use of a diagrammatic 
description of the invariant tensors given in Section \ref{diagdep}. 
We consider a diagram category 
${\mathscr T}(\alpha, \beta)$ of unoriented framed tangles with a coupon, 
which depends on two parameters $\alpha$ and $\beta$. 
It follows from results of  \cite{kuper:s, lz:sm, Ms} that  for 
$\alpha = -q^{-6}$ and $\beta=[7]_q-1$, 
there is an essentially surjective and full tensor functor 
from the category to the full subcategory ${\mathscr V}$ of $\UG$-modules with 
objects $V^{\ot r}$ for $r\in\ZZ_+$ (see Theorem \ref{thm:functor}).  
Furthermore, $\Hom_{\UG}(V^{\ot r}, V^{\ot s})$ 
can be described in terms 
of the special class of diagrams given in Theorem \ref{lem:basis}.

The vector space isomorphism 
$\left(V^{\ot r}\right)^{\UG} \simeq \Hom_{\UG}(\CC(q), V^{\ot r})$
leads to a diagrammatic description of $\left(V^{\ot r}\right)^{\UG}$, 
which enables us to prove the results (\ref{2}) and (\ref{3}). 

This diagrammatic method was introduced  in \cite{lzz:ft} for the quantum groups 
associated with classical Lie algebras, but its adaption to the ${\rm U}_q(G_2)$ case is 
technically  challenging. The diagrammatics in the present case is much more 
complicated due to the appearance of trivalent graphs, and this makes the structure of 
the subalgebra of invariants much more difficult to describe.  
An array of new techniques are needed to treat the FFT 
and commutation relations of the subalgebra of invariants.  
See the proof of Theorem \ref{FFT}  and Section \ref{sect:struct} for details. 


\smallskip
\noindent{1.6.}
We now comment on the relationship between Section \ref{sect:diagrams}
and Kuperberg's  theory of spiders for ${\rm U}_q(G_2)$  \cite{kuper:s}
and a simplified version of it in \cite{Ms}. 
Theorem 5.1 loc. cit.  in the case of ${\rm U}_q(G_2)$ gives a diagrammatic description of the morphisms 
$\Hom_{\UG}(V_{\bf i}, V_{\bf j})$ with ${\bf i}=(i_1, \dots, i_r)\in\{1, 2\}^r$ and ${\bf j}=(j_1, \dots, j_s)\in\{1, 2\}^s$ for all $r, s$, 
where $V_{\bf i}=V_{\lambda_{i_1}}\ot V_{\lambda_{i_2}}\ot \dots \ot V_{\lambda_{i_r}}$, and  
$V_{\bf j}=V_{\lambda_{j_1}}\ot V_{\lambda_{j_2}}\ot \dots \ot V_{\lambda_{j_s}}$. Here  
$V_{\lambda_1}$ and $V_{\lambda_2}$ are the simple $\UG$-modules with the fundamental weights
$\lambda_1$ and $\lambda_2$ as 
the respective highest weights  (where $V_{\lambda_1}=V$). 
The morphisms were described as images under some monoidal functor of linear combinations of  diagrams, which were generated, in our convention,  by two copies of each of the diagrams in Figure \ref{fig:generators},  
where the additional diagrams involve new arcs drawn in double-lines \cite{kuper:s}. 
In particular, there is a generator
\begin{tikzpicture}[baseline=-3pt,scale=0.65,color=\clr]
		\draw [-,line width=1pt](0,0)--(-0.5,-0.5);
		\draw [-,line width=1pt](0,0)--(0.5,-0.5);		
		\draw[-,line width=1pt] (-0.04,-0.05)--(-0.04,0.6);
		\draw [-,line width=1pt](0.04,-0.05)--(0.04,0.6);
	\end{tikzpicture}
called a web, which leads to a morphism $V_{\lambda_1}\ot V_{\lambda_1}\lra V_{\lambda_2}$.  
When composing diagrams, one is only allowed to join end points of lines of the same type. 

The presence of the additional diagrams makes the description of morphisms 
in the $\UG$ case of  \cite[Theorem 5.1]{kuper:s} a lot
more complex than that in our Theorem \ref{thm:functor}. 
However, note that
we can extract from our Theorem \ref{thm:functor} 
another diagrammatic description for $\Hom_{\UG}(V_{\bf i}, V_{\bf j})$ with   
${\bf i}\in\{1, 2\}^r$ and ${\bf j}\in\{1, 2\}^s$ for all $r, s$, 
since $V_{\lambda_2}$ is a summand of $V\ot V$ with multiplicity $1$.
Also, Morrison \cite{Ms} extracted a diagram category of trivalent graphs (no crossings) from \cite{kuper:s}, 
which was apparently equivalent to the category of $\UG$-modules with objects $V^{\ot r}$ for all $r$ \cite[Theorem 1.2]{Ms}.  

We wish to emphasise that results of 
\cite{kuper:q, kuper:s, Ms} can not be immediately applied to  prove Theorem \ref{FFT}, the FFT of invariant theory for $\UG$. The diagrammatics developed in Section \ref{diagdep} is what required for our problem. 

\smallskip
\noindent{1.7.}
Now it is appropriate to mention that a modern version of classical invariant theory 
\cite{LZZ-2,LZ-bc, LZ-rev, SW, Wh} seeks  
to describe representation categories of particular interest 
as images of diagram categories under monoidal functors.
This was partly inspired by the theory of quantum invariants of knots developed 
by Jones \cite{J}, Reshetikhin-Turaev \cite{rt:rg, rt:in}, Witten \cite{MW, Witten} and others.  
There is extensive work on such categorical invariant theory for
Lie groups, 
Lie supergroups, and their quantum analogues (see, e.g., 
\cite{BSW,LZZ-2,LZ-bc, LZ-rev, SW, Wh} and references therein), 
 leading to much progress, 
particularly in understanding Schur-Weyl dualities between diagram categories 
and categories of representations of Lie (super)groups and quantum  (super)groups.
The material of Section \ref{diagdep} in this paper lies within this categorical framework. 

In fact the construction of link invariants using braided tensor categories and quantum (super)groups \cite{rt:rg} may also be regarded as part of the invariant theory of quantum (super)groups. 
In this connection, we mention the work \cite{Zr}, which constructed topological invarants of $3$-manifolds using quantum groups associated with the exceptional Lie algebras, 
including $\UG$, at roots of unity. 
A modular category was involved in the work loc. cit., 
which is closely related to the material in Section \ref{diagdep}, 
and Theorem \ref{thm:functor} in particular. 

The categorical invariant theory and the (non) commutative algebra version of invariant theory 
investigated here are complementary to each other. We intend to make use of the categorical invariant theory for 
the quantum orthosymplectic supergroups developed in \cite{LZZ-2} to  
study a non-commutative algebraic version. \\

We work over $\CC(q)$  throughout. 
Denote $[n]_q=\frac{q^n-q^{-n}}{q-q^{-1}}$ for any non-zero $n\in \ZZ_+$.

\section{Background material}\label{sect:background}
We discuss some basic facts about the the $7$-dimensional simple module $V$
for the quantum group $\UG$, 
and introduce the quantum symmetric algebra and braided symmetric algebras over $V$.

\subsection{The standard ${\rm U}_q(G_2)$-module}\label{sect:st-mod}

We adopt the following convention for $G_2$. The simple roots will be denoted by
$\alpha_{1}$ and $\alpha_{2}$, where $\alpha_{1}$ is short with
length $\sqrt{2}$. The fundamental weights will be denoted by 
$\lambda_1$ and $\lambda_2$. The set of positive roots of $G_2$ is 
\begin{align*}
	\Phi^+&=\{\alpha_{1},\alpha_{2},\alpha_{3}=\alpha_{1}+\alpha_{2},
	\alpha_{4}=2\alpha_{1}+\alpha_{2},
	\alpha_{5}=3\alpha_{1}+\alpha_{2},
	\alpha_{6}=3\alpha_{1}+2\alpha_{2}
	\}.
\end{align*}

We present the quantum group $\UG={\rm U}_q(G_2)$ of $G_2$ in the standard way 
with the generators $E_i, F_i, K_i^{\pm}$, for $i=1, 2$, 
and the usual relations (see \cite{jan:qg}).
It is well known that $\UG$ is a Hopf algebra.
	Write $\Delta, \epsilon$ and $S$ for the co-multiplication,
	co-unit and the antipode respectively, 
 and take the following co-multiplication:
	\[\Delta(E_i)=E_i\otimes K_i+1\otimes E_i,\quad
	\Delta(F_i)=F_i\otimes 1+K_i^{-1}\otimes F_i,\quad
	\Delta(K_i)=K_i\otimes K_i.
	\]
The antipode satisfies $S^2(x)= K_{2\rho} x K_{2\rho}^{-1}$ for all $x\in \UG$, where  
$K_{2\rho}=\left(K_1^5 K_2^3\right)^2$ which, in particular,  satisfies $K_{2\rho}E_i K_{2\rho}^{-1}=q^{(\alpha_i, 2\rho)}E_i$ for $i=1, 2$, 
with $\rho=\frac{1}{2}\sum_{\alpha\in\Phi^+}\alpha$ being the Weyl vector  of $G_2$.

Let $P^+$ be the set of integral dominant $G_2$-weights. 
For any $\lambda\in P^+$, we denote by $V_\lambda$ the type-${\bf 1}$ simple $\UG$-module with the highest weight $\lambda$. 
We shall  consider 
$\UG$-modules of type-${\bf 1}$ only throughout the paper.

Let $V=V_{\lambda_1}$, which is $7$-dimensional, 
and is strongly multiplicity free in the sense of \cite{lz:sm}. 
Call $V$ the standard $\UG$-module for easy reference.  
The corresponding representation will be denoted by $\pi: {\UG} \longrightarrow \End_{\CC(q)}(V)$.
Note that $\lambda_1=\alpha_4$, and the complete set of weights of $V$ is  
\beq\label{eq:wts}
\alpha_4, \alpha_3, \alpha_1, 0,  -\alpha_1, -\alpha_3, - \alpha_4.
\eeq
We choose for $V$ an ordered weight basis 
$\{ v_1, v_2, v_3, v_0, v_{-3}, v_{-2}, v_{-1} \}$, the weights of which are \eqref{eq:wts} in the given order.  
Let $E_{ab}$ be the matrix units in $\End_\CC(V)$ relative to this basis, i.e., 
$E_{ab}(v_c)=\delta_{bc}v_a$ for all $a, b, c$. 
The standard representation of $\UG$ is given by 
	\begin{align*}
		\pi(E_1)&=E_{12}+(q+q^{-1})E_{30}-(q+q^{-1})E_{0,-3}-E_{-2,-1},\\
		\pi(E_2)&=E_{23}-E_{-3,-2};\\
		\pi(F_1)&=E_{21}+E_{03}-E_{-3,0}-E_{-1,-2},\\
		\pi(F_2)&=E_{32}-E_{-2,-3};\\
		\pi(K_1)&=qE_{11}+q^{-1}E_{22}+q^2E_{33}+E_{00}+q^{-2}E_{-3,-3}
		+qE_{-2,-2}+q^{-1}E_{-1,-1},\\
		\pi(K_2)&=E_{11}+q^3E_{22}+q^{-3}E_{33}+E_{00}+q^3E_{-3,-3}+q^{-3}E_{-2,-2}
		+E_{-1,-1}.	
	\end{align*}

It is clear from \eqref{eq:wts} that the standard $\UG$-module $V$ is self dual. 
Thus there exists a $\UG$-invariant non-degenerate 
bilinear form  $( \ , \ ): V \times V\lra \CC(q)$, which 
 is unique up to scalar multiples. We take
\begin{equation}\label{bilinear}
\begin{array}{l l l l}
(v_{1},v_{-1})=q^6, 		&(v_2,v_{-2})=q^5, 			&(v_{3},v_{-3})=q^2	, 
											&(v_0,v_{0})=1, \\
(v_{-1},v_{1})=q^{-4}, &(v_{-2},v_{2})=q^{-3}, 	&(v_{-3},v_{3})=1, 
											&\text{rest}=0.
\end{array}
\end{equation}

The quantum dimension of any weight module $M$ is defined by $\dim_q M := \textit{tr}_M\left(K_{2\rho}\right)$.  For the standard module, 
\beq
\dim_q V =  q^{10} + q^{8} + q^{2} +1  + q^{-2}  + q^{-8}+ q^{-10}. 
\eeq

The tensor product $V\ot V$ decomposes into
\beq \label{eq:VtV}
V\otimes V=V_{2\lambda_{1}}\oplus V_{\lambda_2}\oplus V_{\lambda_{1}}\oplus V_{0},
\eeq
where $V_{\lambda_{1}}\cong V$ and  $V_0=\CC(q)$,  
and  the other simple submodules respectively have dimensions 
$\dim(V_{\lambda_2})=14$,  and
$\dim(V_{2\lambda_{1}})= 27$. 

\begin{remark}\label{lem:all-sim}
The fundamental modules $V_{\lambda_1}=V$ and $V_{\lambda_2}$ are submodules of $V\ot V$,  
thus all simple modules $V_\lambda$ with $\lambda\in P^+$  can be obtained as summands of tensor powers of $V$.
\end{remark}

We will need explicit bases of the simple submodules of $V\ot V$, both for developing presentations of, 
and constructing $\UG$-invariants in, braided symmetric algebras over $V$.

We have constructed the following weight bases for the simple submodules.

\begin{itemize}

\item Basis $B(V_0)$ for $V_0$: 
	\beq\label{basis0}
	c_0&=&	q^4v_1\otimes v_{-1}{+}q^{-6}v_{-1}\otimes v_{1}{+}q^3v_2\otimes v_{-2}\\
&&{+}q^{-5}v_{-2}\otimes v_{2}{+}v_3\otimes v_{-3}{+}q^{-2}v_{-3}\otimes v_{3}{+}v_0\otimes v_0. \nonumber
	\eeq	

\item Basis $B(V_{2\lambda_1})$ for $V_{2\lambda_1}$:
\beq	\label{basis20}
\begin{aligned}
		&	v_i\otimes v_i,\quad v_{-i}\otimes v_{-i}, \quad i=1,2,3,\\
		&	v_{i}{\otimes} v_j{+}b_{(i j)} v_j{\otimes} v_i,\ 
		v_{{-}j}{\otimes}v_{{-}i}{+}b_{(i j)} v_{{-}i}{\otimes} v_{{-}j}, \ \text{for $(i j)\in \widetilde{\mathcal P}$}, \\ 
		&v_0\otimes v_1+q^{-2}v_1\otimes v_0+q^{-1}v_2\otimes v_3+q^{-2}v_3\otimes v_2,\\
		&q^2v_{-1}\otimes v_0+v_{-3}\otimes v_{-2}+qv_{-2}\otimes v_{-3}+v_0\otimes v_{-1},\\
		&(q^{-2}+1)\left(v_0\otimes v_2+v_2\otimes v_0\right)-v_{-3}\otimes v_1-q^{-3}v_1\otimes v_{-3},\\
		&v_0\otimes v_{-2}+v_{-2}\otimes v_0-q^{-2}v_3\otimes v_{-1}-qv_{-1}\otimes v_3,\\
		&v_{-2}\otimes v_1+q^{-3}v_1\otimes v_{-2}+(q^{-2}+1)\left(v_3\otimes v_0+v_0\otimes v_3\right),\\
		&v_{-1}\otimes v_2+q^{-3}v_2\otimes v_{-1}+q^{-1}v_0\otimes v_{-3}+q^{-1}v_{-3}\otimes v_0,\\
		&v_{-2}\otimes v_2+q^{-4}v_2\otimes v_{-2}-q^{-1}v_3\otimes v_{-3}-q^{-3}v_{-3}\otimes v_3,\\
		&\left(q^{-2}+1\right)v_0\otimes v_0-v_{-3}\otimes v_3-q^{-4}v_3\otimes v_{-3},\\
		&2(q^{-2}{+}1)v_0{\otimes }v_0{+}q^{-3}v_2{\otimes} v_{-2}{+}q^{-1}v_{{-}2}{\otimes} v_2-v_{{-}1}{\otimes} v_1-q^{-4}v_1{\otimes} v_{{-1}}\\
		&-(q^{-2}{+}1)\left(v_{{-}3}{\otimes} v_3{+}q^{{-}2}v_3{\otimes} v_{{-}3}\right),
\end{aligned}
\eeq
where $\widetilde{\mathcal P}=\widetilde{\mathcal P}_1\bigcup \widetilde{\mathcal P}_2\bigcup \widetilde{\mathcal P}_3$,  and $b_{(i j)}=q^\ell \ \mbox{ for }(i,j)\in\widetilde{\mathcal P}_\ell$, with
\[
\baln
&\widetilde{\mathcal P}_1=\{(1,2),(1,3),(2,{-}3)\}, \quad 
\widetilde{\mathcal P}_2=  \{(2,0),(3,0)\}, \quad \widetilde{\mathcal P}_3=(2, 3).
\ealn
\]
	
\item Basis $B(V_{\lambda_2})$ for $V_{\lambda_{2}}$:
\beq\label{basis01}
	\begin{aligned}
		&	v_i{\otimes} v_j-q^{-1}v_j{\otimes} v_i,\  v_{-j}{\otimes} v_{-i}-q^{-1}v_{-i}{\otimes} v_{-j}, \\
		&\qquad\qquad  \text{ for } 1\leq i<j\leq 3; \mbox{~and~} (i, j)=(2, -3), \\
		&v_2\otimes v_3-q^{-3}v_3\otimes v_2
		+q^{-1}\left(v_1\otimes v_0-v_0\otimes v_1\right), \\
		&v_{-3}\otimes v_{-2}-q^{-3}v_{-2}\otimes v_{-3}
		+(q^{-2}+1)\left(v_0\otimes v_{-1}-v_{-1}\otimes v_0\right),\\	
		&v_1\otimes v_{-3}-qv_{-3}\otimes v_1
		+(q^{-1}+q)\left(v_0\otimes v_2-q^2v_2\otimes v_0\right), \\
		&v_3\otimes v_{-1}-qv_{-1}\otimes v_3+v_{-2}\otimes v_0-q^2v_0\otimes v_{-2}, \\
		&v_1\otimes v_{-2}-qv_{-2}\otimes v_1+(q^{-1}
		+q)\left(q^2v_3\otimes v_0-v_0\otimes v_3\right), 	\\
		&v_2\otimes v_{-1}-qv_{-1}\otimes v_2+q^2v_0\otimes v_{-3}-v_{-3}\otimes v_0,\\
		&v_2\otimes v_{-2}-q^2v_{-2}\otimes v_2-q^3v_3\otimes v_{-3}
		+q^{-1}v_{-3}\otimes v_3, \\
		&(q^3{-}q^{-1})v_0{\otimes} v_0-q^{-1}\left(v_1{\otimes} v_{-1}
		{-}q^2v_{-1}{\otimes} v_1\right)\\
		&{+}\left(v_2{\otimes} v_{-2}{-}v_{-2}{\otimes} v_2\right)
		{-}(q^{-1}{+}q)\left(v_3{\otimes} v_{-3}{-}v_{-3}{\otimes} v_3\right).
	\end{aligned}
\eeq

\item Basis $B(V_{\lambda_1})$ for $V_{\lambda_{1}}$:
write $B(V_{\lambda_1}) =\{b_a\mid a=0, \pm 1, \pm 2, \pm 3\}$, 
and impose the additional condition  that 
$v_a\mapsto b_a$, for all $a$, extends to a $\UG$-module isomorphism $V\simeq V_{\lambda_1}\subset V\ot V$.  Then
\beq\label{basis10}
\begin{aligned}
b_1=			&v_{1}\otimes v_{0}-q^{-6}v_{0}\otimes v_{1}
					-(q^{-3}+q^{-1})\left(v_{2}\otimes v_3-q^{-3}v_3\otimes v_2\right)\\
b_2=			&q^{-4}\left(v_0\otimes v_2-q^{2}v_{2}\otimes v_{0}\right)
					-q^{-1}v_{1}\otimes v_{-3}+q^{-6}v_{-3}\otimes v_{1},\\
b_3=			&	q^{-1}v_{1}\otimes v_{-2}-q^{-6}v_{-2}\otimes v_{1}
					-q^{-4}\left(q^{2}v_{3}\otimes v_{0}-v_0\otimes v_3\right),\\
b_0=			&q^{-6}v_{-1}{\otimes} v_{1}{-}q^{-2}v_{1}{\otimes} v_{-1}
					{-}(q^{-2}{-}q^{-4})v_{0}{\otimes}  v_{0}\\
	          &{+}q^{-1}v_{2}{\otimes} v_{-2}-q^{-7}v_{-2}{\otimes} v_{2}
					{+}q^{-4}\left(v_{3}{\otimes} v_{-3}{-}v_{-3}{\otimes}v_3\right),\\
b_{-3}=	&q^{-2}(1+q^2)	\left(v_{2}\otimes v_{-1}-q^{-5}v_{-1}\otimes v_{2}\right)\\
					&-q^{-4}\left(q^{2}v_{0}\otimes v_{-3}-v_{-3}\otimes v_{0}\right),\\
b_{-2}=	&q^{-4}\left(v_{-2}\otimes v_0-q^{2}v_{0}\otimes v_{-2}\right)\\
					&-q^{-2}(1+q^2)\left(v_{3}\otimes v_{-1}-q^{-5}v_{-1}\otimes v_{3}\right),\\
b_{-1}=	&v_{0}\otimes v_{-1}-q^{-6}v_{-1}\otimes v_{0}
					-q^{-2}\left(v_{-3}\otimes v_{-2} -q^{-3}v_{-2}\otimes v_{-3}\right).
\end{aligned} 
\eeq
\end{itemize}	

%
%
%

%
An important structural property of $\UG$ is the existence of a universal $R$-matrix $\R$,
 belonging to a completion of $\UG\otimes \UG$. We write 
\[\R=\sum_t\alpha_t\ot\beta_t.\]  
It has a well defined action on the tensor product of any two locally finite $\U$-modules of type-${\bf 1}$. 
For any such modules $V_1$ and $V_2$, the $R$-matrix 
$
\check{R}_{V_1,V_2}:V_1\otimes V_2\rightarrow V_2\otimes V_1
$
is an isomorphism of $\U$-modules defined as follows. 
\beq
\check{R}_{V_1,V_2}(v\ot v') = \sum_t\beta_t\cdot  v' \ot \alpha_t\cdot  v, \quad \forall v\in V_1, v'\in V_2.  
\eeq 
Given any locally finite $\U$-module $V_1$, $V_2$ and $V_3$, we have the following 
celebrated Yang-Baxter relation 
between $\U$-module isomorphisms $V_1\ot V_2\ot V_3 \lra V_3\ot V_2 \ot V_1$. 
\beq
&&(\check{R}_{V_2,V_3}\ot \id_{V_1}) (\id_{V_2}\ot \check{R}_{V_1,V_3}) (\check{R}_{V_1,V_2}\ot \id_{V_3})\\ 
&&= (\id_{V_3}\ot \check{R}_{V_1,V_2}) (\check{R}_{V_1,V_3}\ot \id_{V_2}) (\id_{V_1}\ot \check{R}_{V_2,V_3}). \nonumber
\eeq

If $V_2= V_3=V_1$, the Yang-Baxter relation reduces to the following braid relation
\[
(\check{R}_{V_1,V_1}\ot \id_{V_1}) (\id_{V_1}\ot \check{R}_{V_1,V_1}) (\check{R}_{V_1,V_1}\ot \id_{V_1}) 
= (\id_{V_1}\ot \check{R}_{V_1,V_1}) (\check{R}_{V_1,V_1}\ot \id_{V_1}) (\id_{V_1}\ot \check{R}_{V_1,V_1}).
\]

Consider 
$\check{R}:=\check{R}_{V, V}$ $\in {\rm End}_{\UG} (V\ot V)$, which has a simple description, as $V$ is strongly multiplicity free (in the sense of \cite{lz:sm}). 
Denote by $P[\lambda]$ the idempotent projecting $V\otimes V$ onto the simple submodule $V_\lambda$ 
for $\lambda=2\lambda_{1}, 0$, $\lambda_{2}$, or $ \lambda_{1}$.
Then $\check{R}$ has the following spectral decomposition \cite{lz:sm}.
\begin{equation}\label{Rmatrix1}
	\check{R}=q^2P[2\lambda_{1}]+q^{-12}P[0]-P[\lambda_{2}]-q^{-6}P[\lambda_{1}].
\end{equation}
It follows that 
\begin{equation}\label{R0}
	(\check{R} -q^2 \id_V)(\check{R}-q^{-12} \id_V)(\check{R}+ \id_V)(\check{R}+q^{-6} \id_V)=0.
\end{equation}

\subsection{Definition of braided symmetric algebras}

The notion of module algebras over a Hopf algebra \cite[\S4.1]{m:ha} will be used extensively.
Let $H$ be a Hopf algebra. 
	An associative unital algebra $(A,\mu)$ with multiplication $\mu$
	is a $H$-module algebra if $A$ is a $H$-module and $\mu$ is a $H$-module morphism, 
	\[x\cdot \mu(a\otimes b)=\sum_{(x)}\mu\big(x_{(1)}\cdot a\otimes x_{(2)}\cdot b\big), 
	\quad a, b\in A,  x\in H, 
	\]
	where Sweedler's notation $\Delta(x)=\sum_{(x)}x_{(1)}\otimes x_{(2)}$ for the comultiplication is used.
A module algebra $A$ for a Hopf algebra $H$ is called \textit{locally finite} if $H\cdot f$ 
is finite dimensional for any $f\in A$. 

We shall use the term $H$-algebras as a synonym for $H$-module algebras.

Assume that $H$ is a quasi-triangular Hopf algebra, i.e., it is a Hopf algebra that  admits a universal $R$-matrix $R$. 
If $(A,\mu_A)$ and $(B,\mu_B)$ are locally finite $H$-algebras, it is well know that 
the tensor product $A\otimes B$ with the multiplication 
\beq\label{eq:mu-AB}
\mu_{A,B}:=(\mu_A\otimes \mu_B)\circ(\id_A\otimes \check{R}_{A, B}\otimes\id_B)
\eeq  
is a $H$-module algebra (see, e.g., \cite{b:dcb, lzz:ft, m:ha}). 
This is usually referred to as the braided tensor product of $(A,\mu_A)$ and $(B,\mu_B)$.

Now we return to the standard module $V$ for $\UG$. 
Denote by $T(V)$ the tensor algebra over $V$. 
Then $T(V)=\sum_{r=0}^\infty V^{\ot r}$ is a $\UG$-module algebra in the obvious way.  It is $\ZZ_+$-graded with $V^{\ot r}$ being the degree $r$ homogeneous submodule.  

Define the \textit{quantum skew symmetric tensor} $\land_q^2V\subset V\otimes V$ to be the sum of the eigenspaces of $\check{R}$ corresponding to the eigenvalues $-1$ and $-q^{-6}$. 
Then $\land_q^2V \simeq V_{\lambda_2}\oplus V_{\lambda_1}$. 
Let $\mathfrak{I}_q(V)$ be the two-sided ideal of $T(V)$ generated by 
$\land_q^2V\simeq V_{\lambda_2}\oplus V_{\lambda_1}$. The quantum symmetric algebra over $V$ is defined by
\beq \label{eq:SqV}
S_q(V):=T(V)/\mathfrak{I}_q(V).
\eeq
We denote the canonical surjection by
\beq
\tau: T(V)\lra S_q(V).
\eeq

\begin{remark}\label{rmk:non-flat}
It is known that $S_q(V)$ is not flat in the sense of \cite{bz:bs}.  This follows from a 
result of \cite{z:pa} and the fact that $\wedge_a^2 V$ is not a simple $\UG$-module. 
\end{remark}

Since $\land_q^2V$ is a $\UG$-submodule of $T(V)$, which is homogeneous of degree $2$, 
the two-sided ideal $\mathfrak{I}_q(V)$ is a $\mathbb{Z}_+$-graded submodule of $T(V)$.
Hence the quotient algebra $S_q(V)$ is a $\mathbb{Z}_+$-graded $\UG$-module algebra. 
Let $S_q(V)_n$ be the homogeneous component of degree $n$, thus 
$
S_q(V)=\sum_{n\ge 0} S_q(V)_n.
$
For each $n$, the restriction of $\tau$ to the degree $n$ homogeneous subspace leads to the surjective map 
\beq\label{eq:tau-restrict}
\tau_n:  V^{\ot n}\lra S_q(V)_n.
\eeq

By iterating the construction of braided tensor product for the $\UG$-algebra $S_q(V)$, 
we obtain the $\UG$-algebra 
\beq\label{eq:Am}
\cA_{m}(V):=(S_q(V)^{\otimes m}, \mu_{\cA_m}), \quad \mu_{\cA_m}:=\mu_{S_q(V), \cA_{m-1}(V)}. 
\eeq
It follows from properties of the universal $R$-matrix that  $ \mu_{\cA_m}=\mu_{\cA_{r}(V), \cA_{m-r}(V)}$ for any $r=1, 2, \dots, m-1$.  
We refer to $\cA_{m}(V)$ as a braided symmetric algebra over $V$.

Since $S_q(V)$ is a $\ZZ_+$-graded $\UG$-algebra, $\cA_{m}(V)$ has a natural $\ZZ_+^m$-grading with  the homogeneous component $\cA_{m}(V)_{\textbf{d}}$ of degree $\textbf{d}=(d_1,\cdots,d_m)$  given by
\beq
\cA_{m}(V)_{\textbf{d}}=S_q(V)_{d_1}\otimes\cdots\otimes S_q(V)_{d_m}.
\eeq

\section{Structure of braided symmetric algebras}\label{sect:algebras}
We now develop the structure of the braided symmetric algebras $\cA_m(V)$ over $V$. 
\subsection{Decomposition of the quantum symmetric algebra}\label{sect:Sq}

 Let $X_a =\tau(v_a)$ for the basis elements $v_a$ of $V$.
 Using the bases $B(V_{\lambda_1})$ and $B(V_{\lambda_2})$  constructed in Section \ref{sect:st-mod} for the submodules $V_{\lambda_1}$ and $ V_{\lambda_2}$ of $V\otimes V$, we immediately obtain a complete set of defining relations of $S_q(V)$.    
We have the following result.

\begin{lemma}\label{bsa}
	The quantum symmetric algebra $S_q(V)$ of the standard ${\UG}$-module $V$ is generated by $\{X_i|i=0, \pm 1,\pm 2, \pm 3\}$ subject to the following relations.
	\begin{align}
		&X_j X_i=a_{(i, j)} X_iX_j,\quad X_{-i}X_{-j}=a_{(i,  j)}X_{-j}X_{-i}, 
		\quad \text{for }  (i, j)\in\mathcal{P}, \label{r1}\\
		&\begin{array}{l}
			X_3X_2{=}q^3X_2X_3{+}(q^2{-}q^4)X_1X_0,\quad
			X_{-2}X_{-3}{=}q^3X_{-3}  X_{-2}{+}(q{-}q^5)X_{0} X_{-1};\\
			X_0X_2{=}q^2X_2X_0{+}(q^3{-}q)X_1X_{-3},\quad
			X_{-2}X_0{=}q^2X_0X_{-2}{+}(q^4{-}1)X_3X_{-1 };\label{r2}\\
			X_0X_3{=}q^2X_3X_0{+}(q{-}q^3)X_1X_{-2},\quad
			X_{-3}X_0{=}q^2X_0X_{-3}{+}(1{-}q^4)X_2X_{-1};\end{array}\\
		&\begin{array}{l}
			X_{-1}X_1=q^2 X_1X_{-1}+(q^{-1}{-}q^{3} )X_2X_{-2}+(q^{-4}{-}1)X_3X_{-3}\\
			\phantom{XXXX} +(q^{-4}{-   }q^{-2})(X_0X_{0}+X_1X_{-1}),\\
			X_{-2}X_2=X_2X_{-2}
			+(q^{-3}-q)X_3X_{-3}+(q^{-3}-q^{-1})(X_0X_{0}+X_1X_{-1}),\\
			X_{-3}X_3=X_3X_{-3}
			+(q^{3}-q)X_2X_{-2}+(1-q^{2})(X_0X_{0}+X_1X_{-1}),  \label{r3}
		\end{array}
	\end{align}
where $\mathcal{P}=\mathcal{P}_1\cup \mathcal{P}_2\cup \mathcal{P}_3$ with
$
\mathcal{P}_1=\{(1,2), (1,3), (2,-3)\}, 
\mathcal{P}_2=\{(1, 0)\},  \mathcal{P}_3=\{(1, -3), (1,-2)\},
$ 
and 
$a_{(i, j)} = q^\ell$ for $(i, j)\in \mathcal{P}_\ell$.

\end{lemma}
\begin{proof}
The submodule $\wedge^2_q V$ of $V\otimes V$ has a basis 
 $B_- =B(V_{\lambda_1})\cup B(V_{\lambda_2})$. 
The number of relations  given in the lemma is equal to $\#B_-=21$.  
We have performed a basis change
on  $B_-$ to obtain a new basis 
	for $\wedge^2_q V$. The relations 
	for $S_q(V)$ listed in the lemma arise from the elements of the new basis. 
\end{proof}

Denote by $\Phi\in S_q(V)$ the image of the element $c_0$ given by \eqref{basis0}. Then 
\beq\label{eq:quadr-inv}
 &&\Phi:=q^4X_1X_{-1}+q^3X_2X_{-2}+X_3X_{-3}+q^{-6}X_{-1}X_1  \\
&&\phantom{XX}+q^{-5}X_{-2}X_2+q^{-2}X_{-3}X_3+X_0X_{0},   \nonumber
\eeq
which spans the subspace $(S_q(V)_2)^{\UG}$ 
of invariants.

We have the following result. It in particular shows explicitly that $S_q(V)$ is not flat (cf., Remark \ref{rmk:non-flat}). 

\begin{lemma} \label{lem:Phi-nil}
	The element $\Phi$ satisfies 
	$
	\Phi X_a = X_a \Phi =0
	$
	for all $a=0, \pm1, \pm 2,\pm3$.  This in particular implies $\Phi^2=0$.
\end{lemma}
\begin{proof}
The proof requires delicate computations using Lemma \ref{bsa}. As this is a crucial result, 
we present the proof in deatil, even though it is elementary for the most part. 

Let us  first eliminate $X_{-a} X_a$, for $a=1, 2, 3$, from the definition of the invariant $\Phi$ (equation \eqref{eq:quadr-inv}) by using the three relations \eqref{r3}. We obtain 
\beq
\Phi&=&\left(q^{-10}-q^{-6}+q^{-2}\right)\Phi^\prime, \quad \text{with}\\
\Phi^\prime:&=&(1+q^{6})X_1X_{-1}+(q^3+q^{5})X_2X_{-2}+(1+q^2)X_3X_{-3}+X_0X_{0}.
\eeq
Now $X_{-3}\Phi^\prime$ is equal to 
	\begin{align*}
		(1+q^{6})X_{-3}X_1X_{-1}+(q^3+q^{5})X_{-3}X_2X_{-2}
		+(1+q^2)X_{-3}X_3X_{-3}+X_{-3}X_0X_{0}.
	\end{align*}
We use Lemma \ref{bsa} to move $X_{-3}$ to the right in all the terms, obtaining  
\begin{align*}
&(1+q^{6})q^{2}X_1X_{-1}X_{-3}+(q^3+q^{5})\left(q^{-2}X_2X_{-2}X_{-3}-(q^{-1}-q^3)X_2X_0X_{-1}\right)\\
		&+(1+q^2)\left(X_3X_{-3}X_{-3}+(1-q^{2})\left(X_1X_{-1}X_{-3} -qX_2X_{-2}X_{-3}+X_0X_{0}X_{-3}\right)\right)\\
		&+q^{4}X_0X_{0}X_{-3}+(q^4+q^2)(1-q^4)X_{2}X_0X_{-1}\\
		&+(1-q^4)(q^5-q^3)q^{-1}X_1X_{-1}X_{-3}\\
		&=\left((1+q^{6})X_1X_{-1}+(q^3+q^{5})X_2X_{-2}+(1+q^2)X_3X_{-3}+X_0X_{0}	\right)X_{-3}.
	\end{align*}
The right hand side of the above equation is equal to $\Phi^\prime X_{-3}$. Hence  
	\begin{align}\label{r2-1}
	X_{-3}\Phi^\prime=	\Phi^\prime X_{-3}, \  \text{  i.e. } \ 	X_{-3}\Phi=\Phi X_{-3}.
	\end{align}

Now let us consider the element $(q^{-4}-q^{-2})\Phi^\prime$. We have 
\[
\baln
(q^{-4}-q^{-2})\Phi^\prime= & 
(q^{-4}-q^{-2}+q^2-q^{4})X_1X_{-1}+(q^{-1}-q^{3})X_2X_{-2}\\
&+(q^{-4}-1)X_3X_{-3} +(q^{-4}-q^{-2})X_0X_{0}.
\ealn
\]
Thus  the first relation in \eqref{r3} is equivalent to 
$X_{-1}X_1=(q^{-4}-q^{-2})\Phi^\prime+q^4X_1X_{-1}.$
Set 
\beq
\vartheta:=X_{-1}X_1-q^4X_1X_{-1}.
\eeq
Then 
$ \Phi^\prime = \frac{\vartheta}{q^{-4}-q^{-2}}$ and $ 
\Phi = \frac{q^{-8}-q^{-4}+1}{q^{-2}-1}\vartheta.
$

Now we determine the commutation relations of $\vartheta$ with the $X_i$'s.
Using Lemma \ref{bsa}, we can prove the following relations,  
\beq\label{quasicom}
	\begin{aligned}
		X_{i}\vartheta 
		&=q^{-2}\vartheta X_{i},\mbox{~if~}i= 2, 3,\\ 
		X_{i}\vartheta 
		&=q^{2}\vartheta X_{i},\mbox{~if~}i=- 2, -3,\\	
		X_{0}\vartheta 
		&=\vartheta X_{0}, \\
		X_1\vartheta &=(q^{-4}+q^2-q^{-2})^{-1}\vartheta X_1, \\
		X_{-1}\vartheta &=(q^{-4}+q^2-q^{-2})\vartheta X_{-1}, 
	\end{aligned}
	\eeq
and hence the same relations with $\vartheta$ replaced by $\Phi$. 

The first three relations are easier to prove. Let us 
consider the the fourth relation as an example. Using the first relation of \eqref{r3}, 
we obtain 
	\begin{align*}
		&q^4(q^{-4}+q^2-q^{-2})X_1X_1X_{-1}\\
		=&q^4X_1\left[X_{-1}X_1-\big((q^{-1}{-}q^{3} )X_2X_{-2}+(q^{-4}{-}1)X_3X_{-3}+(q^{-4}{-   }q^{-2})X_0X_{0}\big)\right]\\
		=&q^4X_1X_{-1}X_1-q^4X_1\big((q^{-1}{-}q^{3} )X_2X_{-2}+(q^{-4}{-}1)X_3X_{-3}+(q^{-4}{-   }q^{-2})X_0X_{0}\big).
	\end{align*}
By using \eqref{r1} to the terms after the first minus sign, we can cast the right hand side into  
\begin{align*}
q^4X_1X_{-1}X_1-\big((q^{-1}{-}q^{3} )X_2X_{-2}+(q^{-4}{-}1)X_3X_{-3}+(q^{-4}{-   }q^{-2})X_0X_{0}\big)X_1, 
	\end{align*}
which can be further simplied to 
\begin{align*}
		q^4X_1X_{-1}X_1+\left[(q^{-4}+q^2-q^{-2})X_1X_{-1}-X_{-1}X_1\right]X_1
		\end{align*}
by using the first relation of \eqref{r3}. Hence 
\begin{align*}
	q^4(q^{-4}+q^2-q^{-2})X_1X_1X_{-1}
		&=q^4X_1X_{-1}X_1+(q^{-4}+q^2-q^{-2})X_1X_{-1}X_1-X_{-1}X_1X_1.
	\end{align*}
This immediately leads to 
$
X_1\vartheta =(q^{-4}+q^2-q^{-2})^{-1}\vartheta, 
$
which is what we seek to prove.

Compare the relation \eqref{r2-1} with the $i=-3$ case of \eqref{quasicom} with $\vartheta = \frac{q^{-2}-1}{q^{-8}-q^{-4}+1} \Phi$, we obtain $\Phi X_{-3}= X_{-3}\Phi=0$.
Since $\Phi$ is a  $\UG$-invariant, applying the $\UG$-action to the relation we obtain $\Phi X_i = 0 = X_i \Phi$ for all $i$. It immediately follows that $\Phi^2=0$. 
\end{proof}

The element $X_1^n\in S_q(V)_n$ generates a simple $\UG$-submodule, denoted by ${\mathcal H}_n$, which is isomorphic to $V_{n\lambda_1}$.  We have the following decomposition of $S_q(V)$ into simple $\UG$-submodules. 
\begin{theorem}\label{homogeneous}
The homogeneous components of $S_q(V)$ are given by
\[
\baln
S_q(V)_2={\mathcal H}_2+\CC(q)\Phi, \quad  S_q(V)_n= {\mathcal H}_n \ 	\text{ for\   $n\ne 2$}, 
\ealn
\]
i.e.,  $S_q(V)=\CC(q)\Phi+ \sum_{n\ge 0} \CH_n$.  
\end{theorem}

The proof of the theorem will be given in Section \ref{sect:Sq-struct}, which makes use of 
an alternative construction of $S_q(V)$.

The following corollary immediately follows from the theorem and Lemma \ref{lem:Phi-nil}. 
\begin{corollary}\label{cor:homo}
	The subalgebra of $\UG$-invariants in $S_q(V)$ is given by 
	\[
	S_q(V)^{\UG}=\CC +\CC \Phi. 
	\]
\end{corollary}

\begin{remark}
Theorem \ref{homogeneous} is the only case where the decomposition of a non-flat quantum symmetric algebra 
into simple ${\rm U}_q(\g)$-modules is understood. 
\end{remark}

%
%
\subsection{Proof of Theorem \ref{homogeneous}}
We need to better understand the quantum symmetric algebra $S_q(V)$ in order to prove Theorem \ref{homogeneous}. 

\subsubsection{An alternative construction of $S_q(V)$}
The general framework for the material in this section is set up in \cite{GZ}. 
Denote by $\UD$ the finite dual of $\UG$, which is the subalgebra of ${\rm U}_q^*$ spanned by the matrix elements 
of the finite dimensional representations of $\UG$. It has the structure of a Hopf algebra.  
There are two natural left actions of $\UG$ on $\UD$, 
\[
L, R: \UG\ot \UD \lra \UD, 
\]
which are respectively defined,  for any $x\in\UG$ and $f\in \UD$, by 
\beq
L_x(f)(y)= f(S(x)y), \quad R_x(f)(y) = f(y x), \quad \forall y\in \UG.
\eeq
One can readily verify that these are indeed left actions of $\UG$, 
and they commute with each other.  Furthermore, for any $f, g\in \UD$ and $x\in\UG$, 
\beq\label{eq:LR-prod}
L_x(fg)=\sum L_{x_{(2)}}(f)  L_{x_{(1)}}(g), \quad R_x(fg)=\sum R_{x_{(1)}}(f)  R_{x_{(2)}}(g). 
\eeq
Thus $\UD$ is a module algebra for $R_\UG$ with respect to the co-multiplication $\Delta$, and also for $L_\UG$ with respect to $\Delta'$. 

Let $\CC_q[G_2]$ be the Hopf subalgebra of $\UD$ spanned by the matrix elements of 
the finite dimensional type-${\bf 1}$ representations of $\UG$. 
The restrictions of  $L_{\UG}$ and $R_{\UG}$ leads to 
left actions of $\UG$ on $\CC_q[G_2]$.  For any $n\in\ZZ_+$, we define the subspace 
\[
\Gamma_n:= \{ f\in \CC_q[G_2] \mid L_{E_i}(f)=0, \ L_{K_i}(f) = q^{-n(\alpha_i, \lambda_1)} f, \ i=1, 2\},
\]
which is a $\UG$-module under the action of $R_{\UG}$ because of the commutativity between $L_{\UG}$ and $R_{\UG}$. 

\begin{lemma} 
Let $\Gamma=\sum_{n\in\ZZ_+}\Gamma_n$. Then $\Gamma\simeq \sum_{n\ge 0} V_{n\lambda_1}$ as $R_{\UG}$-module. 
Furthermore, $\Gamma$ is a $\UG$-module subalgebra of $\CC_q[G_2]$, which is 
\begin{enumerate}[i)]
\item  $\ZZ_+$-graded with $\Gamma_m \Gamma_n\subseteq \Gamma_{m+n}$ for all $m, n$; and 
\item  generated by $\Gamma_1$. 
\end{enumerate}
\end{lemma}

\begin{proof}
In the setting of \cite{GZ}, $\Gamma_n$ can be interpreted as the space of global sections of 
a quantum line bundle on the quantum flag manifold.  
It follows from the quantum Borel-Weil theorem \cite{APW, GZ} that 
$\Gamma_n\simeq V_{n\lambda_1}$ as $R_\UG$-module.  Hence $\Gamma\simeq \sum_{n\ge 0} V_{n\lambda_1}$.

It is evident that $f g\in\CC_q[G_2]$ for any $f\in\Gamma_m$ and $g\in \Gamma_n$.
By the first relation of \eqref{eq:LR-prod},  we have
\[
\baln
L_{K_i}(f g) &= L_{K_i}(f)  L_{K_i}(g) = q^{-(m+n)(\alpha_i, \lambda_1)} fg,\\
L_{E_i}(f g) &=  \sum L_{{E_i}_{(2)}}(f) L_{{E_i}_{(1)}}(g),  \quad i=1, 2, 
\ealn
\]
where Sweedler's notation $\Delta(E_i) = \sum  {E_i}_{(1)} \ot {E_i}_{(2)}$ for the  the co-multiplication of $E_i$ is used 
in the second equation. Now the opposite co-multiplication is given by 
$\Delta'(E_i) = \sum {E_i}_{(2)}\ot {E_i}_{(1)}$.  We have  $\Delta'(E_i) = E_i\ot K_i + 1\ot E_i$.   Thus one of the two components ${E_i}_{(1)}$ and ${E_i}_{(2)}$ is $E_i$, and hence either $L_{{E_i}_{(2)}}(f)=0$ or $L_{{E_i}_{(1)}}(g)=0$. 
This shows that 
\[
L_{E_i}(f g) =0,   \quad i=1, 2. 
\]
Thus
$fg\in\Gamma_{m+n}$, proving that $\Gamma$ is a $\ZZ_+$-graded algebra. 

Since each $\Gamma_n$ is a $\UG$-module under the $R_{\UG}$-action, 
it is evident from the second relation of \eqref{eq:LR-prod} that $\Gamma$ is a $\UG$-module algebra.

Finally, we prove that $\Gamma_1$ generates $\Gamma$. 

Denote by $\Gamma^{(\lambda)}$ the subspace of $\CC_q[G_2]$ spanned by the matrix elements of 
$V_\lambda$ for $\lambda\in P^+$. Then 
$\CC_q[G_2]=\sum_{\lambda\in P^+} \Gamma^{(\lambda)}$, which is known as the quantum Peter-Weyl theorem. 
Recall that $V$ generates all finite dimensional $\UG$-modules of type-${\bf 1}$ 
in the sense that any such module is a summand in the direct sum of some tensor powers of $V$. 
Thus $\Gamma^{(\lambda_1)}$ generates $\CC_q[G_2]$ as an algebra. 

Now $\Gamma_1=\{f\in \Gamma^{(\lambda_1)}\mid L_{E_1}(f)= L_{E_2}(f)= 0\}$. Denote by $\left(\Gamma_1\right)^n$ the subspace of $\Gamma$ spanned by the elements of the form $f_1 f_2 \dots f_n$ with $f_i \in \Gamma_1$ for all $i$. 
Then $\left(\Gamma_1\right)^n\subseteq \Gamma_n$ as $\UG$-module. 
We will show presently that $\left(\Gamma_1\right)^n\ne 0$, hence $\left(\Gamma_1\right)^n=\Gamma_n$ since $\Gamma_n$ is a simple $\UG$-module.

To prove $\left(\Gamma_1\right)^n\ne 0$, we consider the $\CC_q[G_2]$-co-actions on $V^{\ot n}$ 
denoted by
\[
\delta_n: V^{\ot n}\lra V^{\ot n}\ot \CC_q[G_2]. 
\]
Recall that we used $v_1$ to denote the highest weight vector of $V$. Then $\delta_1(v_1)=v_1\ot t_{11} + \text{``lower-weight terms"}$, where $t_{11}\in \Gamma^{(\lambda_1)}$, and the $\text{``lower-weight terms"}$ is a linear combination of tensor products of vectors with weights lower than that of $v_1$ and elements of $\Gamma^{(\lambda_1)}$. 
Also, 
$\delta_n(v_1^{\ot n})=  v_1^{\ot n}\ot t_{11}^n + \text{``lower weight terms"}$, for any $n$. 
As $v_1^{\ot n}$ is the highest weight vector of the simple $\UG$-submodule $V_{n\lambda_1}$ in $V^{\ot r}$, 
we must have $t_{11}^n\ne 0$. Hence $\left(\Gamma_1\right)^n\ne 0$. 

This completes the proof of the theorem.
\end{proof}

We define an algebra related to $S_q(V)$ as follows. Denote by $\widetilde{\mathfrak{I}}_q(V)$ the
the two-sided ideal of $T(V)$ generated by $\wedge_q^2 V + \CC(q) c_0$, and let
\beq
HS_q(V)=T(V)/\widetilde{\mathfrak{I}}_q(V).
\eeq
Then 
$HS_q(V)= S_q(V)/\langle \Phi\rangle$, where $\langle \Phi\rangle$ is the two-sided ideal of $S_q(V)$ generated by $\Phi$.

We have the following result.
\begin{theorem}\label{thm:CH-Sq}
There is the following isomorphism of $\ZZ_+$-graded $\UG$-module  algebras.
\[
\Gamma\simeq  HS_q(V).
\]
\end{theorem}

The proof will be given in Section \ref{sect:proof-CH-Sq}. 

\subsubsection{Comments on $S(\CC^7)$ as $G_2$-module algebra}\label{sect:classical}

In order to prove Theorem \ref{thm:CH-Sq}, we need to gather some information on the symmetric algebra $S(\CC^7)$ 
over the simple $G_2$-module $\CC^7$, and also a related algebra of it.  
What discussed below is not new, which is included here for the convenience of the reader.

We have $\CC^7\ot \CC^7= S^2(\CC^7)\oplus \wedge^2 \CC^7$, 
and $S^2(\CC^7)$ contains the $1$-dimensional $G_2$-submodule $(\CC^7\ot \CC^7)^{G_2}$.  
Then $S(\CC^7)=T(\CC^7)/\mathfrak{I}(\CC^7)$, where 
$\mathfrak{I}(\CC^7)$ is the two-sided ideal of $T(\CC^7)$ generated by $\wedge^2 \CC^7$. 
Denote by  $\widetilde{\mathfrak{I}}(\CC^7)$ the two-sided ideal generated 
by $\wedge^2 \CC^7\oplus (\CC^7\ot \CC^7)^{G_2}$, and let
\[
HS(\CC^7)= T(\CC^7)/\widetilde{\mathfrak{I}}(\CC^7). 
\]
Clearly $\dim HS_q(V)_i = \dim HS(\CC^7)_i$ for $i=0, 1, 2$.  It is a general fact that
\beq\label{eq:quadratic}
\dim HS_q(V)_n \le \dim HS(\CC^7)_n, \quad \forall n\ge 3. 
\eeq

Let
$\phi$ be an element spanning $(\CC^7\ot \CC^7)^{G_2}\subset S^2(\CC^7)$. Then $HS(\CC^7)=S(\CC^7)/\langle \phi\rangle$.

Denote by $V^{(0)}_\lambda$ the simple $G_2$-module with the highest weight $\lambda\in P^+$. 
It is evident that the homogeneous component $S^n(\CC^7)$ of  $S(\CC^7)$  contains a $G_2$-submodule 
${\mathcal H}^{(0)}_n\simeq V^{(0)}_{n\lambda_1}$ for any $n$. 
In particular, $\CC^7=V^{(0)}_{\lambda_1}$ and $S^2(\CC^7)= \CH_2^{(0)}\oplus \CC\phi$. 
Recall that $\dim S^k(\CC^7) =\begin{pmatrix} k+6\\ k \end{pmatrix}$ for all $k$, and by
Weyl's dimension formula,  
\[
\baln
	\dim  V^{(0)}_{n\lambda_1}&= 
\frac{(n+4)! (2n+5)}{5! n!}=\begin{pmatrix} n+6\\ n \end{pmatrix} - \begin{pmatrix} n+4\\ n-2 \end{pmatrix}\\
&= \dim S^n(\CC^7) - \dim S^{n-2}(\CC^7), \quad \forall n\ge 2. 
\ealn
\]
An easy dimension counting argument leads to 
\[
S^n (\CC^7) = {\mathcal H}^{(0)}_n+ \phi S^{n-2}(\CC^7), \quad \forall n\ge 2,  
\] 
and hence 
\beq\label{eq:harmonic}
S^n(\CC^7)=\sum_{k=0}^{[n/2]} \phi^k {\mathcal H}^{(0)}_{n-2k}, \quad HS(\CC^7)_n= {\mathcal H}^{(0)}_n. 
\eeq
Thus ${\mathcal H}^{(0)}_{n}$ is the ``harmonic space" of $S^n(\CC^7)$.

Now we are ready to prove Theorem \ref{thm:CH-Sq}.

\subsubsection{Proof of Theorem \ref{thm:CH-Sq}}\label{sect:proof-CH-Sq}

\begin{proof}

Consider the $\CC_q[G_2]$-co-action on the elements of the basis $\{v_a\mid a=0, \pm 1, \pm 2, \pm 3\}$ 
of $V$ given in Section \ref{sect:st-mod}. There exist  $t_{a' a}\in \CC_q[G_2]$ such that 
\[
\delta_1(v_a)= \sum_{a'} v_{a'}\ot  t_{a' a}, \quad \Gamma^{(\lambda_1)}=\sum_{a' a} \CC(q)t_{a' a}. 
\]
From this description of $\Gamma^{(\lambda_1)}$, we immediately obtain 
\beq
\Gamma_1=\sum_a \CC(q) t_{1 a}. 
\eeq 
We claim that the algebra isomorphism of Theorem \ref{thm:CH-Sq} is given by 
\beq\label{eq:iso-T-Sq}
 X_a + \langle \Phi\rangle \mapsto t_{1 a}, \quad a=0, \pm 1, \pm 2, \pm 3.
\eeq
To prove this, we consider  the $\CC_q[G_2]$-co-action on $\wedge_q^2 V$. 
For any $\sum c_{a b} v_a\ot v_b\in \wedge_q^2 V$, 
\[
\delta_2\left(\sum_{a, b} c_{a b} v_a\ot v_b\right)= \sum_{a', b'}  v_{a'}\ot v_{b'} \ot \sum_{a, b}c_{a b} t_{a' a}t_{b' b}. 
\]
Since $v_1\ot v_1\not\in \wedge_q^2 V$, we have 
\beq
 \sum_{a, b}c_{a b} t_{1 a}t_{1 b} =0. 
\eeq 
Thus the elements $t_{1 a}$ satisfy the defining relations of $S_q(V)$ obeyed by $X_a$. 

Let us write the basis vector $c_0$ of $(V\ot V)^{\UG}$ as 
\beq\label{eq:c0}
c_0=\sum_{a, b} \varphi_{a b} v_a\ot v_b, 
\eeq 
where the scalars $\varphi_{a b}$ can be read off \eqref{basis0}.
Then under
the map \eqref{eq:iso-T-Sq},  
\[
0= \sum_{a, b}\varphi_{a b} X_a X_b + \langle \Phi\rangle\mapsto 
\varphi:=\sum_{a, b}\varphi_{a b} t_{1 a}t_{1 b}. 
\]
To see that  $\varphi=0$ in $\Gamma$, consider
\[
\delta_2(c_0)=\sum_{a', b'}  v_{a'}\ot v_{b'} \ot \sum_{a, b}\varphi_{a b} t_{a' a}t_{b' b} = c_0\ot 1.
\]
As $v_1\ot v_1$ is not a term in $c_0$, we have 
\beq
\sum_{a, b}\varphi_{a b} t_{1 a}t_{1 b} =0. 
\eeq
This proves that \eqref{eq:iso-T-Sq} is an algebra homomorphism. 

Now if we can further show that  
$\Gamma\simeq HS_q(V)$ as $\ZZ_+$-graded $\UG$-module,   
then \eqref{eq:iso-T-Sq} is an algebra isomorphism.

Recall that $S_q(V)_n\supset \CH_n\simeq V_{n\lambda_1}$. 
It follows \eqref{eq:quadratic} that
\[
\dim V_{n\lambda_1}\le \dim HS_q(V)_n \le\dim HS(\CC^7)_n=\dim V^{(0)}_{n\lambda_1}. 
\]
Since $\dim V_{n\lambda_1} = \dim V^{(0)}_{n\lambda_1}$, we conclude that $\dim HS_q(V)_n=\dim V_{n\lambda_1} = \dim \Gamma_n$. Hence 
$\Gamma\simeq HS_q(V)$ as $\ZZ_+$-graded $\UG$-module.  

This completes the proof of the theorem.     
\end{proof}

\subsubsection{Re-construction of $S_q(V)$}\label{sect:Sq-struct}
Theorem \ref{thm:CH-Sq} enables us to reconstruct $S_q(V)$  from $\Gamma$ as follows. 

Let $\widetilde{\Gamma}$ be the $\ZZ_+$-graded algebra generated by $\Gamma$ and the extra generator $\theta$ of degree $2$ and additional relations
$
\theta^2=0$ and $\theta \Gamma_n=  \Gamma_n \theta=0$ for all $n.
$
Then  
\[
\widetilde{\Gamma}=\Gamma+\CC(q)\theta.
\]
Since $X_a \Phi=\Phi X_a=0$ by Lemma \ref{lem:Phi-nil}, it immediately follows Theorem \ref{thm:CH-Sq} that 
\beq
\widetilde{\Gamma}\simeq S_q(V)
\eeq 
as $\ZZ_+$-graded $\UG$-module algebra. 

\begin{proof}[Proof of Theorem \ref{homogeneous}]
The isomorphism $\widetilde{\Gamma}\simeq S_q(V)$ of $\ZZ_+$-graded $\UG$-module algebras immediately implies Theorem \ref{homogeneous}. 
\end{proof}


%
%
\subsection{Presentations of braided symmetric algebras}\label{sect:Am}

We now develop the structure of the braided symmetric algebra $\cA_{m}(V)$ in more details.

The multiplication of $\cA_{m}(V)$ can be easily described. 
For any $A, B\in S_q(V)$, let $A_i=1\otimes \cdots\otimes 1\otimes \underbrace{A}_i\otimes 1\otimes\cdots\otimes 1$ and $B_j=1\otimes \cdots\otimes 1\otimes \underbrace{B}_j\otimes 1\otimes\cdots\otimes 1$ 
with $i<j$, then $B_j A_i= \sum_{t}\beta_t(A_i)\alpha_t(B_j)$. 

More specifically, denote $X_{i a}=\underbrace{1\otimes \cdots\otimes 1}_{i-1}\otimes X_a\otimes \underbrace{1\otimes\cdots\otimes 1}_{m-i}$ in $\cA_m(V)$,   for $i=1, 2, \dots, m$ 
and $a=0, \pm 1, \pm 2, \pm3$,
where $X_a$ are the generators of $S_q(V)$. 	
Then  
\begin{equation}\label{braid}
		X_{ja}X_{ib}=\sum_{t}\beta_t(X_{ia})\alpha_t(X_{jb}), \quad 1\le i < j \le m.
\end{equation}

The elements  $X_{ia}$ generate $\cA_{m}(V)$.  
To describe their relations, we introduce some notation.  
For any integers $r<s$, we let $[r, s]=\{r, r+1, \dots, s\}$. 
Define the map 
\[
\baln 
&[1, 7]\lra [-3, 3], \quad a\mapsto  \ol{a}, \\
&\text{$\ol{4}=0$,} \quad \text{$\ol{a}=a$, if $a=1, 2, 3$,}\\
&\text{$\ol{a}=a-8$,  if $a=5, 6, 7$.}
\ealn
\]
We also introduce the following subsets of $[1, 7]^2=[1, 7]\times[1,7]$.  
Let $\mathcal{I}=\left\{(1,2),(1,3),(2,5),(3,6),(5,7),(6,7)\right\}$, 
and 
$\mathcal{J}_r = \{(a, b)\in [1, 7]^{2} \mid a<b, \ a+b=r\} \backslash\mathcal{I}$ for $8\ne r\in [5, 11]$. 
One can easily check,  case by case,  that $\#\mathcal{J}_r=2$ for all
$r\in[5, 11]\backslash\{8\}$. Thus $\mathcal{J}_r=\{(a_1, r-a_1), (a_2, r-a_2)\}$ for some $a_1< a_2$ (depending on $r$).
\begin{proposition}\label{pre-Am}
Retain notation above. 
The braided symmetric algebra $\cA_{m}(V)$ is generated by $X_{ia}$ with $i\in[1,m]$ 
and $a\in[-3,3]$, subject to the following relations:
\begin{enumerate}
\item for fixed $i$,
			the elements $X_{ia}$, for all $a\in[-3, 3]$, obey the relations \eqref{r1}, 
			\eqref{r2} and \eqref{r3} with $X_a$ replaced by $X_{i a}$; 
\item for $j>i$ in $[1,m]$, and $a, b\in[1, 7]$:
\beq	
&\begin{aligned}
				X_{j\ol{a}} X_{i\ol{a}}&=q^2X_{i\ol{a}} X_{j\ol{a}},\, a\neq 4;\\		
				X_{j\ol{a}} X_{i\ol{b}}&=qX_{i\ol{b}}X_{j\ol{a}},  \,(a,b)\in\mathcal{I},\\
	X_{j\ol{b}}X_{i\ol{a}}&=qX_{i\ol{a}} X_{j\ol{b}}{+}(q^2-1) X_{i\ol{b}}X_{j\ol{a}},\,(a,b)\in\mathcal{I};\\
			\end{aligned}\\
%
&\begin{pmatrix}
				X_{j 1}X_{i, -1}\\
				X_{j 2}X_{i, -2}\\
				X_{j 3}X_{i, -3}\\
				X_{j 0}X_{i0}\\ 
				X_{j, -3}X_{i3}\\
				X_{j, -2}X_{i2}\\
				X_{j, -1}X_{i1}
\end{pmatrix}
=A_8
\begin{pmatrix}
X_{i, -1}X_{j1}\\
X_{i, -2}X_{j2}\\
X_{i, -3}X_{j3}\\	
X_{i0}X_{j0}\\
X_{i3}X_{j, -3}\\
X_{i2}X_{j, -2}\\
X_{i1}X_{j, -1}
\end{pmatrix}, 
\eeq
and 
for $r\in[5, 11]\backslash\{8\}$, write $\mathcal{J}_r=\{(a_1, r-a_1), (a_2, r-a_2)\}$ 
with $a_1< a_2$, then 
\beq\left(\begin{array}{c}	
				X_{j,\ol{a_1}}X_{i, \ol{r-a_1}}\\
				X_{j, \ol{a_2}}X_{i, \ol{r-a_2}}\\	
				X_{j, \ol{r-a_2}}X_{i,\ol{a_2}}\\
				X_{j, \ol{r-a_1}}X_{i,  \ol{a_1}}
			\end{array}
			\right)=A_r\left(\begin{array}{c}	
				X_{i,\ol{r-a_1}}X_{j, \ol{a_1}}\\
				X_{i, \ol{r-a_2}}X_{j, \ol{a_2}}\\
				X_{i, \ol{a_2}}X_{j, \ol{r-a_2}}\\
				X_{i, \ol{a_1}}X_{j, \ol{r-a_1}}
			\end{array}\right); 
\eeq
\end{enumerate}
where $A_8$ is a $7\times 7$,  and $A_r$, for each $r\in[5, 11]\backslash\{8\}$,  is a $4\times 4$, non-singular lower-triangular matrix given below.
\end{proposition}

\noindent{\bf The matrices $A_\ell$ for $\ell\in[5, 11]$}.  

$\hspace{3.5cm}
			A_{5}{=}\left(\begin{array}{cccc}	
				1&0&0&0\\
				q\mathbf{d}&q^{-1}&0&0\\	
				{-}q^{-2}\mathbf{d}&\mathbf{b}\mathbf{d}&q^{-1}&0\\
				\mathbf{c}{-}q^{-2}\mathbf{d}&{-}(q^{-4}{+}q^{-2})\mathbf{d}&(q{+}q^{{-}1})\mathbf{d}&1
			\end{array}
			\right),
$

$\hspace{3.5cm}
			A_{6}{=}\left(\begin{array}{cccc}	
				q^{-1}&0&0&0\\
				-\mathbf{d}&1&0&0\\	
				q^{-2}\mathbf{d}&(1{+}q^{2})\mathbf{d}&1&0\\
				\mathbf{c}&(1{+}q^{-2})\mathbf{d}&-(1{+}q^{2})\mathbf{d}&q^{-1}
			\end{array}
			\right),
$

$\hspace{3.5cm}
A_{9}{=}\left(\begin{array}{cccc}	
				q^{-1}&0&0&0\\
				(q{+}q^{-1})\mathbf{d}&1&0&0\\	
				-q^{-1}\mathbf{d}&-(1{+}q^{2})\mathbf{d}&1&0\\
				\mathbf{c}&-q^{-1}\mathbf{d}&q\mathbf{d}&q^{-1}
			\end{array}
			\right), 
$

$\hspace{3.5cm}
A_{11}{=}\left(\begin{array}{cccc}	
				1&0&0&0\\
				(1+q^{2})\mathbf{d}&q^{-1}&0&0\\	
				-q^{-1}\mathbf{d}&\mathbf{b}\mathbf{d}&q^{-1}&0\\
				\mathbf{c}{-}q^{-2}\mathbf{d}&-q^{-3}\mathbf{d}&\mathbf{d}&1
			\end{array}
			\right),
$

\noindent and 
$A_{7}$, $A_{10}$ are related to $A_{6}, A_{9}$ as follows:
the $(i, j)$-entry of $A_{7}$ (resp. $A_{10}$) is the inverse of the $(i,j)$-entry of $A_{6}$ (resp. $A_{9}$) for $(i,j)\in\{(2,1),(3,1),(4,2),(4,3)\}$,  
while all other entries remain the same. Also,  

$\hspace{.5cm}
A_8=\begin{pmatrix}
q^{-2}							&\\
- q^{-1} \mathbf{d}  	& q^{-2}&\\
- q^{-4} \mathbf{d}  	& (q^{-5}-q)							& q^{-2}					&\\		
q^{-2}\mathbf{t}	& (q^{-7}{-}q^{-3})			& (q^{-4}{-}1)	& 1\\
\mathbf{s}					& {-}q^{{-}5}\mathbf{d} & q^2\mathbf{t}& (q^{{-}2}{-}q^2)	&q^{-2}\\
q^{-3}\mathbf{s}& (q^2{-}q^{-8})\mathbf{d} 	& {-}q^{-5}\mathbf{d}
						&(q^{-5}{-}q^{-1}) & (q^{-5}-q)&q^{-2}	\\		
q^{-4}\mathbf{s}+\mathbf{c}  & q^{-3}\mathbf{s} & \mathbf{s}&  \mathbf{t}	
&- q^{-4} \mathbf{d} & - q^{-1} \mathbf{d}  & q^{-2}
\end{pmatrix}.
$
			
\noindent
Here 
	$\mathbf{b}=q^2{+}1{+}q^{-2}$,
	$\mathbf{c}=q^2{+}q^{-2}{-}q^{-6}{-}1$,
	$\mathbf{d}=1-q^{-2}$,
	$\mathbf{s}=(q^{-8}{+}q^{-4}{-}q^{-6}{-}1)$, and 
	$\mathbf{t}=(q^{-6}{+}1{-}q^{-4}{-}q^{2})$.

\begin{proof}[Proof of Proposition \ref{pre-Am}]
The proposition merely spells out \eqref{braid} in the present case. 
The $R$-matrix involved is that given by \eqref{Rmatrix1}. 
To prove the proposition, one only needs to consider the case with $i=1$ and $j=2$. 

We have the explicit bases for the simple submodules of 
$S_q(V)_1\ot S_q(V)_1 = V\ot V$ constructed in Section \ref{sect:st-mod}.   
Expressing $X_{1 a} X_{2 b}$ in terms of the basis elements, 
we can then easily work out the right hand side of \eqref{braid} (for $i=1, j=2$)
using the formula \eqref{Rmatrix1} for the $R$-matrix  $\check{R}$. 

This requires nothing more than elementary linear algebra. 
The sort of calculations involved are similar to those in the proof of \cite[Lemma 3.7]{lzz:ft}, 
but much more lengthy and tedious. We omit the details, which are not very illuminating anyway. 
\end{proof}

\subsection{Classical limits of braided symmetric algebras}\label{rem:qto1}

Denote $\AA:=\CC[q, q^{-1}]\subset\CC(q)$, and let ${\rm U}_\AA(G_2)$ be the quantum group of $G_2$ over $\AA$. 
Let $V_\AA=\sum_a \AA v_a$ and $V_\AA^{\ot n}= \underbrace{V_{\AA}\ot_{\AA} V_{\AA}\ot_{\AA} \dots \ot_{\AA} V_{\AA} }_n$. Then we have the  ${\rm U}_\AA(G_2)$-algebras $S_\AA(V_\AA)= S_q(V)_n\cap V_\AA^{\ot n}$ and 
$\cA_m(V_\AA) = \underbrace{S_\AA(V_\AA)\ot _\AA S_\AA(V_\AA)\ot _\AA \dots \ot_\AA S_\AA(V_\AA)}_m$, 
where  $\cA_m(V_\AA)$ is endowed with the braided multiplication.  
The classical limits of $S_q(V)$ and $\cA_m(V)$, which will be denoted by $\ol{S_q(V)}$ and $\ol{\cA_m(V)}$ respectively, are the specialisations of $S_\AA(V_\AA)$ and $\cA_m(V_\AA)$ with respect to the $\CC$-linear ring homomorphism $\theta: \AA\lra\CC$ with $\theta(q)=1$. Explicitly, 
\[
\ol{S_q(V)}= S_\AA(V_\AA)\ot_\theta \CC, \quad \ol{\cA_m(V)}= \cA_m(V_\AA)\ot_\theta \CC. 
\]

Theorem \ref{homogeneous} states that $S_q(V)=\CC(q)\Phi+ \sum_{n\ge 0} \CH_n$, where 
$\CH_n\simeq V_{n\lambda_1}$. As any simple $\UG$-module specialises to 
the corresponding simple $G_2$-module in the classical limit, we obtain 
\beq\label{eq:Sq0}
\ol{S_q(V)}= \CC\Phi^{(0)}+\sum_{n\ge 0} \CH^{(0)}_n,
\eeq
where $\Phi^{(0)}$ is the specialisation of $\Phi$, and $\CH^{(0)}_n$ as defined in Section \ref{sect:classical} but regarded as the specialisation of $\CH_n$ here.

We have the following result. 
\begin{lemma} As algebras, 
\begin{enumerate}
\item  
$
\ol{S_q(V)}\simeq S(\CC^7)/\langle S^1(\CC^7)\phi\rangle,
$
where $\langle S^1(\CC^7)\phi\rangle$ is the two-sided ideal in $S(\CC^7)$ generated by 
$ S^1(\CC^7)\phi$; and 

\item 
$
\ol{\cA_m(V)} = \ol{S_q(V)}^{\ot m} 
$
with component wise multiplication.   
\end{enumerate}
\begin{proof}
One can easily see part (1) by comparing the equations \eqref{eq:harmonic} and \eqref{eq:Sq0}. 
Part (2) follows from the definition of $\cA_m(V)$ and the fact that  the the $R$-matrix $\check{R}$ reduces to the permutation $x\ot y\mapsto y\ot x$, for all $x, y$,  in the classical limit. 
\end{proof}
\end{lemma}

\begin{remark}
The true interest of studying $\cA_m(V)$ in that we shall regard it as defining a non-commutative geometry admitting a ${\rm U}_q(G_2)$-action. Its classical limit yields some space with $G_2$-symmetry in the commutative world, which is not a variety, but an affine scheme, as $\ol{\cA_m(V)}$ is not an integral domain. 
\end{remark}

\section{Diagrammatics for ${\rm U}_q(G_2)$-representations}\label{sect:diagrams}
In this section, we develop diagrammatics suitable for studying
${\rm U}_q(G_2)$-invariants in braided symmetric algebras.  

%
\subsection{The category ${\mathscr V}$ of tensor modules for ${\rm U}_q(G_2)$}
Denote by ${\mathscr V}$  the full sub-category of $\UG$-modules 
with objects $V^{\ot r}$ for all $r\in\ZZ_+$. 
This is a braided tensor category, which is semi-simple. 
We shall refer to it as the category of tensor modules for $\UG$. 

We now discuss some morphisms of ${\mathscr V}$, 
which will play important roles later.

By \eqref{eq:VtV},  the tensor product $V\otimes V$ contains the $1$-dimensional $\UG$ submodule
$V_0=\CC c_0$, where the explicit formula for $c_0$ is given by \eqref{basis0}. Recall from \eqref{eq:c0} that $
c_0=\sum_{a, b} \varphi_{a b} v_a\ot v_b, 
$
where $\varphi_{a b}\in \CC(q)$ 
can be read off \eqref{basis0}:
\beq\label{eq:check-C}
\baln
& \varphi_{0 0}= 1, \quad  \varphi_{1, -1}= q^4, \quad \varphi_{-1, 1}= q^{-6}, \quad \varphi_{2, -2}= q^3, \quad 
\varphi_{-2, 2}= q^{-5}, \\
&\varphi_{3, -3}= 1, \quad \varphi_{-3, 3}= q^{-2}, \quad \text{rest} =0.
\ealn
\eeq
Note that the $\UG$-invariant non-degenerate bilinear form \eqref{bilinear} on $V$ satisfies
\beq
(v_a, v_b)= (\varphi^{-1})_{a b}, \quad \varphi=(\varphi_{a b}).  
\eeq

Now we have the following $\UG$-maps
\beq
&\check{C}:& \mathbb{C}(q) \longrightarrow V \otimes V, \quad 1\mapsto c_0, \label{phi} \\
&\hat{C}:&  V \otimes V\lra \CC(q), \quad  v_a\ot v_b\mapsto (v_a, v_b), \quad \forall a, b. \label{eq:hat-C}
\eeq

The tensor product $V\ot V$ also contains a simple $\UG$-submodule isomorphic to $V$ with multiplicity $1$.  
As $V\ot V$ is semi-simple, $\Hom_{\UG}(V, V\ot V) \simeq \Hom_{\UG}(V\ot V, V)\simeq\CC(q)$. 
Thus there exist $\UG$-maps,  
\beq\label{eq:gam}
\baln
\upgamma:  V\longrightarrow V\otimes V, \quad p: V\otimes V \longrightarrow V, 
\ealn
\eeq
which are unique up to scalar multiples.  
We choose 
$\upgamma$ so that 
$\upgamma(v_1)=b_1$, thus  
\[
\upgamma(v_a) = b_a, \quad \forall a=0, \pm1, \pm2, \pm3.
\]
We also take the map $p$ so that 
\beq
p\circ \upgamma=-([7]_q{-}1)\id_V. \label{eq:p-gam}
\eeq

We have the following result. 
\begin{lemma}\label{lem:easy}
The following relations hold for $\chi_0(V)= -12$ and $\varepsilon_0(V)=1$. 
\beq
&(\hat{C}\ot\id_V)\circ(\id_V\ot \check{C}) = \id_V, \quad   (\id_V\ot \hat{C})\circ(\check{C}\ot \id_V)= \id_V, \label{eq:C-1}\\
& \check{C} \circ\hat{C} =  \dim_q V P[0], \quad \hat{C}\circ \check{C} = \dim_q V, \label{eq:C-2}\\
& \check{R}^{\pm 1}\circ  \check{C} =\varepsilon_0(V)q^{\pm \chi_0(V)}  \check{C}, \quad \hat{C}\circ  \check{R}^{\pm 1}= \varepsilon_0(V)q^{\pm \chi_0(V)}  \hat{C}; \label{eq:C-3} \\
&(\id_V \ot \hat{C})\circ(\upgamma\ot \id_V)= -q^{-6} p =(\hat{C}\ot\id_V)\circ(\id_V\ot \upgamma), \label{eq:gamma-p} \label{up-p}\\
&\hat{C}\circ(\id\ot \hat{C}\ot \id)\circ(\upgamma\ot \upgamma)=q^{-6}\left([7]_q-1\right)\hat{C}. \label{eq:gamma-C}
\eeq

\end{lemma}  

\begin{proof}
Equations \eqref{eq:C-1}--\eqref{eq:C-3}
are familiar results valid for any self dual simple module for any quantum (super)group with the appropriate $\chi_0(V)$ and $\varepsilon_0(V)$. The relations in \eqref{eq:C-2} can be easily verified using the definitions of $\check{C}$ and $\hat{C}$, 
and the relations in \eqref{eq:C-3} are a consequence of the spectral decomposition of $\check{R}$ and the definitions of $\check{C}$ and $\hat{C}$. Equation \eqref{eq:C-1} involves the canonical identifications $V\ot\CC(q) \simeq V \simeq \CC(q)\ot V$. To illustrate the proof, we consider the first relation as an example. We have
$$
(\hat{C}\ot\id_V)\circ(\id_V\ot \check{C})(v_a) = \sum_{b, c} \varphi_{b c} (v_a, v_b) v_c 
= \sum_{b, c}  (\varphi^{-1})_{a b} \varphi_{b c} v_c = v_a, \quad \forall a, 
$$
as desired.  

Equations \eqref{eq:gamma-p} and \eqref{eq:gamma-C} are verified by direct calculations using \eqref{eq:gam} and \eqref{eq:p-gam}, and the explicit basis $B(\lambda_1)$ of the submodule $V_{\lambda_1}$ in $V\ot V$.  
Since the maps involved are all $\UG$-morphisms, we only need to show 
that the relations hold when applied to the highest weight vectors 
$v_1\ot v_1$, $v_1 \ot v_2 - q v_2\ot v_1$,  $b_1$ and $c_0$ of the simple submodules of $V\ot V$.  
However, very lengthy calculations are needed even for these vectors.  

Let us prove \eqref{eq:gamma-p} for $b_1$. Using the explicit formulae for $b_1$, we obtain
\begin{align*}
(\upgamma\ot \id_V)(b_1)
=&b_{1}\otimes v_{0}-q^{-6}b_{0}\otimes v_{1}-(q^{-3}+q^{-1})\left(b_{2}\otimes v_3-q^{-3}b_3\otimes v_2\right).
\end{align*}
Since $\hat{C}(v_i, v_j)=0$ unless $i+j=0$, we have 
\begin{align*}
&(\id_V \ot \hat{C})\circ(\upgamma\ot \id_V)(b_1) \\
=&(\id_V \ot \hat{C})(v_1\otimes v_0\otimes v_{0}+q^{-6}q^{-2}v_1\otimes v_{-1}\otimes v_{1}\\
& -(q^{-3}+q^{-1})
(-q^{-1}v_1\otimes v_{-3}\otimes v_3-q^{-3}q^{-1}v_1\otimes v_{-2}\otimes v_2))\\
=&v_1+q^{-8}q^{-4}v_1+(q^{-3}+q^{-1})q^{-1}v_1+(q^{-3}+q^{-1})q^{-4}q^{-3}v_1\\
=&q^{-6}([7]_q-1)v_1=-q^{-6}p(b_1).
\end{align*}
This proves the first equality. The proof of the second equality is similar:
\begin{align*}	
&(\hat{C}\ot\id_V)\circ(\id_V\ot \upgamma)(b_1)\\				
=&(\hat{C}\ot\id_V)(v_{1}\otimes b_{0}-q^{-6}v_{0}\otimes b_{1}
-(q^{-3}+q^{-1})(v_{2}\otimes b_3-q^{-3}v_3\otimes b_2))\\
=&(\hat{C}\ot\id_V)(q^{-6}v_{1}\otimes v_{-1}\otimes v_{1}+q^{-6}q^{-6}v_{0}\otimes v_{0}\otimes v_{1}\\
&-(q^{-3}+q^{-1})
(-q^{-6}v_{2}\otimes v_{-2}\otimes v_{1}-q^{-3}q^{-6}v_3\otimes v_{-3}\otimes v_{1}))\\
=&\left(1+q^{-12}+q^{-3}(q+q^{-1})+q^{-9}(q+q^{-1})\right)v_1\\
=&q^{-6}([7]_q-1)v_1=-q^{-6}p(b_1).
\end{align*}
This shows that  \eqref{eq:gamma-p} indeed holds when applied to $b_1$.  
The proofs for the  other $3$ highest weight vectors are similar but simpler. We omit the details. 

Now we prove \eqref{eq:gamma-C} for $c_0$. Apply both sides to $c_0$. We have  
\[
\baln
\text{LHS}:=& \hat{C}\circ(\id\ot \hat{C}\ot \id)\circ(\upgamma\ot \upgamma)(c_0) 
= \sum_{a, c} \varphi_{a c} \hat{C}\circ(\id\ot \hat{C}\ot \id) (b_a\ot  b_c), \\
\text{RHS}:=& q^{-6}\left([7]_q-1\right)\hat{C}(c_0)= q^{-6}\left([7]_q-1\right)\dim_qV, 
\ealn
\]
where the second relation of \eqref{eq:C-2} was used in obtaining the result for $\text{RHS}$. To consider $\text{LHS}$, 
recall that $\hat{C}\circ(\id\ot \hat{C}\ot \id) (b_a\ot  b_c)=0$ if $a+c\ne 0$.  
We can show that  
\beq\label{eq:comp+-}
\begin{aligned}
\hat{C}\circ(\id\ot \hat{C}\ot \id)(b_1\otimes b_{-1})&=[7]_q-1, \\
\hat{C}\circ(\id\ot \hat{C}\ot \id)(b_{-1}\otimes b_{1})&=q^{-10}([7]_q-1),\\
\hat{C}\circ(\id\ot \hat{C}\ot \id)(b_2\otimes b_{-2})&=q^{-1}([7]_q-1),\\
\hat{C}\circ(\id\ot \hat{C}\ot \id)(b_{-2}\otimes b_{2})&=q^{-9}([7]_q-1),\\
\hat{C}\circ(\id\ot \hat{C}\ot \id)(b_3\otimes b_{-3})&=q^{-4}([7]_q-1),\\
\hat{C}\circ(\id\ot \hat{C}\ot \id)(b_{-3}\otimes b_{3})&=q^{-6}([7]_q-1),\\
\hat{C}\circ(\id\ot \hat{C}\ot \id)(b_{0}\otimes b_{0})&=q^{-6}([7]_q-1).  
\end{aligned}
\eeq
Consider, e.g.,  the first relation.  The terms in $b_1\ot b_{-1}$ with non-zero contributions are 
\begin{align*}
& v_1\ot v_0\ot v_0\ot v_{-1}+q^{-12}v_0\ot v_1\ot v_{-1}\ot v_{0}\\
&+q^{-2}(q^{-3}+q^{-1})v_2\ot v_3\ot v_{-3}\ot v_{-2}+q^{-8}(q^{-3}+q^{-1})v_3\ot v_2\ot v_{-2}\ot v_{-3}.
\end{align*}
Hence 
\begin{align*}
&\hat{C}\circ(\id\ot \hat{C}\ot \id)(b_1\otimes b_{-1})\\
=&\hat{C}(v_1\ot v_{-1}+q^{-6}v_0\ot v_{0}+(q^{-3}+q^{-1})v_2\ot v_{-2}+q^{-3}(q^{-3}+q^{-1})v_3\ot v_{-3})\\
=&q^6+q^{-6}+q^5(q^{-3}+q^{-1})+q^{-1}(q^{-3}+q^{-1})
=[7]_q-1.
\end{align*}
The other relations in \eqref{eq:comp+-} can be proved similarly.

It immediately follows from \eqref{eq:comp+-}  that 
\begin{align*}
\text{LHS}
=&\left(q^4+q^{-16}+q^{2}+
q^{-14}+q^{-4}+q^{-8}+q^{-6}\right)([7]_q-1)\\
=&q^{-6}([7]_q-1)\dim_qV.
\end{align*}	
This shows that  \eqref{eq:gamma-C} indeed holds when applied to $c_0$.  
The proofs for other $3$ highest weight vectors are similar but simpler.
\end{proof}

\subsection{Diagrammatic description of ${\mathscr V}$}\label{diagdep}
\subsubsection{A diagram category and a tensor functor}
The diagram category suitable for describing ${\mathscr V}$  is a monoidal category of un-oriented framed tangles \cite{fy:bc} with coupons 
(i.e., a category of un-direct oriented ribbon graphs with coupons if adopting the terminology of \cite{rt:rg}). 
The set of objects of the diagram category is $\ZZ_+$, and the morphisms are generated by the ribbon graphs in Figure \ref{fig:generators} under composition and juxtaposition.
\begin{figure}[h]
\begin{gather*}
	\begin{tikzpicture}[baseline=5pt,scale=1,color=\clr]
		\draw[-,line width=1pt](0,0)--(0,0.7);
	\end{tikzpicture}\quad\quad
	\begin{tikzpicture}[baseline=25pt,scale=0.7,color=\clr]
		\draw[-,line width=1pt] (0.5,2)--(1.5,1);
		\draw [-,line width=1pt](0.5,1)--(0.9,1.4);
		\draw [-,line width=1pt](1.1,1.6)--(1.5,2);
	\end{tikzpicture}\quad\quad
	\begin{tikzpicture}[baseline=25pt,scale=0.7,color=\clr]
		\draw [-,line width=1pt](5.5,1)--(6.5,2);
		\draw[-,line width=1pt] (6.5,1)--(6.1,1.4);
		\draw [-,line width=1pt](5.9,1.6)--(5.5,2);
	\end{tikzpicture}\quad\quad
	\begin{tikzpicture}[baseline=5pt,scale=1,color=\clr]
		\draw[-,line width=1pt](0,0.7)parabola bend (0.35,0) (0.7,0.7);
	\end{tikzpicture}\quad\quad
	\begin{tikzpicture}[baseline=5pt,scale=1,color=\clr]
		\draw[-,line width=1pt,color=\clr](0,0)parabola bend (0.35,0.7) (0.7,0);
	\end{tikzpicture}\quad\quad
	\begin{tikzpicture}[baseline=20pt,scale = 0.7,color=\clr]
		\draw[-, line width=1pt] (0.5,1.4)--(0.5,2);
		\draw [-, line width=1pt](0.5,1.4)--(0,0.8);
		\draw [-, line width=1pt](0.5,1.4)--(1,0.8);
	\end{tikzpicture}
\end{gather*}
\caption{Generators~ of~ un-oriented ~$3$-tangle~diagrams}
\label{fig:generators}
\end{figure}
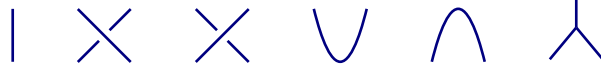

The first four diagrams in Figure 1 are the standard generators for the unoriented tangle category obeying  the usual relations
(i..e., the relations in \cite{rt:rg} without orientation, also see \cite{fy:bc}). The last generator is a coupon, which obeys the ``sliding" relations

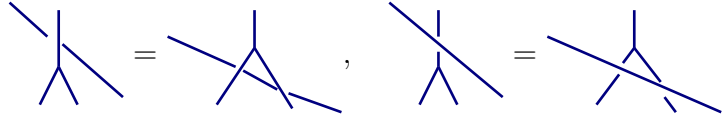
\begin{figure}[h]
\begin{gather*}
\begin{tikzpicture}[baseline=37pt,scale=0.5,color=\clr]
\draw[-,line width=1pt] (0.5,1.5)--(1,2.5);\draw[-,line width=1pt] (1.5,1.5)--(1,2.5);\draw[-,line width=1pt] (1,2.5)--(1,4);
\draw[-,line width=1pt] (-0.3,4.2)--(0.7,3.3);\draw[-,line width=1pt] (1.1,3.1)--(2.7,1.7);
\end{tikzpicture}
=\begin{tikzpicture}[baseline=37pt,scale=0.5,color=\clr]
\draw[-,line width=1pt] (0,1.5)--(1,3);\draw[-,line width=1pt] (2,1.4)--(1,3);\draw[-,line width=1pt] (1,3)--(1,4);
\draw[-,line width=1pt] (-1.3,3.3)--(0.5,2.5);\draw[-,line width=1pt] (0.7,2.4)--(1.6,1.9);
\draw[-,line width=1pt] (1.9,1.8)--(3.3,1.3);
\end{tikzpicture}, 
\quad
\begin{tikzpicture}[baseline=37pt,scale=0.5,color=\clr]
\draw[-,line width=1pt] (0.5,1.5)--(1,2.5);\draw[-,line width=1pt] (1.5,1.5)--(1,2.5);
\draw[-,line width=1pt] (1,2.5)--(1,2.9);\draw[-,line width=1pt] (1,3.3)--(1,4);
\draw[-,line width=1pt] (-0.3,4.2)--(2.7,1.7);
\end{tikzpicture}
=\begin{tikzpicture}[baseline=37pt,scale=0.5,color=\clr]
\draw[-,line width=1pt] (0,1.5)--(0.6,2.3);
\draw[-,line width=1pt] (0.7,2.6)--(1,3);
\draw[-,line width=1pt] (2.1,1.3)--(1.8,1.7);
\draw[-,line width=1pt] (1.7,2.1)--(1,3);
\draw[-,line width=1pt] (1,3)--(1,4);
\draw[-,line width=1pt] (-1.3,3.3)--(3.3,1.3);\end{tikzpicture}. 
\end{gather*}
\caption{Relations involving  the coupon}
\label{fig:slide}
\end{figure}

\noindent
We denote by $\widetilde{\mathscr T}$ this monoidal category of un-oriented framed tangles with coupons.

However, for the purpose of studying representations of $\UG$, it is more convenient to consider a quotient monoidal category 
of $\widetilde{\mathscr T}$ by imposing further relations on the coupon. 
Fix non-zero scalars $\alpha, \beta$. Denote by ${\mathfrak I}(\alpha, \beta)$ the tensor ideal of $\widetilde{\mathscr T}$ generated by the following morphisms under composition and juxtaposition of diagrams. 
\begin{equation}\label{J-gen}
\begin{tikzpicture}[baseline=29pt,scale=0.6,color=\clr]
		\draw[-,line width=1pt] (0.5,2)--(1.5,1);\draw [-,line width=1pt](0.5,1)--(0.9,1.4);\draw [-,line width=1pt](1.1,1.6)--(1.5,2);
	     \draw[-,line width=1pt] (0.5,2)--(1,2.5);\draw[-,line width=1pt] (1.5,2)--(1,2.5);\draw[-,line width=1pt] (1,2.5)--(1,3);
\end{tikzpicture}
-\alpha\begin{tikzpicture}[baseline=29pt,scale=0.6,color=\clr]
		\draw[-,line width=1pt] (0.5,1)--(1,2);\draw[-,line width=1pt] (1.5,1)--(1,2);\draw[-,line width=1pt] (1,2)--(1,3);
	\end{tikzpicture},\qquad
\begin{tikzpicture}[baseline=30pt,scale = 0.6,color=\clr]
		\draw[-,line width=1pt] (1.5,0.8)to (1.5,1.5);
		\draw[-,line width=1pt] (1.5,1.5) to (1,2);
		\draw [-,line width=1pt](1.5,1.5) to (2,2);
		\draw[-,line width=1pt] (1,2) to (1.5,2.5);
		\draw[-,line width=1pt] (2,2) to (1.5,2.5);
		\draw[-,line width=1pt] (1.5,2.5) to (1.5,3);
	\end{tikzpicture}
	-\beta\ 
	\begin{tikzpicture}[baseline=30pt,scale = 0.6,color=\clr]
		\draw [-,line width=1pt](8.7,0.8) to (8.7,3);
	\end{tikzpicture}~, 
\qquad
\begin{tikzpicture}[baseline=37pt,scale=0.6,color=\clr]
\draw[-,line width=1pt] (0.5,2)--(1,2.5);\draw[-,line width=1pt] (1.5,2)--(1,2.5);\draw[-,line width=1pt] (1,2.5)--(1,3);
\draw[-,line width=1pt](0.5,2)parabola bend (1,1.5) (1.5,2);
\end{tikzpicture}~, 
\qquad
\begin{tikzpicture}[baseline=20pt,scale = 0.6,color=\clr]		
	\draw[-, line width=1pt] (6,1.4)--(6,2);
	\draw [-, line width=1pt](6,1.4)--(5.5,0.8);
	\draw [-, line width=1pt](6,1.4)--(6.5,0.8);
	\draw[-, line width=1pt] (5.5,0.8)parabola(5.2,2);
\end{tikzpicture}-
\begin{tikzpicture}[baseline=20pt,scale = 0.6,color=\clr]		
	\draw[-, line width=1pt] (3,1.4)--(3,2);
	\draw [-, line width=1pt](3,1.4)--(2.5,0.8);
	\draw [-, line width=1pt](3,1.4)--(3.5,0.8);\draw[thick, line width=1pt] (3.5,0.8)parabola(3.8,2);
\end{tikzpicture}~.
\end{equation}

\begin{definition}
Fix non-zero scalars $\alpha, \beta$. Let ${\mathscr T}(\alpha, \beta)=\widetilde{\mathscr T}/{\mathfrak I}(\alpha, \beta)$. 
\end{definition}
This quotient  category is monoidal.  
\begin{remark} 
We shall use the same pictures for diagrams in  $\widetilde{\mathscr T}$ to denote their images in the quotient category ${\mathscr T}(\alpha, \beta)$. 
\end{remark}

We now consider some frequently used diagrams in ${\mathscr T}(\alpha, \beta)$, 
some of which will be given more convenient pictorial representations. 
\begin{enumerate}[i)]
\item Denote 
\begin{tikzpicture}[baseline=20pt,scale = 0.6,color=\clr]		
	\draw[-, line width=1pt] (1,0.8)--(1,1.5);
	\draw[-, line width=1pt] (1,1.5)--(0.5,2);
	\draw[-, line width=1pt] (1,1.5)--(1.5,2);	
\end{tikzpicture}:=\begin{tikzpicture}[baseline=20pt,scale = 0.6,color=\clr]		
	\draw[-, line width=1pt] (6,1.4)--(6,2);
	\draw [-, line width=1pt](6,1.4)--(5.5,0.8);
	\draw [-, line width=1pt](6,1.4)--(6.5,0.8);
	\draw[-, line width=1pt] (5.5,0.8)parabola(5.2,2);
\end{tikzpicture}=
\begin{tikzpicture}[baseline=20pt,scale = 0.6,color=\clr]		
	\draw[-, line width=1pt] (3,1.4)--(3,2);
	\draw [-, line width=1pt](3,1.4)--(2.5,0.8);
	\draw [-, line width=1pt](3,1.4)--(3.5,0.8);\draw[thick, line width=1pt] (3.5,0.8)parabola(3.8,2);
\end{tikzpicture}.

\item Composing 
\begin{tikzpicture}[baseline=20pt,scale = 0.6,color=\clr]		
	\draw[-, line width=1pt] (1,0.8)--(1,1.5);
	\draw[-, line width=1pt] (1,1.5)--(0.5,2);
	\draw[-, line width=1pt] (1,1.5)--(1.5,2);	
\end{tikzpicture}
and
\begin{tikzpicture}[baseline=20pt,scale = 0.7,color=\clr]
		\draw[-, line width=1pt] (0.5,1.4)--(0.5,2);
		\draw [-, line width=1pt](0.5,1.4)--(0,0.8);
		\draw [-, line width=1pt](0.5,1.4)--(1,0.8);
\end{tikzpicture}
with the former diagram on top, 
we obtain  
\begin{tikzpicture}[baseline=22pt,scale=0.6,color=\clr]		  
	\draw[-,line width=1pt] (1,1.3)--(1,1.7);
	\draw[-,line width=1pt] (1,1.7)--(0.5,2);
	\draw[-,line width=1pt] (1,1.7)--(1.5,2);
	\draw[-,line width=1pt] (1,1.3)--(0.5,1);
	\draw[-,line width=1pt] (1,1.3)--(1.5,1);    
\end{tikzpicture}.
\item 
We have
$
\begin{tikzpicture}[baseline=5pt,-,color=\clr]
		\draw[-,line width=1pt](0,0.4)parabola bend (0.3,0) (0.6,0.6);
		\draw[-,line width=1pt](0,0.4)--(-0.2,0.6);\draw[-,line width=1pt](0,0.4)--(0.2,0.6);
	\end{tikzpicture}
	=\begin{tikzpicture}[baseline=5pt,-,color=\clr]
		\draw[-,line width=1pt](0,0.6)parabola bend (0.3,0) (0.6,0.4);
		\draw[-,line width=1pt](0.6,0.4)--(0.4,0.6);\draw[-,line width=1pt](0.6,0.4)--(0.8,0.6);
	\end{tikzpicture},  
$
and will also denote both diagrams by 
$
\begin{tikzpicture}[baseline=-11pt,scale=1.3,color=\clr]
		\draw[-,line width=1pt](0,0)parabola bend (0.3,-0.4) (0.6,0);\draw[-,line width=1pt](0.3,-0.4)--(0.3,0);
	\end{tikzpicture}~. 
$

\item 
By concatenating 
\begin{tikzpicture}[baseline=-13pt,scale=0.75,color=\clr]
	\draw[-,line width=1pt](0,-0.8)--(0,-0.2);\draw[-,line width=1pt](1.4,-0.8)--(1.4,-0.2);
	\draw[-,line width=1pt](0.3,-0.4)--(0.1,-0.2);\draw[-,line width=1pt](0.3,-0.4)--(0.5,-0.2);
	\draw[-,line width=1pt](1.1,-0.4)--(0.9,-0.2);\draw[-,line width=1pt](1.1,-0.4)--(1.3,-0.2);
	\draw[-,line width=1pt](0.3,-0.4)parabola bend (0.7,-0.8) (1.1,-0.4);
\end{tikzpicture} 
and $\caps~ \ids ~ \ids ~ \caps$ with the latter diagram on top, we obtain the diagram 
\begin{tikzpicture}[baseline=-8pt,scale=0.5,color=\clr]
\draw[-,line width=1pt](-0.6,-0.8)parabola bend (-0.3,0.2) (0,0); 
\draw[-,line width=1pt](0,0)parabola bend (0.3,-0.4) (0.6,0.2);
\draw[-,line width=1pt](0.8,0.2)parabola bend (1.1,-0.4) (1.4,0);	
\draw[-,line width=1pt](1.4,0)parabola bend (1.7,0.2) (2,-0.8);
\draw[-,line width=1pt](0.3,-0.4)parabola bend (0.7,-0.8) (1.1,-0.4);
\end{tikzpicture}. 
We will denote it by  
\begin{tikzpicture}[baseline=2pt,scale=0.5,color=\clr]
\draw[-,line width=1pt](8.3,1)--(8.8,0.5);
\draw[-,line width=1pt](8.3,0)--(8.8,0.5);
\draw[-,line width=1pt](8.8,0.5)--(9.2,0.5);
\draw[-,line width=1pt](9.2,0.5)--(9.7,1);
\draw[-,line width=1pt](9.2,0.5)--(9.7,0);
\end{tikzpicture}.

\end{enumerate}

The following relation is clear:
\begin{tikzpicture}[baseline=2pt,scale=0.5,color=\clr]
	\draw[-,line width=1pt](8.3,1)--(8.8,0.5);
	\draw[-,line width=1pt](8.3,0)--(8.8,0.5);
	\draw[-,line width=1pt](8.8,0.5)--(9.2,0.5);
	\draw[-,line width=1pt](9.2,0.5)--(9.7,1);
	\draw[-,line width=1pt](9.2,0.5)--(9.7,0);
	\draw[-,line width=1pt](8.3,0)parabola bend (8.1,-0.1)(7.9,1);
	\draw[-,line width=1pt](9.7,1)parabola bend (9.9,1.1)(10.1,0);
\end{tikzpicture}
=\begin{tikzpicture}[baseline=22pt,scale=0.6,color=\clr]		  
	\draw[-,line width=1pt] (1,1.3)--(1,1.7);
	\draw[-,line width=1pt] (1,1.7)--(0.5,2);
	\draw[-,line width=1pt] (1,1.7)--(1.5,2);
	\draw[-,line width=1pt] (1,1.3)--(0.5,1);
	\draw[-,line width=1pt] (1,1.3)--(1.5,1);    
\end{tikzpicture}.

\medskip

The following theorem is required for the proof of the main result of this paper,
the first fundamental theorem of invariant theory for $\UG$ (see Theorem \ref{FFT}). 

\begin{theorem}\label{thm:functor}
Denote by ${\mathscr T}$ the category ${\mathscr T}(\alpha, \beta)$ with  
$\alpha = -q^{-6}$ and $\beta=[7]_q-1$. 
\begin{enumerate}
\item There is a unique tensor functor ${\mathcal F}: {\mathscr T}\lra {\mathscr V}$, which maps an object $r$ of ${\mathscr T}$ to $V^{\ot r}$ for all $r\in\ZZ_+$, and acts on the generators of morphisms of ${\mathscr T}$ as follows
\begin{gather*}
\mathcal{F}\left(\, \ids\, \right)=\id_V, \quad
\mathcal{F}\left(\caps\right)=\widehat{C},\quad
\mathcal{F}\left(\cups\right)=\check{C},\\
\mathcal{F}\left(\cross\right)=\check{R},\quad
\mathcal{F}\left(\Icross\right)=\check{R}^{-1},\quad
\mathcal{F}\left(\merges\right)=q^{-3}p.
\end{gather*}
\item The functor ${\mathcal F}: {\mathscr T}\lra {\mathscr V}$ is essentially surjective and full. 
\end{enumerate}
\end{theorem}
\begin{proof}[Remarks on the proof]
The proof of part (1) requires the verification of the defining relations among the generators of morphisms of ${\mathscr T}$, and this is a consequence of Lemma \ref{lem:easy} and general properties of the $R$-matrix.  For the tangle generators, i.e., 
the elements of Figure \ref{fig:generators} but the last one, 
the verification is the same as, e.g.,  in the case of the natural representation of 
${\rm U}_q(\frak{osp}_{m|2n})$ \cite{LZZ-2}. 
The coupon is particular to our case, and we define $\mathcal{F}$ so that 
it maps the coupon to $q^{-3}p$.    
Then for $\alpha = -q^{-6}$ and $\beta=[7]_q-1$, one can verify by direct calculations that the elements in \eqref{J-gen} are all mapped to $0$ by $\mathcal{F}$. 
The calculations are lengthy but straightforward. 

Evidently the functor is essentially surjective. It follows from \cite[\S8]{lz:sm} (and also \cite{Ms}) that the images under ${\mathcal F}$ of the  diagrams in Figure \ref{fig:generators}  generate all morphisms of ${\mathscr V}$ by composition and tensor product. 
\end{proof}

\begin{remark}
Note that ${\mathcal F}({\mathscr T})$ is a quantum analogue of the category ${\mathscr T}_{\Gamma}$ with $\Gamma=\{\gamma_0,\gamma_1,\gamma_2\}$ in \cite{benkart},
where $\gamma_0,\gamma_1,\gamma_2$ corresponds to $\eqref{eq:loop}$, \eqref{checkTRI} and \eqref{cross}, respectively.
\end{remark}


\subsubsection{Frequently used morphisms} {\ }

\noindent{\bf Notation and convention}.  Hereafter we will suppress the functor $\mathcal{F}$ from all notation for simplicity. Thus all diagrams will be regarded as morphisms in ${\mathcal F}({\mathscr T})$.   
 
\smallskip

We list some relations among morphisms for later use.

The second relation of \eqref{eq:C-2} can now be written as  
\beq\label{eq:loop}
\begin{tikzpicture}[baseline = -2mm,scale=.8,color=\clr]
\draw[-,line width=1pt] (0.06,-.51) to [out=15,in=-15] (0.06,.31);
\draw[-,line width=1pt] (0.06,-.51) to [out=165,in=-165] (0.06,.31);
\end{tikzpicture} = \dim_q V; 
\eeq

The following relations are consequences of properties of the $R$-matrix. 
\begin{equation}\label{checkID}
	\begin{tikzpicture}[baseline=55pt,scale=0.7,color=\clr]
		\draw[-,line width=1pt](0,3)parabola bend (0.3,2.5) (0.6,3.9);
		\draw[-,line width=1pt](0.4,3)arc (0:180:0.2);
		\draw[-,line width=1pt](0.4,3)arc (0:180:0.2);  
		\draw[-,line width=1pt](0.55,2.9)parabola (0.7,1.9);
	\end{tikzpicture}=q^{-12}{\ }
	\begin{tikzpicture}[baseline=55pt,scale=0.7,color=\clr]\draw[-,line width=1pt](0,1.9)--(0,3.9);\end{tikzpicture},\quad
	\begin{tikzpicture}[baseline=55pt,scale=0.7,color=\clr]
		\draw[-,line width=1pt](7,2.7)parabola bend (7.4,3.1) (7.7,1.9);
		\draw[-,line width=1pt](7.5,2.7)arc (0:-180:0.25);
		\draw[-,line width=1pt](7.62,2.8)parabola (7.9,3.9);\end{tikzpicture}=q^{12}{\ }
	\begin{tikzpicture}[baseline=55pt,scale=0.7,color=\clr]
		\draw[-,line width=1pt](0,1.9)--(0,3.9);	
	\end{tikzpicture}~.
\end{equation}
\begin{equation}\label{checkPOLO}
	\begin{tikzpicture}[baseline=2pt,scale=0.6,color=\clr]
		\draw[-,line width=1pt](0,0)arc (0:-180:0.5);
		\draw[-,line width=1pt](0,0)--(-1,1);	\draw[-,line width=1pt](-1,0)--(-0.6,0.4);\draw[-,line width=1pt](-0.4,0.6)--(0,1);	
	\end{tikzpicture}=q^{-12}
	\begin{tikzpicture}[baseline=2pt,scale=0.6,color=\clr]
		\draw[-,line width=1pt](2,1)parabola bend (2.5,-0.5) (3,1);\end{tikzpicture},\quad
	\begin{tikzpicture}[baseline=2pt,scale=0.6,color=\clr]
		\draw[-,line width=1pt](7,0)arc (0:-180:0.5);
		\draw[-,line width=1pt](7,0)--(6.6,0.4);\draw[-,line width=1pt](6.4,0.6)--(6,1);	\draw[-,line width=1pt](6,0)--(7,1);	
	\end{tikzpicture}=q^{12}
	\begin{tikzpicture}[baseline=2pt,scale=0.6,color=\clr]
		\draw[-,line width=1pt](9,1)parabola bend (9.5,-0.5) (10,1);
	\end{tikzpicture},
\end{equation}
Their proofs are quite standard, which we omit.
The relations below are required by the definition of the category $\mathscr T$. 
\begin{equation}\label{checkTRI}
	\begin{tikzpicture}[baseline=29pt,scale=0.6,color=\clr]
		\draw[-,line width=1pt] (0.5,2)--(1.5,1);\draw [-,line width=1pt](0.5,1)--(0.9,1.4);\draw [-,line width=1pt](1.1,1.6)--(1.5,2);
	     \draw[-,line width=1pt] (0.5,2)--(1,2.5);\draw[-,line width=1pt] (1.5,2)--(1,2.5);\draw[-,line width=1pt] (1,2.5)--(1,3);
	\end{tikzpicture}
	=-q^{-6}\begin{tikzpicture}[baseline=29pt,scale=0.6,color=\clr]
		\draw[-,line width=1pt] (0.5,1)--(1,2);\draw[-,line width=1pt] (1.5,1)--(1,2);\draw[-,line width=1pt] (1,2)--(1,3);
	\end{tikzpicture},\quad
	\begin{tikzpicture}[baseline=29pt,scale=0.6,color=\clr]
		\draw[-,line width=1pt] (7.5,2)--(7.9,1.6);\draw[-,line width=1pt](8.1,1.4)--(8.5,1);\draw [-,line width=1pt](7.5,1)--(8.5,2);
		\draw[-,line width=1pt] (7.5,2)--(8,2.5);\draw[-,line width=1pt] (8.5,2)--(8,2.5);\draw[-,line width=1pt] (8,2.5)--(8,3);
	\end{tikzpicture}=-q^{6}
	\begin{tikzpicture}[baseline=29pt,scale=0.6,color=\clr]
	\draw[-,line width=1pt] (7.5,1)--(8,2);\draw[-,line width=1pt] (8.5,1)--(8,2);\draw[-,line width=1pt] (8,2)--(8,3);
	\end{tikzpicture},
\end{equation}
\begin{equation}\label{cycle2}
	\begin{tikzpicture}[baseline=30pt,scale = 0.6,color=\clr]
		\draw[-,line width=1pt] (1.5,0.8)to (1.5,1.5);
		\draw[-,line width=1pt] (1.5,1.5) to (1,2);
		\draw [-,line width=1pt](1.5,1.5) to (2,2);
		\draw[-,line width=1pt] (1,2) to (1.5,2.5);
		\draw[-,line width=1pt] (2,2) to (1.5,2.5);
		\draw[-,line width=1pt] (1.5,2.5) to (1.5,3);
	\end{tikzpicture}
	=([7]_q-1)
	\begin{tikzpicture}[baseline=30pt,scale = 0.6,color=\clr]
		\draw [-,line width=1pt](8.7,0.8) to (8.7,3);
	\end{tikzpicture}~.
\end{equation}
\begin{equation}\label{checkCIR}
	\begin{tikzpicture}[baseline=-2pt,scale=0.6,color=\clr]
		\draw[-,line width=1pt] (0.5,0.5) circle (10pt);\draw[-,line width=1pt](0.5,0.15)--(0.5,-0.7);
	\end{tikzpicture}=0,	\qquad
	\begin{tikzpicture}[baseline=-2pt,scale=0.6,color=\clr]
		\draw[-,line width=1pt] (4,-0.35) circle (10pt);\draw[-,line width=1pt](4,0)--(4,0.8);
	\end{tikzpicture}=0.
\end{equation}
Their proof constitutes part of the proof of Theorem  \ref{thm:functor}. 
Also it follows \eqref{up-p} that 
\beq
\upgamma=-q^{-3}\splits.
\eeq

The projection operators $P[0], P[{\lambda_1}]$, which map $V\ot V$ to the simple submodules $V_0$ and $V_{\lambda_1}$ respectively, can be diagrammatically represented by 
\begin{equation}
\frac{1}{\dim_qV}\begin{tikzpicture}[baseline=27pt,scale=0.7,color=\clr]		
	\draw[-,line width=1pt](0.5,2.1)parabola bend (1,1.6) (1.5,2.1);
	\draw[-,line width=1pt](0.5,0.9)parabola bend (1,1.4) (1.5,0.9);
\end{tikzpicture}= P[0]
	,\qquad 
\frac{1}{[7]_q-1}\begin{tikzpicture}[baseline=27pt,scale=0.7,color=\clr]		  
		\draw[-,line width=1pt] (1,1.3)--(1,1.7);
		\draw[-,line width=1pt] (1,1.7)--(0.5,2);
		\draw[-,line width=1pt] (1,1.7)--(1.5,2);
		\draw[-,line width=1pt] (1,1.3)--(0.5,1);
		\draw[-,line width=1pt] (1,1.3)--(1.5,1);    
	\end{tikzpicture} =P[{\lambda_1}]
.\end{equation}
The following result easily follows from the spectral decomposition \eqref{Rmatrix1} 
of $\check{R}$. 

\begin{lemma}\label{basicgraph}
The following relation holds.
\begin{equation}\label{checkD}
	q^{-1}\begin{tikzpicture}[baseline=25pt,scale=0.7,color=\clr]
		\draw[-,line width=1pt] (0.5,2)--(1.5,1);
		\draw [-,line width=1pt](0.5,1)--(0.9,1.4);
		\draw [-,line width=1pt](1.1,1.6)--(1.5,2);
	\end{tikzpicture}-q
	\begin{tikzpicture}[baseline=25pt,scale=0.7,color=\clr]
		\draw [-,line width=1pt](0.5,1)--(1.5,2);
		\draw[-,line width=1pt] (1.5,1)--(1.1,1.4);
		\draw [-,line width=1pt](0.9,1.6)--(0.5,2);
	\end{tikzpicture}=(q{-}q^{-1}) {\, }
	\begin{tikzpicture}[baseline=25pt,scale=0.7,color=\clr]		
		\draw [-,line width=1pt](0.7,1)--(0.7,2);
		\draw [-,line width=1pt](1.5,1)--(1.5,2);
	\end{tikzpicture}+(q^{-1}{-}q)(q^2{+}q^{-2})
	\begin{tikzpicture}[baseline=25pt,scale=0.7,color=\clr]		
		\draw[-,line width=1pt](0.5,2)parabola bend (1,1.6) (1.5,2);
		\draw[-,line width=1pt](0.5,1)parabola bend (1,1.4) (1.5,1);
	\end{tikzpicture}+(q{-}q^{-1})
	\begin{tikzpicture}[baseline=25pt,scale=0.7,color=\clr]		  
		\draw[-,line width=1pt] (1,1.3)--(1,1.7);
		\draw[-,line width=1pt] (1,1.7)--(0.5,2);
		\draw[-,line width=1pt] (1,1.7)--(1.5,2);
		\draw[-,line width=1pt] (1,1.3)--(0.5,1);
		\draw[-,line width=1pt] (1,1.3)--(1.5,1);    
	\end{tikzpicture}~.
\end{equation}	
\end{lemma}
\begin{proof}
It follows from the spectral decomposition \eqref{Rmatrix1} of the $R$-matrix that 
\[
\baln
q^{-1} \check{R} - q \check{R}^{-1} 
= &
(q - q^{-1}) P[2\lambda_{1}]- (q^{13}- q^{-13}) P[0] \\
&+ (q- q^{-1})  P[\lambda_{2}] + (q^7-q^{-7})P[\lambda_{1}].\\
\ealn
\]
As the idempotents $P[\mu]$,  for $\mu=2\lambda_1, \lambda_2, \lambda_1, 0$, are complete,  
we can rewrite the right hand side as 
\[
 (q - q^{-1}) \id_V\ot\id_V - (q^{13}- q^{-13} + q - q^{-1}) P[0] + (q^7-q^{-7} -q + q^{-1})P[\lambda_{1}].
\]
Note that $q^{13}- q^{-13} + q - q^{-1}=(q-q^{-1})(q^2 + q^{-2})\dim_q V$,  and $q^7-q^{-7} -q + q^{-1}=(q-q^{-1})([7]_q -1)$. Hence
\[
\baln
q^{-1} \check{R} - q \check{R}^{-1} 
=(q - q^{-1}) \left(\id_V\ot\id_V -(q^2+q^{-2})\dim_q V P[0] + ([7]_q-1)P[\lambda_{1}]\right).
\ealn
\]
Equation \eqref{checkD} is the diagrammatic representation of this relation.  
\end{proof} 

We have the following result. 
\begin{lemma}
The morphism 
$
\eta:=\begin{tikzpicture}[baseline=2pt,scale=0.5,color=\clr]
\draw[-,line width=1pt](8.3,1)--(8.8,0.5);
\draw[-,line width=1pt](8.3,0)--(8.8,0.5);
\draw[-,line width=1pt](8.8,0.5)--(9.2,0.5);
\draw[-,line width=1pt](9.2,0.5)--(9.7,1);
\draw[-,line width=1pt](9.2,0.5)--(9.7,0);
\end{tikzpicture}:  V\ot V\lra V\ot V
$
has the following spectral decomposition.
\beq
\phantom{XXX} \eta=([7]_q-1)P[0] -P[2\lambda_1] +(q^{2}+1+q^{-2})P[\lambda_2]-(q^{4}+1+q^{-4})P[\lambda_1].
\eeq
\end{lemma}

Using the lemma and the spectral decomposition of $\check{R}$, we easily  obtain the following result.  
\begin{lemma}\label{double}
The following relations hold.
\begin{equation}\label{cross}
	(q+q^{-1})\begin{tikzpicture}[baseline=25pt,scale=0.7,color=\clr]
		\draw[-,line width=1pt] (0.5,2)--(1.5,1);\draw [-,line width=1pt](0.5,1)--(0.9,1.4);\draw [-,line width=1pt](1.1,1.6)--(1.5,2);
	\end{tikzpicture}
	=q^3\begin{tikzpicture}[baseline=25pt,scale=0.7,color=\clr]
		\draw [-,line width=1pt](0.5,1)--(0.5,2);\draw [-,line width=1pt](1.5,1)--(1.5,2);\end{tikzpicture}+q^{-3}
	\begin{tikzpicture}[baseline=25pt,scale=0.7,color=\clr]
		\draw[-,line width=1pt](0.5,2)parabola bend (1,1.6) (1.5,2);
		\draw[-,line width=1pt](0.5,1)parabola bend (1,1.4) (1.5,1);  
	\end{tikzpicture}-q
	\begin{tikzpicture}[baseline=25pt,scale=0.7,color=\clr]
		\draw[-,line width=1pt](0.5,2)--(0.8,1.5);
		\draw[-,line width=1pt](0.5,1)--(0.8,1.5);
		\draw[-,line width=1pt](0.8,1.5)--(1.2,1.5);
		\draw[-,line width=1pt](1.2,1.5)--(1.5,2);
		\draw[-,line width=1pt](1.2,1.5)--(1.5,1);\end{tikzpicture}-q^{-1}
	\begin{tikzpicture}[baseline=25pt,scale=0.7,color=\clr]
		\draw[-,line width=1pt] (1,1.3)--(1,1.7);
		\draw[-,line width=1pt] (1,1.7)--(0.5,2);
		\draw[-,line width=1pt] (1,1.7)--(1.5,2);
		\draw[-,line width=1pt] (1,1.3)--(0.5,1);
		\draw[-,line width=1pt] (1,1.3)--(1.5,1);    
	\end{tikzpicture}~,
\end{equation}
\begin{equation}\label{crossprime}
	(q+q^{-1})\begin{tikzpicture}[baseline=25pt,scale=0.7,color=\clr]
		\draw [-,line width=1pt](0.5,1)--(1.5,2);
		\draw[-,line width=1pt] (1.5,1)--(1.1,1.4);
		\draw [-,line width=1pt](0.9,1.6)--(0.5,2);
	\end{tikzpicture}
	=q^{-3}\begin{tikzpicture}[baseline=25pt,scale=0.7,color=\clr]
		\draw [-,line width=1pt](0.5,1)--(0.5,2);\draw [-,line width=1pt](1.5,1)--(1.5,2);\end{tikzpicture}+q^{3}
	\begin{tikzpicture}[baseline=25pt,scale=0.7,color=\clr]
		\draw[-,line width=1pt](0.5,2)parabola bend (1,1.6) (1.5,2);
		\draw[-,line width=1pt](0.5,1)parabola bend (1,1.4) (1.5,1);  
	\end{tikzpicture}-q^{-1}
	\begin{tikzpicture}[baseline=25pt,scale=0.7,color=\clr]
		\draw[-,line width=1pt](0.5,2)--(0.8,1.5);
		\draw[-,line width=1pt](0.5,1)--(0.8,1.5);
		\draw[-,line width=1pt](0.8,1.5)--(1.2,1.5);
		\draw[-,line width=1pt](1.2,1.5)--(1.5,2);
		\draw[-,line width=1pt](1.2,1.5)--(1.5,1);\end{tikzpicture}-q
	\begin{tikzpicture}[baseline=25pt,scale=0.7,color=\clr]
		\draw[-,line width=1pt] (1,1.3)--(1,1.7);
		\draw[-,line width=1pt] (1,1.7)--(0.5,2);
		\draw[-,line width=1pt] (1,1.7)--(1.5,2);
		\draw[-,line width=1pt] (1,1.3)--(0.5,1);
		\draw[-,line width=1pt] (1,1.3)--(1.5,1);    
	\end{tikzpicture}~.
\end{equation}
\end{lemma}
\begin{proof}
Note that $P[2\lambda_1]$ and $P[\lambda_2]$ can be expressed in terns 
of $\id_V\ot\id_V$, $P[0]$ and $P[\lambda_1]$ and $\eta$. Using them in 
\eqref{Rmatrix1}, we obtain the first relation in the lemma. 
The second relation can be similarly obtained from the spectral decomposition of $\check{R}^{-1}$. 
\end{proof}

\begin{remark}
It is important to observe that given any morphism in ${\mathscr V}$ represented by a diagram $D$, we can replace all 
crossings using Lemma \ref{double}, to obtain a linear combination of trivalent graphs without crossings.  
\end{remark}

\subsection{Acyclic trivalent graphs}\label{de-crossing}

If a diagram in ${\mathscr V}$ has no crossing, it is a trivalent graph.  A graph without cycles is called acyclic. 
We denote by $\Acycl(m, n)$ the set of acyclic trivalent graphs in $\Hom_{\UG}(V^{\ot m}, V^{\ot n})$, 
and let $\Acycl=\cup_{m, n}\Acycl(m, n)$.

\begin{theorem}\label{lem:basis} Retain notation above. 
The space  $\Hom_{\UG}(V^{\ot k}, V^{\ot \ell})$ of morphisms is spanned by $\Acycl(k, \ell)$ for any $k, \ell\in\ZZ_{\geq0}$.
\end{theorem}

\begin{remark}\label{rmk:kuper-reduct}
This result can be extracted \cite{Ms}
from \cite[Theorem 6.10]{kuper:s} on $\UG$-spiders, which are  
 more complicated as they have two types of edges and more trivalent vertices.
As the proof in our case is quite simple, we briefly explain it below, 
as Theorem \ref{lem:basis}  plays a key role in this paper.  
\end{remark}

\begin{proof}[Comments on the proof of Theorem \ref{lem:basis}]
We have already seen from Lemma \ref{double} that trivalent graphs without crossings span 
all morphisms in ${\mathscr V}$. In view of Theorem \ref{thm:functor}, we only need to show that 
cycles in  ${\mathscr V}$ can all be reduced to linear combinations of acyclic trivalent graphs, 
in order to prove the theorem.

Now assume that a given a $(k, \ell)$ trivalent graph contains a cycle of $n$ vertices, 
which is a $(r, n-r)$-diagram for some $r$. 
We will denote the cycle by $\rectangle$ without specifying the $r$.  
We will prove the theorem by using induction on $n$ to show that we can always replace $\rectangle$ by 
a linear combination of acyclic trivalent graphs. 

The $n=2$ case follows from \eqref{cycle2}. For $n=3$, we have 
\beq
\begin{tikzpicture}[baseline=20pt,scale=0.45,color=\clr]
	\draw[-,line width=1pt] (0.5,1.5)--(1,2.5);\draw[-,line width=1pt] (1.5,1.5)--(1,2.5);\draw[-,line width=1pt] (1,2.5)--(1,3);
	\draw[-,line width=1pt](0.2,1)--(0.5,1.5);
	\draw[-,line width=1pt](0.5,1.5)--(1.5,1.5);
	\draw[-,line width=1pt](1.5,1.5)--(1.8,1);
\end{tikzpicture}
=-(q^4{+}1{+}q^{-4})
\begin{tikzpicture}[baseline=21pt,scale=0.5,color=\clr]
	\draw[-,line width=1pt] (1,2)--(1,2.6);\draw[-,line width=1pt] (1,2)--(0.5,1);\draw[-,line width=1pt] (1,2)--(1.5,1);
\end{tikzpicture} .
\eeq
To prove this, note that 
\begin{tikzpicture}[baseline=24pt,scale=0.5,color=\clr]
	\draw[-,line width=1pt] (0.5,1.5)--(1,2.5);\draw[-,line width=1pt] (1.5,1.5)--(1,2.5);\draw[-,line width=1pt] (1,2.5)--(1,3);
	\draw[-,line width=1pt](0.2,1)--(0.5,1.5);
	\draw[-,line width=1pt](0.5,1.5)--(1.5,1.5);
	\draw[-,line width=1pt](1.5,1.5)--(1.8,1);
\end{tikzpicture}
can be obtained by composition of graphs as follows 
\begin{tikzpicture}[baseline=28pt,scale=0.55,color=\clr]
	\draw[-,line width=1pt] (0.5,2)--(1,2.5);\draw[-,line width=1pt] (1.5,2)--(1,2.5);\draw[-,line width=1pt] (1,2.5)--(1,3);
	\draw[-,line width=1pt](0.5,2)--(0.8,1.5);
	\draw[-,line width=1pt](0.5,1)--(0.8,1.5);
	\draw[-,line width=1pt](0.8,1.5)--(1.2,1.5);
	\draw[-,line width=1pt](1.2,1.5)--(1.5,2);
	\draw[-,line width=1pt](1.2,1.5)--(1.5,1);
	\draw[densely dashed,red](0,2)--(2,2);\end{tikzpicture}. 
Then the relation \eqref{cross} in Lemma \ref{double} gives
\begin{equation*}
	(q+q^{-1})\begin{tikzpicture}[baseline=25pt,scale=0.55,color=\clr]
		\draw[-,line width=1pt] (0.5,2)--(1,2.5);\draw[-,line width=1pt] (1.5,2)--(1,2.5);\draw[-,line width=1pt] (1,2.5)--(1,3);
		\draw[-,line width=1pt] (0.5,2)--(1.5,1);\draw [-,line width=1pt](0.5,1)--(0.9,1.4);\draw [-,line width=1pt](1.1,1.6)--(1.5,2);
	\end{tikzpicture}
	=q^3\begin{tikzpicture}[baseline=25pt,scale=0.55,color=\clr]
		\draw[-,line width=1pt] (0.5,2)--(1,2.5);\draw[-,line width=1pt] (1.5,2)--(1,2.5);\draw[-,line width=1pt] (1,2.5)--(1,3);
		\draw [-,line width=1pt](0.5,1)--(0.5,2);\draw [-,line width=1pt](1.5,1)--(1.5,2);\end{tikzpicture}+q^{-3}
	\begin{tikzpicture}[baseline=25pt,scale=0.55,color=\clr]
		\draw[-,line width=1pt] (0.5,2)--(1,2.5);\draw[-,line width=1pt] (1.5,2)--(1,2.5);\draw[-,line width=1pt] (1,2.5)--(1,3);
		\draw[-,line width=1pt](0.5,2)parabola bend (1,1.6) (1.5,2);
		\draw[-,line width=1pt](0.5,1)parabola bend (1,1.4) (1.5,1);  
	\end{tikzpicture}-q
	\begin{tikzpicture}[baseline=20pt,scale=0.45,color=\clr]
	\draw[-,line width=1pt] (0.5,1.5)--(1,2.5);\draw[-,line width=1pt] (1.5,1.5)--(1,2.5);\draw[-,line width=1pt] (1,2.5)--(1,3);
	\draw[-,line width=1pt](0.2,1)--(0.5,1.5);
	\draw[-,line width=1pt](0.5,1.5)--(1.5,1.5);
	\draw[-,line width=1pt](1.5,1.5)--(1.8,1);\end{tikzpicture}
-q^{-1}
	\begin{tikzpicture}[baseline=25pt,scale=0.55,color=\clr]
		\draw[-,line width=1pt] (0.5,2)--(1,2.5);\draw[-,line width=1pt] (1.5,2)--(1,2.5);\draw[-,line width=1pt] (1,2.5)--(1,3);
		\draw[-,line width=1pt] (1,1.3)--(1,1.7);
		\draw[-,line width=1pt] (1,1.7)--(0.5,2);
		\draw[-,line width=1pt] (1,1.7)--(1.5,2);
		\draw[-,line width=1pt] (1,1.3)--(0.5,1);
		\draw[-,line width=1pt] (1,1.3)--(1.5,1);    
	\end{tikzpicture}.
\end{equation*}
Using \eqref{cycle2} and \eqref{checkTRI} in the relation, 
we easily obtain the result.

In the case $n=4$, we have 
\begin{tikzpicture}[baseline=18pt,scale=0.5,color=\clr]
	\draw[-,line width=1pt](-5,1)rectangle(-4,2);\draw[-,line width=1pt](-5,1)--(-5.3,0.7);\draw[-,line width=1pt](-4,1)--(-3.7,0.7);\draw[-,line width=1pt](-4,2)--(-3.7,2.3);\draw[-,line width=1pt](-5,2)--(-5.3,2.3);	
\end{tikzpicture} 
$=$
\begin{tikzpicture}[baseline=-3pt,scale=0.5,color=\clr]
	\draw[-,line width=1pt](8.3,1)--(8.8,0.5);
	\draw[-,line width=1pt](8.3,0)--(8.8,0.5);
	\draw[-,line width=1pt](8.8,0.5)--(9.2,0.5);
	\draw[-,line width=1pt](9.2,0.5)--(9.7,1);
	\draw[-,line width=1pt](9.2,0.5)--(9.7,0);%
	\draw[-,line width=1pt](8.3,0)--(8.8,-0.5);
	\draw[-,line width=1pt](8.3,-1)--(8.8,-0.5);
	\draw[-,line width=1pt](8.8,-0.5)--(9.2,-0.5);
	\draw[-,line width=1pt](9.2,-0.5)--(9.7,0);
	\draw[-,line width=1pt](9.2,-0.5)--(9.7,-1);
	\draw[densely dashed,red](8,0)--(10,0);
\end{tikzpicture} . 
By concatenating each diagram in \eqref{cross} with $\horiz$ at the bottom, we arrive at
\[
	\begin{aligned}
		\begin{tikzpicture}[baseline=25pt,scale=0.65,color=\clr]
			\draw[-,line width=1pt](-5,1)rectangle(-4,2);\draw[-,line width=1pt](-5,1)--(-5.3,0.7);\draw[-,line width=1pt](-4,1)--(-3.7,0.7);\draw[-,line width=1pt](-4,2)--(-3.7,2.3);\draw[-,line width=1pt](-5,2)--(-5.3,2.3);
		\end{tikzpicture}
&=q^{-4}\left([7]-1\right)
		\begin{tikzpicture}[baseline=25pt,scale=0.7,color=\clr]
			\draw[-,line width=1pt](4,2)parabola bend (4.4,1.6) (4.9,2);
			\draw[-,line width=1pt](4,1)parabola bend (4.4,1.4) (4.9,1); 
		\end{tikzpicture}+q^{2}
		\begin{tikzpicture}[baseline=25pt,scale=0.7,color=\clr]
			\draw[-,line width=1pt](6.3,2)--(6.6,1.5);\draw[-,line width=1pt](6.3,1)--(6.6,1.5);\draw[-,line width=1pt](6.6,1.5)--(7,1.5);\draw[-,line width=1pt](7,1.5)--(7.3,2);\draw[-,line width=1pt](7,1.5)--(7.3,1);
		\end{tikzpicture}-(1{+}q^{-2})
		\begin{tikzpicture}[baseline=25pt,scale=0.7,color=\clr]
			\draw[-,line width=1pt](-1.5,1)--(-1,1.5);
			\draw[-,line width=1pt](-1,1.5)--(-0.4,1.5);
			\draw[-,line width=1pt](-0.4,1.5)--(-0.9,2.2);
			\draw[-,line width=1pt](-0.4,1.5)--(-0.1,1);
			\draw[-,line width=1pt](-1,1.5)--(-0.7,1.8);
			\draw[-,line width=1pt](-0.6,1.9)--(-0.3,2.2);
		\end{tikzpicture}+\left(q^{2}{+}q^{-2}{+}q^{-6}\right)
		\begin{tikzpicture}[baseline=25pt,scale=0.7,color=\clr]
			\draw[-,line width=1pt] (12.5,1.3)--(12.5,1.7);
			\draw[-,line width=1pt] (12.5,1.7)--(12,2);
			\draw[-,line width=1pt] (12.5,1.7)--(13,2);
			\draw[-,line width=1pt] (12.5,1.3)--(12,1);
			\draw[-,line width=1pt] (12.5,1.3)--(13,1); 
		\end{tikzpicture} .
	\end{aligned}	
\]
Using the following relation 
\[
	\baln
	(1{+}q^{-2})\begin{tikzpicture}[baseline=25pt,scale=0.7,color=\clr]
		\draw[-,line width=1pt](-1.5,1)--(-1,1.5);
		\draw[-,line width=1pt](-1,1.5)--(-0.4,1.5);
		\draw[-,line width=1pt](-0.4,1.5)--(-0.9,2.2);
		\draw[-,line width=1pt](-0.4,1.5)--(-0.1,1);
		\draw[-,line width=1pt](-1,1.5)--(-0.7,1.8);
		\draw[-,line width=1pt](-0.6,1.9)--(-0.3,2.2);
	\end{tikzpicture}
	=(q^{-10}{+}q^{-8}{+}q^{-6})
	\begin{tikzpicture}[baseline=25pt,scale=0.7,color=\clr]
		\draw[-,line width=1pt](0.5,2)parabola bend (1,1.6) (1.5,2);
		\draw[-,line width=1pt](0.5,1)parabola bend (1,1.4) (1.5,1);  
	\end{tikzpicture}-q^{-2}
	\begin{tikzpicture}[baseline=25pt,scale=0.7,color=\clr]
		\draw[-,line width=1pt](0.5,2)--(0.8,1.5);
		\draw[-,line width=1pt](0.5,1)--(0.8,1.5);
		\draw[-,line width=1pt](0.8,1.5)--(1.2,1.5);
		\draw[-,line width=1pt](1.2,1.5)--(1.5,2);
		\draw[-,line width=1pt](1.2,1.5)--(1.5,1);\end{tikzpicture}+q^{-6}
	\begin{tikzpicture}[baseline=25pt,scale=0.7,color=\clr]
		\draw[-,line width=1pt] (1,1.3)--(1,1.7);
		\draw[-,line width=1pt] (1,1.7)--(0.5,2);
		\draw[-,line width=1pt] (1,1.7)--(1.5,2);
		\draw[-,line width=1pt] (1,1.3)--(0.5,1);
		\draw[-,line width=1pt] (1,1.3)--(1.5,1);    
	\end{tikzpicture}-(q^2{+}1{+}q^{-2})
	\begin{tikzpicture}[baseline=25pt,scale=0.7,color=\clr]
	\draw[-,line width=1pt] (16,1)--(16,2); \draw[-,line width=1pt] (16.5,1)--(16.5,2);
	\end{tikzpicture}
	\ealn
\]
to the right hand side, we obtain 
\begin{equation}\label{rec}
	\begin{aligned}
		\begin{tikzpicture}[baseline=25pt,scale=0.65,color=\clr]
			\draw[-,line width=1pt](-5,1)rectangle(-4,2);\draw[-,line width=1pt](-5,1)--(-5.3,0.7);\draw[-,line width=1pt](-4,1)--(-3.7,0.7);\draw[-,line width=1pt](-4,2)--(-3.7,2.3);\draw[-,line width=1pt](-5,2)--(-5.3,2.3);
		\end{tikzpicture}
=
		(q^2{+}1{+}q^{-2})\left[\begin{tikzpicture}[baseline=-13pt,scale=0.7,color=\clr]
			\draw[-,line width=1pt](0.6,-1)--(0.6,0);\draw[-,line width=1pt](1.6,-1)--(1.6,0);
		\end{tikzpicture}+
		\begin{tikzpicture}[baseline=-13pt,scale=0.7,color=\clr]
			\draw[-,line width=1pt](2.5,0.1)parabola bend (2.9,-0.4) (3.3,0.1);
			\draw[-,line width=1pt](2.5,-1.1)parabola bend (2.9,-0.7) (3.3,-1.1); 
		\end{tikzpicture}\right]+(q^2{+}q^{-2})
		\left[	\begin{tikzpicture}[baseline=-13pt,scale=0.7,color=\clr]
			\draw[-,line width=1pt](7.4,0.1)--(7.7,-0.4);\draw[-,line width=1pt](7.4,-0.9)--(7.7,-0.4);\draw[-,line width=1pt](7.7,-0.4)--(8.1,-0.4);\draw[-,line width=1pt](8.1,-0.4)--(8.4,0.1);\draw[-,line width=1pt](8.1,-0.4)--(8.4,-0.9);
		\end{tikzpicture}+
		\begin{tikzpicture}[baseline=-13pt,scale=0.7,color=\clr]
			\draw[-,line width=1pt] (10,-0.6)--(10,-0.2);
			\draw[-,line width=1pt] (10,-0.2)--(9.5,0.1);
			\draw[-,line width=1pt] (10,-0.2)--(10.5,0.1);
			\draw[-,line width=1pt] (10,-0.6)--(9.5,-0.9);
			\draw[-,line width=1pt] (10,-0.6)--(10.5,-0.9); 
		\end{tikzpicture}\right].
	\end{aligned}	
\end{equation}

In the general case with $n\geq5$, 
\begin{tikzpicture}[baseline=3pt,scale=0.5,color=\clr]
		\draw[-,line width=1pt](0,0)rectangle(2,1);	
		\node at(1,0.5){$R_n$};
	\end{tikzpicture} 
is an $n$-gon with $n$ legs. We can always put a  part of it on the top as $\horiz$. 
Remove this $\horiz$, and denote the resulting graph by 
	\begin{tikzpicture}[baseline=3pt,scale=0.5,color=\clr]
		\draw[-,line width=1pt](0,0)rectangle(2,1);	
		\node at(1,0.5){$C_{n-1}$};
	\end{tikzpicture}.
It is a connected acyclic trivalent graph with $n-2$ vertexes, i.e., it is a tree.

Now  \eqref{cross} leads to 
\begin{equation}\label{cross-1}
	q \begin{tikzpicture}[baseline=25pt,scale=0.7,color=\clr]
		\draw[-,line width=1pt](0.5,2)--(0.8,1.5);
		\draw[-,line width=1pt](0.5,1)--(0.8,1.5);
		\draw[-,line width=1pt](0.8,1.5)--(1.2,1.5);
		\draw[-,line width=1pt](1.2,1.5)--(1.5,2);
		\draw[-,line width=1pt](1.2,1.5)--(1.5,1);
\end{tikzpicture}
	=- (q+q^{-1})\begin{tikzpicture}[baseline=25pt,scale=0.7,color=\clr]
		\draw[-,line width=1pt] (0.5,2)--(1.5,1);\draw [-,line width=1pt](0.5,1)--(0.9,1.4);\draw [-,line width=1pt](1.1,1.6)--(1.5,2);
	\end{tikzpicture} \ +  q^3 \ 
\begin{tikzpicture}[baseline=25pt,scale=0.7,color=\clr]
		\draw [-,line width=1pt](0.5,1)--(0.5,2);\draw [-,line width=1pt](1.5,1)--(1.5,2);\end{tikzpicture}+q^{-3}
	\begin{tikzpicture}[baseline=25pt,scale=0.7,color=\clr]
		\draw[-,line width=1pt](0.5,2)parabola bend (1,1.6) (1.5,2);
		\draw[-,line width=1pt](0.5,1)parabola bend (1,1.4) (1.5,1);  
	\end{tikzpicture}-q^{-1}
	\begin{tikzpicture}[baseline=25pt,scale=0.7,color=\clr]
		\draw[-,line width=1pt] (1,1.3)--(1,1.7);
		\draw[-,line width=1pt] (1,1.7)--(0.5,2);
		\draw[-,line width=1pt] (1,1.7)--(1.5,2);
		\draw[-,line width=1pt] (1,1.3)--(0.5,1);
		\draw[-,line width=1pt] (1,1.3)--(1.5,1);    
	\end{tikzpicture}~.
\end{equation}
We replace the $\horiz$ on the top of 
\begin{tikzpicture}[baseline=3pt,scale=0.5,color=\clr]
		\draw[-,line width=1pt](0,0)rectangle(2,1);	
		\node at(1,0.5){$R_n$};
	\end{tikzpicture} 
by using \eqref{cross-1}. We obtain
\beq\label{eq:de-circle}
\begin{tikzpicture}[baseline=10pt,scale=0.6,color=\clr]
	\draw[-,line width=1pt](0,0.3)rectangle(2,1.3);	
	\node at(1,0.8){$R_{n}$};   
\end{tikzpicture}=q^2\begin{tikzpicture}[baseline=5pt,scale=0.6,color=\clr]
\draw[-,line width=1pt](0,0)rectangle(2,1);	
\node at(1,0.5){$C_{n-1}$};
\end{tikzpicture}-(1+q^{-2})\begin{tikzpicture}[baseline=13pt,scale=0.6,color=\clr]
		\draw[-,line width=1pt] (0.5,2)--(1.5,1);\draw [-,line width=1pt](0.5,1)--(0.9,1.4);\draw [-,line width=1pt](1.1,1.6)--(1.5,2);
		\draw[-,line width=1pt](0,0)rectangle(2,1);	
		\node at(1,0.5){$C_{n-1}$};
	\end{tikzpicture}
	+q^{-4}
	\begin{tikzpicture}[baseline=-5pt,scale=0.6,color=\clr]
		\draw[-,line width=1pt](0.2,0.8)parabola bend (1,0.3) (1.7,0.8);
		\draw[-,line width=1pt](0,-1)rectangle(2,0);	
		\node at(1,-0.5){$R_{n-2}$}; 
	\end{tikzpicture}-q^{-2}
	\begin{tikzpicture}[baseline=17pt,scale=0.6,color=\clr]
		\draw[-,line width=1pt] (1,1.3)--(1,1.7);
		\draw[-,line width=1pt] (1,1.7)--(0.5,2);
		\draw[-,line width=1pt] (1,1.7)--(1.5,2);
		\draw[-,line width=1pt](0,0.3)rectangle(2,1.3);	
		\node at(1,0.8){$R_{n-1}$};   
	\end{tikzpicture}~.
	\eeq

Let us illustrate this in the $n=6$ case.  We have 
\[\begin{tikzpicture}[baseline=3pt,scale=0.5,color=\clr]
		\draw[-,line width=1pt](0,0)rectangle(2,1);	
		\node at(1,0.5){$R_6$};
	\end{tikzpicture} 
=\begin{tikzpicture}[baseline=7pt,scale=0.7,color=\clr]
\draw[-,line width=1pt](0,0)--(0.7,0);	
\draw[-,line width=1pt](-0.4,0.5)--(0,0);
\draw[-,line width=1pt](0.7,0)--(1.1,0.5);
\draw[-,line width=1pt](-0.4,0.5)--(0,1);	
\draw[-,line width=1pt](1.1,0.5)--(0.7,1);
\draw[-,line width=1pt](0,1)--(0.7,1);	
\draw[-,line width=1pt](0,0)--(-0.2,-0.2);
\draw[-,line width=1pt](0.7,0)--(0.9,-0.2);
\draw[-,line width=1pt](-0.4,0.5)--(-0.6,0.5);
\draw[-,line width=1pt](1.1,0.5)--(1.3,0.5);
\draw[-,line width=1pt](0,1)--(-0.2,1.2);
\draw[-,line width=1pt](0.7,1)--(0.9,1.2);
\end{tikzpicture}~,\quad
\begin{tikzpicture}[baseline=3pt,scale=0.5,color=\clr]
		\draw[-,line width=1pt](0,0)rectangle(2,1);	
		\node at(1,0.5){$C_5$};
	\end{tikzpicture} 
=\begin{tikzpicture}[baseline=3pt,scale=0.7,color=\clr]
\draw[-,line width=1pt](0,0)--(0.7,0);	
\draw[-,line width=1pt](-0.4,0.5)--(0,0);
\draw[-,line width=1pt](0.7,0)--(1.1,0.5);
\draw[-,line width=1pt](-0.4,0.5)--(-0.2,0.7);	
\draw[-,line width=1pt](1.1,0.5)--(0.9,0.7);
\draw[-,line width=1pt](0,0)--(-0.2,-0.2);
\draw[-,line width=1pt](0.7,0)--(0.9,-0.2);
\draw[-,line width=1pt](-0.4,0.5)--(-0.6,0.5);
\draw[-,line width=1pt](1.1,0.5)--(1.3,0.5);
\end{tikzpicture},
\quad
\begin{tikzpicture}[baseline=13pt,scale=0.6,color=\clr]
		\draw[-,line width=1pt] (0.5,2)--(1.5,1);\draw [-,line width=1pt](0.5,1)--(0.9,1.4);\draw [-,line width=1pt](1.1,1.6)--(1.5,2);
		\draw[-,line width=1pt](0,0)rectangle(2,1);	
		\node at(1,0.5){$C_{5}$};
	\end{tikzpicture}
	=\begin{tikzpicture}[baseline=7pt,scale=0.7,color=\clr]
\draw[-,line width=1pt](0,0)--(0.7,0);	
\draw[-,line width=1pt](-0.4,0.5)--(0,0);
\draw[-,line width=1pt](0.7,0)--(1.1,0.5);
\draw[-,line width=1pt] (1.1,0.5)--(-0.3,1.3);
\draw [-,line width=1pt](-0.4,0.5)--(0.16,0.8);
\draw[-,line width=1pt](0.47,1)--(1.2,1.3);	
\draw[-,line width=1pt](0,0)--(-0.2,-0.2);
\draw[-,line width=1pt](0.7,0)--(0.9,-0.2);
\draw[-,line width=1pt](-0.4,0.5)--(-0.6,0.5);
\draw[-,line width=1pt](1.1,0.5)--(1.3,0.5);
\end{tikzpicture}~.\]
\[
	\begin{tikzpicture}[baseline=-5pt,scale=0.6,color=\clr]
		\draw[-,line width=1pt](0.2,0.8)parabola bend (1,0.3) (1.7,0.8);
		\draw[-,line width=1pt](0,-1)rectangle(2,0);	
		\node at(1,-0.5){$R_{4}$}; 
	\end{tikzpicture}=
	\begin{tikzpicture}[baseline=10pt,scale=0.7,color=\clr]
\draw[-,line width=1pt](-0.4,1.5)parabola bend (0.35,1.1) (1.1,1.5);
\draw[-,line width=1pt](-0.2,0.7)parabola bend (0.35,0.9) (0.9,0.7);	
\draw[-,line width=1pt](0,0)--(0.7,0);	
\draw[-,line width=1pt](-0.4,0.5)--(0,0);
\draw[-,line width=1pt](0.7,0)--(1.1,0.5);
\draw[-,line width=1pt](-0.4,0.5)--(-0.2,0.7);	
\draw[-,line width=1pt](1.1,0.5)--(0.9,0.7);
\draw[-,line width=1pt](0,0)--(-0.2,-0.2);
\draw[-,line width=1pt](0.7,0)--(0.9,-0.2);
\draw[-,line width=1pt](-0.4,0.5)--(-0.6,0.5);
\draw[-,line width=1pt](1.1,0.5)--(1.3,0.5);
\end{tikzpicture},
\quad
	\begin{tikzpicture}[baseline=17pt,scale=0.6,color=\clr]
		\draw[-,line width=1pt] (1,1.3)--(1,1.7);
		\draw[-,line width=1pt] (1,1.7)--(0.5,2);
		\draw[-,line width=1pt] (1,1.7)--(1.5,2);
		\draw[-,line width=1pt](0,0.3)rectangle(2,1.3);	
		\node at(1,0.8){$R_{5}$};   
	\end{tikzpicture}=
\begin{tikzpicture}[baseline=2pt,scale=0.4,color=\clr]
\draw[-,line width=1pt] (3,0) +(18:1cm) -- +(90:1cm) -- +(162:1cm) -- +(234:1cm) -- +(306:1cm)-- cycle;
\draw[-,line width=1pt] (2.05,0.3)--(1.65,0.3);			
\draw[-,line width=1pt] (3,0.95)--(3,1.5);	
\draw[-,line width=1pt] (3,1.5)--(3.8,2);
\draw[-,line width=1pt] (3,1.5)--(2.2,2);		
\draw[-,line width=1pt] (3.95,0.3)--(4.35,0.3);			
\draw[-,line width=1pt] (3.6,-0.8)--(4,-1.2);			
\draw[-,line width=1pt] (2.4,-0.8)--(2,-1.2);
\end{tikzpicture}~.
\]

Return to equation \eqref{eq:de-circle}. 
Observe that the first trivalent graph on the right hand side has no cycle, the third and fourth have cycles but have less than $n$ vertices.  Thus we only need to analyse the second diagram. It is of the form  
$
\begin{tikzpicture}[baseline=7pt,scale=0.55,color=\clr]
	\draw[-,line width=1pt] (0.3,2)--(1.3,1);\draw [-,line width=1pt](0.3,1)--(0.7,1.4);\draw [-,line width=1pt](0.9,1.6)--(1.3,2);
	\draw[-,line width=1pt](0,0)rectangle(1,1);	
	\node at(0.5,0.5){$D$};%
\draw[-,line width=1pt](0.5,0)--(0.8,-0.5);%
	\draw[-,line width=1pt](0.5,-1)--(0.8,-0.5);
	\draw[-,line width=1pt](0.8,-0.5)--(1.2,-0.5);
	\draw[-,line width=1pt](1.2,-0.5)--(1.5,0);
	\draw[-,line width=1pt](1.2,-0.5)--(1.5,-1);%
	\draw[-,line width=1pt](1.5,0)--(1.5,0.5);\draw[-,line width=1pt](1.5,0.5)--(1.3,1);\draw[-,line width=1pt](1.5,0.5)--(1.7,1);
\end{tikzpicture} , 
$
where $D$ is an acyclic trivalent graph with $n-5$ vertices. 
Using the graphical relation  
\begin{tikzpicture}[baseline=-2pt,scale=0.5,color=\clr]
\draw[-,line width=1pt](0.5,1)--(0.5,0);	
\draw[-,line width=1pt](0.5,0)--(0.8,-0.5);%
\draw[-,line width=1pt](0.5,-1)--(0.8,-0.5);
\draw[-,line width=1pt](0.8,-0.5)--(1.2,-0.5);
\draw[-,line width=1pt](1.2,-0.5)--(1.5,0);
\draw[-,line width=1pt](1.2,-0.5)--(1.5,-1);%
\draw[-,line width=1pt](1.5,0)--(1.5,0.5);
\draw[-,line width=1pt](1.5,0.5)--(1.3,1);
\draw[-,line width=1pt](1.5,0.5)--(1.7,1);
\end{tikzpicture} 
=
\begin{tikzpicture}[baseline=13pt,scale=0.5,color=\clr]
	\draw[-,line width=1pt] (1,1.3)--(1,1.7);
	\draw[-,line width=1pt] (1,1.7)--(0.5,2);
	\draw[-,line width=1pt] (1,1.7)--(1.5,2);
	\draw[-,line width=1pt] (1,1.3)--(0.5,1);
	\draw[-,line width=1pt] (1,1.3)--(1.5,1); %
	\draw[-,line width=1pt] (0,1)--(0,2);\draw[-,line width=1pt] (0,1)--(0.25,0.5); \draw[-,line width=1pt] (0.5,1)--(0.25,0.5);   
 \draw[-,line width=1pt] (0.25,0)--(0.25,0.5); \draw[-,line width=1pt] (1.5,1)--(1.5,0);
\end{tikzpicture} , 
we obtain 
\[
\begin{tikzpicture}[baseline=11pt,scale=0.6,color=\clr]
	\draw[-,line width=1pt] (0.5,2)--(1.5,1);\draw [-,line width=1pt](0.5,1)--(0.9,1.4);\draw [-,line width=1pt](1.1,1.6)--(1.5,2);
	\draw[-,line width=1pt](0,0)rectangle(2,1);	
	\node at(1,0.5){$C_{n-1}$};
\end{tikzpicture}
=
\begin{tikzpicture}[baseline=5pt,scale=0.5,color=\clr]
	\draw[-,line width=1pt] (0.3,2)--(1.3,1);\draw [-,line width=1pt](0.3,1)--(0.7,1.4);\draw [-,line width=1pt](0.9,1.6)--(1.3,2);
	\draw[-,line width=1pt](0,0)rectangle(1,1);	
	\node at(0.5,0.5){$D$};%
	\draw[-,line width=1pt] (1.8,0.3)--(1.8,0.7);%
	\draw[-,line width=1pt] (1.8,0.7)--(1.3,1);
	\draw[-,line width=1pt] (1.8,0.7)--(2.3,1);
	\draw[-,line width=1pt] (1.8,0.3)--(1.3,0);
	\draw[-,line width=1pt] (1.8,0.3)--(2.3,0); \draw[-,line width=1pt] (2.3,1)--(2.3,2); 	%
	\draw[-,line width=1pt] (0.5,0)--(0.9,-0.4);\draw[-,line width=1pt] (1.3,0)--(0.9,-0.4);
	\draw[-,line width=1pt] (0.9,-0.7)--(0.9,-0.4);
	\draw[-,line width=1pt] (2.3,0)--(2.3,-0.7);
\end{tikzpicture} .
\]

Replacing 
\begin{tikzpicture}[baseline=-13pt,scale=0.7,color=\clr]
			\draw[-,line width=1pt] (10,-0.6)--(10,-0.2);
			\draw[-,line width=1pt] (10,-0.2)--(9.5,0.1);
			\draw[-,line width=1pt] (10,-0.2)--(10.5,0.1);
			\draw[-,line width=1pt] (10,-0.6)--(9.5,-0.9);
			\draw[-,line width=1pt] (10,-0.6)--(10.5,-0.9); 
		\end{tikzpicture}
on the right hand side by using \eqref{cross} again, we obtain 
\beq\label{eq:de-crossing}
\begin{tikzpicture}[baseline=13pt,scale=0.6,color=\clr]
	\draw[-,line width=1pt] (0.5,2)--(1.5,1);\draw [-,line width=1pt](0.5,1)--(0.9,1.4);\draw [-,line width=1pt](1.1,1.6)--(1.5,2);
	\draw[-,line width=1pt](0,0)rectangle(2,1);	
	\node at(1,0.5){$C_{n-1}$};
\end{tikzpicture}
=
q^{-2}
\begin{tikzpicture}[baseline=5pt,scale=0.5,color=\clr]
	\draw[-,line width=1pt] (0.3,2)--(1.3,1);\draw [-,line width=1pt](0.3,1)--(0.7,1.4);\draw [-,line width=1pt](0.9,1.6)--(1.3,2);
	\draw[-,line width=1pt](0,0)rectangle(1,1);	
	\node at(0.5,0.5){$D$};%
	\draw[-,line width=1pt](1.3,1)parabola bend (1.8,0.6) (2.3,1);%
	\draw[-,line width=1pt](1.3,0)parabola bend (1.8,0.4) (2.3,0); %
	\draw[-,line width=1pt] (0.5,0)--(0.9,-0.4);\draw[-,line width=1pt] (1.3,0)--(0.9,-0.4);
	\draw[-,line width=1pt] (0.9,-0.7)--(0.9,-0.4);
	\draw[-,line width=1pt] (2.3,0)--(2.3,-0.6);\draw[-,line width=1pt] (2.3,1)--(2.3,2);
\end{tikzpicture}-(1{+}q^2)\begin{tikzpicture}[baseline=5pt,scale=0.5,color=\clr]
	\draw[-,line width=1pt] (0.3,2)--(1.3,1);\draw [-,line width=1pt](0.3,1)--(0.7,1.4);\draw [-,line width=1pt](0.9,1.6)--(1.3,2);
	\draw[-,line width=1pt](0,0)rectangle(1,1);	
	\node at(0.5,0.5){$D$};%
	\draw[-,line width=1pt] (1.3,1)--(2.3,0);%
	\draw [-,line width=1pt](1.3,0)--(1.7,0.4);\draw [-,line width=1pt](1.9,0.6)--(2.3,1);%
	\draw[-,line width=1pt] (0.5,0)--(0.9,-0.4);\draw[-,line width=1pt] (1.3,0)--(0.9,-0.4);
	\draw[-,line width=1pt] (0.9,-0.7)--(0.9,-0.4);
	\draw[-,line width=1pt] (2.3,0)--(2.3,-0.6);\draw[-,line width=1pt] (2.3,1)--(2.3,2);
\end{tikzpicture}+q^{4}
\begin{tikzpicture}[baseline=5pt,scale=0.5,color=\clr]
	\draw[-,line width=1pt] (0.3,2)--(1.3,1);\draw [-,line width=1pt](0.3,1)--(0.7,1.4);\draw [-,line width=1pt](0.9,1.6)--(1.3,2);
	\draw[-,line width=1pt](0,0)rectangle(1,1);	
	\node at(0.5,0.5){$D$};%
	\draw [-,line width=1pt](1.3,0)--(1.3,1);\draw [-,line width=1pt](1.7,0)--(1.7,2);
	\draw[-,line width=1pt] (0.5,0)--(0.9,-0.4);\draw[-,line width=1pt] (1.3,0)--(0.9,-0.4);
	\draw[-,line width=1pt] (0.9,-0.7)--(0.9,-0.4);
	\draw[-,line width=1pt] (1.7,0)--(1.7,-0.6);
\end{tikzpicture}-q^2
\begin{tikzpicture}[baseline=5pt,scale=0.5,color=\clr]
	\draw[-,line width=1pt] (0.3,2)--(1.3,1);\draw [-,line width=1pt](0.3,1)--(0.7,1.4);\draw [-,line width=1pt](0.9,1.6)--(1.3,2);
	\draw[-,line width=1pt](0,0)rectangle(1,1);	
	\node at(0.5,0.5){$D$};%
	\draw[-,line width=1pt](1.3,1)--(1.6,0.5);
	\draw[-,line width=1pt](1.3,0)--(1.6,0.5);
	\draw[-,line width=1pt](1.6,0.5)--(2,0.5);
	\draw[-,line width=1pt](2,0.5)--(2.3,1);
	\draw[-,line width=1pt](2,0.5)--(2.3,0);
	\draw[-,line width=1pt] (0.5,0)--(0.9,-0.4);\draw[-,line width=1pt] (1.3,0)--(0.9,-0.4);
	\draw[-,line width=1pt] (0.9,-0.7)--(0.9,-0.4);
	\draw[-,line width=1pt] (2.3,0)--(2.3,-0.6);\draw[-,line width=1pt] (2.3,1)--(2.3,2);
\end{tikzpicture} .
\eeq
We use \eqref{cross} to replace the crossings in the first and second diagrams to obtain 
\beq\label{eq:de-D}
\baln
(q{+}q^{-1})\begin{tikzpicture}[baseline=5pt,scale=0.5,color=\clr]
	\draw[-,line width=1pt] (0.3,2)--(1.3,1);\draw [-,line width=1pt](0.3,1)--(0.7,1.4);\draw [-,line width=1pt](0.9,1.6)--(1.3,2);
	\draw[-,line width=1pt](0,0)rectangle(1,1);	
	\node at(0.5,0.5){$D$};%
	\draw[-,line width=1pt](1.3,1)parabola bend (1.8,0.6) (2.3,1);%
	\draw[-,line width=1pt](1.3,0)parabola bend (1.8,0.4) (2.3,0); %
	\draw[-,line width=1pt] (0.5,0)--(0.9,-0.4);\draw[-,line width=1pt] (1.3,0)--(0.9,-0.4);
	\draw[-,line width=1pt] (0.9,-0.7)--(0.9,-0.4);\draw[-,line width=1pt] (2.3,0)--(2.3,-0.6);\draw[-,line width=1pt] (2.3,1)--(2.3,2);
\end{tikzpicture}&=q^3\begin{tikzpicture}[baseline=5pt,scale=0.5,color=\clr]
\draw [-,line width=1pt](0.3,1)--(0.3,2);\draw [-,line width=1pt](1.3,1)--(1.3,2);
\draw[-,line width=1pt](0,0)rectangle(1,1);	
\node at(0.5,0.5){$D$};%
\draw[-,line width=1pt](1.3,1)parabola bend (1.8,0.6) (2.3,1);%
\draw[-,line width=1pt](1.3,0)parabola bend (1.8,0.4) (2.3,0); %
\draw[-,line width=1pt] (0.5,0)--(0.9,-0.4);\draw[-,line width=1pt] (1.3,0)--(0.9,-0.4);
\draw[-,line width=1pt] (0.9,-0.7)--(0.9,-0.4);\draw[-,line width=1pt] (2.3,0)--(2.3,-0.6);\draw[-,line width=1pt] (2.3,1)--(2.3,2);
\end{tikzpicture}
+q^{-3}
\begin{tikzpicture}[baseline=5pt,scale=0.5,color=\clr]
\draw[-,line width=1pt](0.3,2)parabola bend (0.8,1.6) (1.3,2);
\draw[-,line width=1pt](0.3,1)parabola bend (0.8,1.4) (1.3,1); \draw[-,line width=1pt](0,0)rectangle(1,1);	
\node at(0.5,0.5){$D$};%
\draw[-,line width=1pt](1.3,1)parabola bend (1.8,0.6) (2.3,1);%
\draw[-,line width=1pt](1.3,0)parabola bend (1.8,0.4) (2.3,0); %
\draw[-,line width=1pt] (0.5,0)--(0.9,-0.4);\draw[-,line width=1pt] (1.3,0)--(0.9,-0.4);
\draw[-,line width=1pt] (0.9,-0.7)--(0.9,-0.4); \draw[-,line width=1pt] (2.3,0)--(2.3,-0.6);\draw[-,line width=1pt] (2.3,1)--(2.3,2);
\end{tikzpicture}+q
\begin{tikzpicture}[baseline=5pt,scale=0.5,color=\clr]
\draw[-,line width=1pt](0.3,2)--(0.6,1.5);
\draw[-,line width=1pt](0.3,1)--(0.6,1.5);
\draw[-,line width=1pt](0.6,1.5)--(1,1.5);
\draw[-,line width=1pt](1,1.5)--(1.3,2);
\draw[-,line width=1pt](1,1.5)--(1.3,1);\draw[-,line width=1pt](0,0)rectangle(1,1);	
\node at(0.5,0.5){$D$};%
\draw[-,line width=1pt](1.3,1)parabola bend (1.8,0.6) (2.3,1);%
\draw[-,line width=1pt](1.3,0)parabola bend (1.8,0.4) (2.3,0); %
\draw[-,line width=1pt] (0.5,0)--(0.9,-0.4);\draw[-,line width=1pt] (1.3,0)--(0.9,-0.4);
\draw[-,line width=1pt] (0.9,-0.7)--(0.9,-0.4);\draw[-,line width=1pt] (2.3,0)--(2.3,-0.6);\draw[-,line width=1pt] (2.3,1)--(2.3,2);\end{tikzpicture}+q^{-1}
\begin{tikzpicture}[baseline=5pt,scale=0.5,color=\clr]
\draw[-,line width=1pt] (0.8,1.3)--(0.8,1.7);
\draw[-,line width=1pt] (0.8,1.7)--(0.3,2);
\draw[-,line width=1pt] (0.8,1.7)--(1.3,2);
\draw[-,line width=1pt] (0.8,1.3)--(0.3,1);
\draw[-,line width=1pt] (0.8,1.3)--(1.3,1);  \draw[-,line width=1pt](0,0)rectangle(1,1);	
\node at(0.5,0.5){$D$};%
\draw[-,line width=1pt](1.3,1)parabola bend (1.8,0.6) (2.3,1);%
\draw[-,line width=1pt](1.3,0)parabola bend (1.8,0.4) (2.3,0); %
\draw[-,line width=1pt] (0.5,0)--(0.9,-0.4);\draw[-,line width=1pt] (1.3,0)--(0.9,-0.4);
\draw[-,line width=1pt] (0.9,-0.7)--(0.9,-0.4);  \draw[-,line width=1pt] (2.3,0)--(2.3,-0.6);\draw[-,line width=1pt] (2.3,1)--(2.3,2);
\end{tikzpicture}~, \\
(1{+}q^2)\begin{tikzpicture}[baseline=5pt,scale=0.5,color=\clr]
	\draw[-,line width=1pt] (0.3,2)--(1.3,1);\draw [-,line width=1pt](0.3,1)--(0.7,1.4);\draw [-,line width=1pt](0.9,1.6)--(1.3,2);
	\draw[-,line width=1pt](0,0)rectangle(1,1);	
	\node at(0.5,0.5){$D$};%
	\draw[-,line width=1pt] (1.3,1)--(2.3,0);%
	\draw [-,line width=1pt](1.3,0)--(1.7,0.4);\draw [-,line width=1pt](1.9,0.6)--(2.3,1);%
	\draw[-,line width=1pt] (0.5,0)--(0.9,-0.4);\draw[-,line width=1pt] (1.3,0)--(0.9,-0.4);
	\draw[-,line width=1pt] (0.9,-0.7)--(0.9,-0.4);
	\draw[-,line width=1pt] (2.3,0)--(2.3,-0.6);\draw[-,line width=1pt] (2.3,1)--(2.3,2);
\end{tikzpicture}&=
(q{+}q^{-1})\begin{tikzpicture}[baseline=-2pt,scale=0.5,color=\clr]
	\draw[-,line width=1pt](0,0)rectangle(1,1);	\draw[-,line width=1pt] (0.5,1)--(0.5,1.5);
	\node at(0.5,0.5){$D$};%
	\draw[-,line width=1pt] (0.5,0)--(0.9,-0.4);\draw[-,line width=1pt] (1.3,0)--(0.9,-0.4);
	\draw[-,line width=1pt] (0.9,-0.6)--(0.9,-0.4);
	\draw[-,line width=1pt] (1.3,0)--(1.3,1.5);
	\draw[-,line width=1pt] (-0.2,-0.5)--(0.9,-1.6);\draw [-,line width=1pt](-0.1,-1.6)--(0.3,-1.2);\draw [-,line width=1pt](0.5,-1)--(0.9,-0.6);%
	\draw[-,line width=1pt] (-0.2,-0.5)--(-0.2,1.5);
\end{tikzpicture}=q^3\begin{tikzpicture}[baseline=-2pt,scale=0.5,color=\clr]
	\draw[-,line width=1pt](0,0)rectangle(1,1);	\draw[-,line width=1pt] (0.5,1)--(0.5,1.5);
	\node at(0.5,0.5){$D$};%
	\draw[-,line width=1pt] (0.5,0)--(0.9,-0.4);\draw[-,line width=1pt] (1.3,0)--(0.9,-0.4);
	\draw[-,line width=1pt] (0.9,-0.6)--(0.9,-0.4);
	\draw[-,line width=1pt] (1.3,0)--(1.3,1.5);
	\draw[-,line width=1pt] (-0.2,-0.5)--(-0.2,-1.6);\draw[-,line width=1pt] (0.9,-0.5)--(0.9,-1.6);%
	\draw[-,line width=1pt] (-0.2,-0.5)--(-0.2,1.5);
\end{tikzpicture}+q^{-3}
\begin{tikzpicture}[baseline=-2pt,scale=0.5,color=\clr]
	\draw[-,line width=1pt](0,0)rectangle(1,1);	\draw[-,line width=1pt] (0.5,1)--(0.5,1.5);
	\node at(0.5,0.5){$D$};%
	\draw[-,line width=1pt] (0.5,0)--(0.9,-0.4);\draw[-,line width=1pt] (1.3,0)--(0.9,-0.4);
	\draw[-,line width=1pt] (0.9,-0.6)--(0.9,-0.4);
	\draw[-,line width=1pt] (1.3,0)--(1.3,1.5);
	\draw[-,line width=1pt] (-0.2,-0.5)parabola bend (0.35,-0.8)(0.9,-0.6);\draw[-,line width=1pt] (-0.2,-1.6)parabola bend (0.35,-1)(0.9,-1.6);%
	\draw[-,line width=1pt] (-0.2,-0.5)--(-0.2,1.5);
\end{tikzpicture}
+q\begin{tikzpicture}[baseline=-2pt,scale=0.5,color=\clr]
	\draw[-,line width=1pt](0,0)rectangle(1,1);	\draw[-,line width=1pt] (0.5,1)--(0.5,1.5);
	\node at(0.5,0.5){$D$};%
	\draw[-,line width=1pt] (0.5,0)--(0.9,-0.4);\draw[-,line width=1pt] (1.3,0)--(0.9,-0.4);
	\draw[-,line width=1pt] (0.9,-0.6)--(0.9,-0.4);
	\draw[-,line width=1pt] (1.3,0)--(1.3,1.5);
	\draw[-,line width=1pt](-0.2,-0.6)--(0.1,-0.9);%
	\draw[-,line width=1pt](-0.2,-1.6)--(0.1,-0.9);
	\draw[-,line width=1pt](0.1,-0.9)--(0.5,-0.9);
	\draw[-,line width=1pt](0.5,-0.9)--(0.9,-0.6);
	\draw[-,line width=1pt](0.5,-0.9)--(0.9,-1.6);%
	\draw[-,line width=1pt] (-0.2,-0.6)--(-0.2,1.5);
\end{tikzpicture}
+q^{-1}\begin{tikzpicture}[baseline=-2pt,scale=0.5,color=\clr]
	\draw[-,line width=1pt](0,0)rectangle(1,1);	\draw[-,line width=1pt] (0.5,1)--(0.5,1.5);
	\node at(0.5,0.5){$D$};%
	\draw[-,line width=1pt] (0.5,0)--(0.9,-0.4);\draw[-,line width=1pt] (1.3,0)--(0.9,-0.4);
	\draw[-,line width=1pt] (0.9,-0.6)--(0.9,-0.4);
	\draw[-,line width=1pt] (1.3,0)--(1.3,1.5);
	\draw[-,line width=1pt] (0.3,-1.3)--(0.3,-0.9);%
	\draw[-,line width=1pt] (0.3,-0.9)--(-0.2,-0.6);
	\draw[-,line width=1pt] (0.3,-0.9)--(0.9,-0.6);
	\draw[-,line width=1pt] (0.3,-1.3)--(-0.2,-1.6);
	\draw[-,line width=1pt] (0.3,-1.3)--(0.9,-1.6);%
	\draw[-,line width=1pt] (-0.2,-0.6)--(-0.2,1.5);
\end{tikzpicture}  , 
\ealn
\eeq
where the trivalent graphs on the right hand sides are all acyclic. 
The third and fourth diagrams on the right hand side of \eqref{eq:de-crossing} are 
\begin{tikzpicture}[baseline=5pt,scale=0.5,color=\clr]
	\draw[-,line width=1pt] (0.5,2)--(1.5,1);\draw [-,line width=1pt](0.5,1)--(0.9,1.4);\draw [-,line width=1pt](1.1,1.6)--(1.5,2);
	\draw[-,line width=1pt](0,0)rectangle(2,1);	
	\node at(1,0.5){$C_{n-3}$};\draw [-,line width=1pt](2.3,-0.6)--(2.3,2);
\end{tikzpicture}
and
\begin{tikzpicture}[baseline=5pt,scale=0.5,color=\clr]
\draw[-,line width=1pt] (0.5,2)--(1.5,1);\draw [-,line width=1pt](0.5,1)--(0.9,1.4);\draw [-,line width=1pt](1.1,1.6)--(1.5,2);
\draw[-,line width=1pt](0,0)rectangle(2,1);	
\node at(1,0.5){$C_{n-2}$};
\draw[-,line width=1pt](2,0.5)--(2.3,0.5);
	\draw[-,line width=1pt](2.3,0.5)--(2.6,1);
	\draw[-,line width=1pt](2.3,0.5)--(2.6,0);
	\draw[-,line width=1pt] (2.6,0)--(2.6,-0.6);\draw[-,line width=1pt] (2.6,1)--(2.6,2);
\end{tikzpicture}, 
 which have $n-4$ and $n-3$ vertices respectively. 

This completes the induction step, proving the theorem. 
\end{proof}

\section{Invariant theory of ${\rm U}_q(G_2)$}\label{sect:inv-theo}

We now develop a non-commutative algebraic invariant theory for $\UG$ within the context 
of braided symmetric algebras. The main result obtained is 
a quantum first fundamental theorem of invariant theory, that is, Theorem \ref{FFT}, 
 which describes the generators of the subalgebra of $\UG$-invariants in $\cA_m(V)$
for all $m$.

Hereafter we denote $\cA_m(V)$ by $\cA_m$. 
\subsection{The strategy}\label{sect:strategy}
Let us start by describing our strategy for proving Theorem \ref{FFT}, 
and determining commutation relations among the generators of invariants. 
Define
\beq
\cA_m^\UG:=\{ f\in \cA_m\mid x\cdot f=\epsilon(x) f,  \forall x\in\UG\},
\eeq
where $\epsilon$ is the co-unit of $\UG$.  This is the subspace of $\UG$-invariants in $\cA_m$.   
It is well known (and easy to show) that $\cA_m^\UG$ is a subalgebra of $\cA_m$. We shall refer to it as the subalgebra of invariants of $\cA_m$.  

Note that $\cA_m^\UG$ is non-commutative for all $m\ge 3$. 
Our aim is to understand the structure of $\cA_m^\UG$ as an associative algebra. 
Let us now outline our strategy for doing this. 

Let $\CT_m=T(V)^{\ot m}$ for any $m$, which is $\ZZ_+^m$-graded with 
\[ 
\left(\CT_{m}\right)_{\mathbf{d}}=T(V)_{d_{1}} \otimes \cdots \otimes T(V)_{d_{m}},  \quad {\bf d}=(d_1, d_2, \dots, d_m)\in\ZZ_+^m.
\]
Set $|{\bf d}|=\sum_{i=1}^m d_i$.  Then we have the linear isomorphism 
\[
j_{\bf d}: \left(\CT_{m}\right)_{\mathbf{d}}\stackrel{\sim}\lra V^{\ot |{\bf d}|}.
\] 

Recall that $\cA_m$ is also $\ZZ_+^m$-graded, with $\left(\cA_{m}\right)_{\mathbf{d}}:=\cA_m(V)_{\mathbf{d}}=S_{q}(V)_{d_{1}} \otimes \cdots \otimes S_{q}(V)_{d_{m}}$. 
The restriction of the map $\tau^{\ot m}: \CT_m \lra \cA_m$ to the homogeneous components leads to the surjections 
\beq\label{eq:tau-d}
\tau_{\bf d}:= \tau_{d_1}\ot \cdots \ot \tau_{d_m}: (\CT_m)_{\bf d}\lra 
(\cA_m)_{\bf d}, \quad \tau_{\bf d}\circ j_{\bf d}^{-1}: V^{\ot |{\bf d}|}\lra (\cA_m)_{\bf d}.
\eeq

Let $A=\cA_{m}^{\UG}$ and $T=\CT_{m}^{\UG}$, which are $\ZZ_+^m$-graded algebras. 
Since both $\CT_{m}$ and $\mathcal{A}_{m}$ are semi-simple as $\UG$-modules, 
and $\tau^{\otimes m}$ is a $\UG$-map, we have $A=\tau^{\otimes m}(T)$.  
Denote by $A_{\bf d}$ and $T_{\bf d}$
the degree ${\bf d}\in\ZZ_+^m$ homogeneous components of $A$ and $B$ respectively.  
Clearly
	$
	A_{\mathbf{d}}=\tau_{\bf d}\left(T_{\mathbf{d}}\right).
	$
Thus we have the surjection 
\beq\label{eq:key-map}
	\hat\tau_{\bf d}:   \operatorname{Hom}_{\UG}\left(\mathbb{C}(q), V^{\otimes |{\bf d}|}\right) {\longrightarrow} A_{\mathbf{d}},
	\eeq
defined by 
the composition of the following maps 
\[
\operatorname{Hom}_{\UG}\left(\mathbb{C}(q), V^{\otimes |{\bf d}|}\right)\stackrel{\sim}\lra \left(V^{\otimes |{\bf d}|}\right)^{\UG} \stackrel{j_{\bf d}^{-1}}\lra T_{\bf d} \stackrel{\tau_{\bf d}}\lra A_{\bf d},
\]
where the first map is given by $f \mapsto f(1)$.

The map $\hat\tau_{\bf d}$ enables us to describe $A_{\bf d}$, quite explicitly, using the diagrammatic method for $ \Hom_{\UG}\left(\mathbb{C}(q), V^{\otimes |{\bf d}|}\right)$ 
 developed in Section \ref{diagdep}. 
Such a diagrammatic method for studying invariants was used in \cite{lzz:ft}. 

\begin{remark}
Note that we only use $\tau^{\ot m}: \CT_m \lra \cA_m$ as a linear map, even though it is a
$\ZZ_+^m$-graded algebra homomorphism when $\CT_m$ is endowed with the braided multiplication (obtained by iterating \eqref{eq:mu-AB}). 
\end{remark}
\subsection{Elementary invariants}\label{sect:elmt}
Let us first consider the invariants arising from the graphs in Figure \ref{Figure-elementary}. For easy reference, we call these graphs, and also the corresponding invariants {\em elementary}.
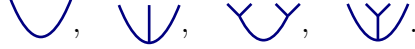
\begin{figure}[h]
\begin{gather*}
\begin{tikzpicture}[baseline=2pt,scale=1,color=\clr]\draw[-,line width=1pt,color=\clr](0,0.5)parabola bend (0.4,0) (0.8,0.5);\end{tikzpicture}, \quad 
\begin{tikzpicture}[baseline=-10pt,scale=1,color=\clr]
	\draw[-,line width=1pt](-0.1,0)parabola bend (0.3,-0.5) (0.7,0);\draw[-,line width=1pt](0.3,-0.5)--(0.3,0);	\end{tikzpicture},\quad
\begin{tikzpicture}[baseline=-15pt,scale=0.8,color=\clr]
	\draw[-,line width=1pt](0.3,-0.4)--(0.1,-0.2);\draw[-,line width=1pt](0.3,-0.4)--(0.5,-0.2);
	\draw[-,line width=1pt](1.1,-0.4)--(0.9,-0.2);\draw[-,line width=1pt](1.1,-0.4)--(1.3,-0.2);
	\draw[-,line width=1pt](0.3,-0.4)parabola bend (0.7,-0.8) (1.1,-0.4);
\end{tikzpicture},\quad 
\begin{tikzpicture}[baseline=-15pt,scale=0.8,color=\clr]
	\draw[-,line width=1pt](0.8,-0.8)--(0.8,-0.4);\draw[-,line width=1pt](0.8,-0.4)--(0.6,-0.2);\draw[-,line width=1pt](0.8,-0.4)--(1,-0.2);
	\draw[-,line width=1pt](0.3,-0.2)parabola bend (0.8,-0.8) (1.3,-0.2);
\end{tikzpicture}.
\end{gather*}
\caption{Elementary $(0, r)$ trivalent  graphs}
\label{Figure-elementary}
\end{figure}
It will be shown in Section \ref{sect:FFT} that the elementary invariants arising from the first two diagrams of Figure \ref{Figure-elementary} generate $\cA_m^{\UG}$. 

We will need the following notation. For $i\in [1, m]$, we define 
${\bf d}_i\in \ZZ_+^m$ by
\[
\baln
{\bf d}_i:=& (0, \dots, 0, \underbrace{1}_{i}, 0, \dots, 0).
\ealn
\]
For any $i_1\le   i_2 \le \dots \le  i_r$ with $i_s\in [1,  m]$ for all $s$, we let 
\beq\label{eq:bf-d}
{\bf d}_{i_1, i_2, \dots, i_r}:={\bf d}_{i_1}+ {\bf d}_{i_2}+\dots + {\bf d}_{i_r} \in \ZZ_+^m.
\eeq
Note that $|{\bf d}_{i_1, i_2, \dots, i_r}|=r$. 

\subsubsection{The elementary invariants}{\ }
\smallskip

\noindent{\em a). The invariants $\Phi^{(i, j)}$}. 
Recall the invariant $\Phi\in S_q(V)_2$ given by \eqref{eq:quadr-inv}. There are associated invariants $\Phi^{(i, j)}$ in $\cA_m$, 
which are defined as follows. 

For any $i\le j\in [1, m]$, we consider $A_{{\bf d}_{i,j}}=\hat{\tau}_{{\bf d}_{i,j}}\left( \Hom_{\UG}\left(\mathbb{C}(q), V^{\ot 2}\right)\right)$. Recall that 
$c_0=\begin{tikzpicture}[baseline=2pt,scale=1,color=\clr]
\draw[-,line width=1pt,color=\clr](0,0.5)parabola bend (0.4,0) (0.8,0.5);\end{tikzpicture}(1)$ spans
$\Hom_{\UG}\left(\mathbb{C}(q), V^{\ot 2}\right)$. 
Using the basis $c_0$ (see \eqref{basis0}) for $(V\ot V)^{\UG}$, we obtain the elements
$\Phi^{(i, j)}:=\hat\tau_{{\bf d}_{i, j}}(c_0)$  
in $A_{{\bf d}_{i, j}}$ for $i\le j$. We have 
\beq
\baln
\Phi^{(i, j)}=&q^4X_{i1}X_{j,-1}{+}q^{-6}X_{i,-1}X_{j1}{+}q^3X_{i2}X_{j, -2}
+q^{-5}X_{i, -2}X_{j2}\\
&+X_{i3}X_{j, -3}+q^{-2}X_{i, -3}X_{j3}+X_{i0}X_{j0}.
\ealn
\eeq
Note in particular that 
$
\Phi^{(i, i)} = \underbrace{1\ot\dots\ot 1}_{i-1}\ot \Phi\ot \underbrace{1\ot\dots\ot 1}_{m-i}.
$
For $i<j$, we will also define $\Phi^{(j, i)}$ by exchanging $i$ and $j$ in the formula. 
Then using the spectral decomposition \eqref{Rmatrix1} for the $R$-matrix $\check{R}$, we obtain 
	\begin{equation}\label{r0}
		\Phi^{(j,i)}=q^{-12}\Phi^{(i,j)}, \quad i<j.
	\end{equation}	

\smallskip
\noindent{\em b).The invariants $\Psi^{(i, j, k)}$}.
Recall from Section \ref{diagdep} the $\UG$-morphisms $\check{C}$ and $\upgamma$, which can be shown to satisfy the relation $(\upgamma\ot\id)\circ \check{C}= (\id\ot \upgamma)\circ \check{C}$. Therefore,  
$
(\upgamma\otimes \id)(c_0)=(\id\otimes\upgamma)(c_0), 
$
which spans  $\left(V\ot V\ot V\right)^{\UG}$. This is equal to $-q^{-3}$ times the following element
\[
\Psi:=\begin{tikzpicture}[baseline=-9pt,-,color=\clr]
		\draw[-,line width=1pt](0,0)parabola bend (0.3,-0.4) (0.6,0);\draw[-,line width=1pt](0.3,-0.4)--(0.3,0);
	\end{tikzpicture}(1).
\]

We write the basis elements $b_a$ of $V_{\lambda_1}\subset V\ot V$ as $b_a=\sum_{c, d} \gamma_{a r s} v_r\ot v_s$ for all $a$, 
where the scalars $\gamma_{a r s}$ can be read off \eqref{basis0}.  Let $c_{r s t} = -q^3 \sum_c \varphi_{r c} \gamma_{c s t }$ for all $r$, where $\varphi_{r c}$ are the scalars in the formula \eqref{eq:c0} for $c_0$. Then  
\beq
\Psi= \sum_{r,s,t}c_{rst}v_r\otimes v_s\otimes v_t.
\eeq

For $m\geq 3$, and any $i<j<k$ in $[1, m]$, we have $|{\bf d}_{i, j, k}|=3$. 
Then $(\cA_m)_{{\bf d}_{i, j, k}}^{\UG}\simeq (V\ot V\ot V)^{\UG}$, and  
\beq\label{eq:Psi}
\Psi^{(i, j, k)}:=\hat\tau_{{\bf d}_{i, j, k}}(\Psi)= \sum_{r,s,t}c_{rst}X_{i r}X_{js}X_{k t}. 
\eeq

Now we define $\Psi^{(i, j, k)}$ for all $i, j, k\in [1, m]$ by \eqref{eq:Psi}. 
Then
	$\Psi^{(i,i,j)}=\Psi^{(i,j,j)}=0$;
	$\Psi^{(j,i,k)}=-q^{-6}\Psi^{(i,j,k)}$ if $j>i$;
	and $\Psi^{(i,k,j)}=-q^{-6}\Psi^{(i,j,k)}$ if $k>j$.

We have the following explicit formula:
\begin{align*}
&\Psi^{(i, j, k)}= - q^7X_{i 1} \left(X_{j 0} X_{k, -1}-q^{-6}X_{j, -1} X_{k 0}
					-q^{-2}\left(X_{j, -3} X_{k, -2} -q^{-3}X_{j, -2} X_{k, -3}\right)\right)\\
&\phantom{X}{-}q^{-3}X_{i, -1} \left(X_{j1} X_{k0}-q^{-6}X_{j0} X_{k1}
					-q^{-3}(1+q^2)\left(X_{j2} X_{k3}-q^{-3}X_{j3} X_{k 2}\right)\right)\\
&\phantom{X}{+}q^4X_{i 2} \left(q^{-2}X_{j, -2} X_{k0}-X_{j 0} X_{k, -2}
					-(1+q^2)\left(X_{j 3} X_{k, -1}-q^{-5}X_{j, -1} X_{k 3}\right)\right)\\
&\phantom{X}{-}q^{-3}X_{i, -2} \left(q^{-3}\left(X_0 X_{2}-q^{2}X_{2} X_{0}\right)
					-X_{1} X_{-3}+q^{-5}X_{-3} X_{1}\right)\\
&\phantom{X}{-}qX_{i 3} \left((1+q^2)	\left(X_{j2} X_{k, -1}-q^{-5}X_{j, -1} X_{k 2}\right)
					-\left(X_{j 0} X_{k, -3}- q^{-2}X_{j, -3} X_{k 0}\right)\right)\\
&\phantom{X}{-}X_{i, -3} \left(X_{j 1} X_{k, -2}-q^{-5}X_{j, -2} X_{k 1}
					-q^{-3}\left(q^{2}X_{j 3} X_{k 0}-X_{j 0} X_{k 3}\right)\right)\\
&\phantom{X}{-}q^{2}X_{i 0} \left(q^{-5}X_{j, -1} X_{k 1}{-}q^{-1}X_{j 1} X_{k, -1}
					+q^{-3}(1-q^{2})X_{j 0}  X_{k 0}
	          {+}X_{j 2} X_{k, -2}-q^{-6}X_{j, -2} X_{k 2}\right). 
\end{align*}

\smallskip
\noindent{\em c).The invariants $\Upsilon^{(i, j, k, \ell)}$}.
Let 
$
\Upsilon:=\begin{tikzpicture}[baseline=-15pt,scale=0.7,color=\clr]
		\draw[-,line width=1pt](0.3,-0.4)--(0.1,-0.2);\draw[-,line width=1pt](0.3,-0.4)--(0.5,-0.2);
		\draw[-,line width=1pt](1.1,-0.4)--(0.9,-0.2);\draw[-,line width=1pt](1.1,-0.4)--(1.3,-0.2);
		\draw[-,line width=1pt](0.3,-0.4)parabola bend (0.7,-0.8) (1.1,-0.4);
	\end{tikzpicture}(1)
$.
Clearly $\Upsilon=(\upgamma\otimes \upgamma)(c_0)$. 
We write  
$\Upsilon
	=\sum_{a, b, c, d}c_{a b c d}v_a\otimes v_b\otimes v_c\otimes v_d$ with 
$c_{a b c d}= q^6 \sum_{f, h} \varphi_{f h}\gamma_{f a b} \gamma_{h c d}$. This leads to the following elements in 
$(\cA_{m})^{\UG}$ for $m\ge 4$: 
\beq
\Upsilon^{(i, j, k, \ell)} :=\tau_{{\bf d}_{i, j, k, \ell} }(\Upsilon)
	=\sum_{a, b, c, d}c_{a b c d}X_{i a} X_{j b} X_{k c} X_{\ell d}, \quad i<j<k<\ell\in[1, m].
\eeq

We will see later that there is no need to consider $\Upsilon^{(i, j, k, \ell)}$ 
with different orders of the superscripts.

Also, the invariant arising from the morphism
$
\upgaprimeform:   \mathbb{C}(q)\longrightarrow V^{\otimes 4}
$
can be expressed in terms of $\Phi^{(i, j)}$ and $\Upsilon^{(i, j, k, \ell)}$, 
see Section \ref{sect:up-prime}.

\subsubsection{Commutation relations}\label{sect:CRs}
In this section, we determine the commutation relations of the elementary invariants 
$\Phi^{(i, j)}$, 
$\Psi^{(i, j, k)}$ and $\Upsilon^{(i,j,k,\ell)}$ with the generators $X_{r a}$ of $\cA_{m}(V)$.

We consider $\Phi^{(i,j)}$ first.   
By pulling up the bottom right (left) leg of the diagram 
\begin{tikzpicture}[baseline=20pt,scale=0.55,color=\clr]		  
	\draw[-,line width=1pt] (1,1.3)--(1,1.7);
	\draw[-,line width=1pt] (1,1.7)--(0.5,2);
	\draw[-,line width=1pt] (1,1.7)--(1.5,2);
	\draw[-,line width=1pt] (1,1.3)--(0.5,1);
	\draw[-,line width=1pt] (1,1.3)--(1.5,1);    
\end{tikzpicture} 
to the top right (left) using $\cups$, 
we obtain
$T_{+}:=
\begin{tikzpicture}[baseline=0pt,scale=0.3,color=\clr]
	\draw[-,line width=1pt](-1,1)--(0,0); \draw[-,line width=1pt](1,1)--(0,0); \draw[-,line width=1pt](0,0)--(0,-1);\draw[-,line width=1pt](-0.5,0.5)--(0,1);
\end{tikzpicture}$ 
and 
$T_{-}:=
\begin{tikzpicture}[baseline=0pt,scale=0.3,color=\clr]
\draw[-,line width=1pt](-1,1)--(0,0); \draw[-,line width=1pt](1,1)--(0,0); \draw[-,line width=1pt](0,0)--(0,-1);\draw[-,line width=1pt](0.5,0.5)--(0,1);
\end{tikzpicture}$.
Write
\[
\baln
&T_{+}: V\longrightarrow V^{\otimes 3}, 
	v_a\mapsto \phi^{+}_{a},\\
& T_{-}: V\longrightarrow V^{\otimes 3}, 
	v_a\mapsto \phi^{-}_{a}, 
\ealn
\]
by introducing the $\phi^{\pm}_{a}\in V^{\ot 3}$, 
and denote $(\phi^{\pm}_a)^{(i,j,k)} =\hat \tau_{{\bf d}_{i, j, k}}(\phi_a^{(\pm)})$. 
	
	\begin{lemma}\label{ppp}
		The elements $\Phi^{(i,j)}$ satisfy the following relations  for all $a\in[-3, 3]$:
		\begin{equation*}
			\begin{aligned}
				X_{ia}\Phi^{(i,i)}&=\Phi^{(i,i)}X_{ia}=0,\quad\mbox{for~ all~} i,\\
				X_{ka}\Phi^{(i,i)}&=\Phi^{(i,i)}X_{ka}, \quad\mbox{for~} k\neq i,\\
				X_{ka}\Phi^{(i,j)}&=\Phi^{(i,j)}X_{ka}, \quad\mbox{for~} i\neq j, k< i,j, \mbox{~or~}k> i,j,\\
				X_{ia}\Phi^{(i,j)}&=q^{2}\Phi^{(i,j)}X_{ia}+\frac{1-q^{14}}{\dim_q(V)}\Phi^{(i,i)}X_{ja}, \quad\mbox{for~} i< j, \\
				X_{ja}\Phi^{(i,j)}&=q^{-2}\Phi^{(i,j)}X_{ja}-q^{-2}\frac{1-q^{14}}{\dim_q(V)}\Phi^{(j,j)}X_{ia}, \quad\mbox{for~} i< j,\\
		\end{aligned}\end{equation*}	
\begin{equation*}	
\baln	
	q^{-1}X_{ka}\Phi^{(i,j)}=&q\Phi^{(i,j)}X_{ka}+(q{-}q^{-1})
			\left[\Phi^{(k,j)}X_{ia}{-}(q^2+q^{-2})\Phi^{(i,k)}X_{ja}+(\phi^{+}_{a})^{(i,k,j)}\right],\\
& \hspace{80mm}\mbox{for~} i{<}k{<}j.
\ealn
\end{equation*}
	\end{lemma}
\begin{proof}
The proofs of all the relations but the last one are conceptually 
similar to that of \cite[Lemma 3.8]{lzz:ft}.  
Thus we will only give the proof of the fourth relation as an example. 
To prove the relation,  it suffices to treat the case $i=1, j=2$.

Observe that $X_{1a}\Phi^{(1, 2)}=(\tau_2\ot \tau_1)(w_a)$ 
and $\Phi^{(1, 2)}X_{1 a}=(\tau_2\ot \tau_1)(w'_a)$, where  
\[
w_a=		
\begin{tikzpicture}[baseline=14pt,scale=0.5,color=\clr]
\draw [-,line width=1pt](6.3,0.5)--(6.3,2);
			\draw[-,line width=1pt](6.5,2)parabola bend (7,0.5) (7.5,2);
\end{tikzpicture}  (v_a), \quad 
w'_a=
\begin{tikzpicture}[baseline=14pt,scale=0.5,color=\clr]
			\draw[-,line width=1pt](2.85,1)parabola (2.5,2); \draw[-,line width=1pt](3.05,1)parabola (3.5,2);\draw [-,line width=1pt](2.95,0.5)--(2.95,2);
		\end{tikzpicture} (v_a). 
\]
Using the spectral decomposition \eqref{Rmatrix1} of $\check{R}$, we obtain  
\[
\baln
w_1 - q^2 w'_a =& \frac{1-q^{14}}{\dim_q V}\, 
\begin{tikzpicture}[baseline=14pt,scale=0.5,color=\clr]
			\draw[-,line width=1pt](1.6,2)parabola bend (2.1,0.5) (2.6,2); \draw[-,line width=1pt](2.7,0.5)--(2.7,2);
\end{tikzpicture}\,(v_a) + \text{terms in ${\rm ker}(\tau_2\ot \tau_1)$}\\
=& \frac{1-q^{14}}{\dim_q V} c_0\ot v_a + \text{terms in ${\rm ker}(\tau_2\ot \tau_1)$}. 
\ealn
\]
Hence 
\[
\baln
X_{1a}\Phi^{(1, 2)}- q^2\Phi^{(1, 2)}X_{1 a}=&(\tau_2\ot \tau_1)(w_a- q^2 w'_a)
=\frac{1-q^{14}}{\dim_q V} \Phi^{(1, 1)} X_{2 a},
\ealn
\]
proving the relation.  

To prove the last relation, 
we only need to consider the case $i=1, k=2$ and $j=3$. 
By pulling the bottom right node of each diagram in \eqref{checkD} to the top right using $\cups$, we obtain 
\begin{equation*}\label{Phir}
		q^{-1}\begin{tikzpicture}[baseline=22pt,scale=0.7,color=\clr]
			\draw[-,line width=1pt](0.5,2)parabola bend (1,1) (1.5,2);\draw [-,line width=1pt](1,2)--(1,1.2);\draw [-,line width=1pt](1,0.9)--(1,0.5);
		\end{tikzpicture}-q
		\begin{tikzpicture}[baseline=22pt,scale=0.7,color=\clr]
			\draw[-,line width=1pt](2.85,1)parabola (2.5,2); \draw[-,line width=1pt](3.05,1)parabola (3.5,2);\draw [-,line width=1pt](2.95,0.5)--(2.95,2);
		\end{tikzpicture}=(q{-}q^{-1})
		\begin{tikzpicture}[baseline=18pt,scale=0.6,color=\clr]
			\draw [-,line width=1pt](6.3,0.5)--(6.3,2);\draw[-,line width=1pt](6.5,2)parabola bend (7,0.5) (7.5,2);
		\end{tikzpicture}\ {+}(q^{-1}{-}q)(q^2{+}q^{-2})
		\begin{tikzpicture}[baseline=18pt,scale=0.6,color=\clr]
			\draw[-,line width=1pt](11.6,2)parabola bend (12.1,0.5) (12.6,2); \draw[-,line width=1pt](12.8,0.5)--(12.8,2);
		\end{tikzpicture}\ {+}(q{-}q^{{-}1})
		\begin{tikzpicture}[baseline=18pt,scale=0.6,color=\clr]
			\draw[-,line width=1pt](15.5,2)--(16.1,1.4); \draw[-,line width=1pt](16.1,1.4)--(16.7,2);\draw[-,line width=1pt](16.1,1.4)--(16.1,0.5);\draw[-,line width=1pt](15.8,1.7)--(16.1,2);
		\end{tikzpicture}.
\end{equation*}
Applying both sides to $v_a$, we obtain an equality of two elements in $V^{\ot 3}$. The images of these elements under the map $\tau_1\ot \tau_1\ot \tau_1$ are equal, leading to the last relation of the lemma.  
\end{proof}

\begin{remark}
		The last relation can also be equivalently expressed as
		\begin{equation*}		
			q^{-1}\Phi^{(i,j)}X_{ka}{-}qX_{ka}\Phi^{(i,j)}{=}(q{-}q^{-1})
			\left[ \Phi^{(i,k)}X_{ja}{-}(q^2+q^{-2})\Phi^{(k,j)}X_{ia}{+}(\phi^{-}_{a})^{(i,k,j)}\right].
		\end{equation*}
	\end{remark}

Next we consider $\Psi^{(i,j,k)}$. 

\begin{lemma}\label{xPsi}
Assume that $m\geq 3$. The elements $\Psi^{(i,j,k)}\in \cA_m^{\UG}$, with $i<j<k$ in $[1, m]$, 
satisfy the following relations for all $a\in[-3, 3]$:
\begin{align}
X_{ra}\Psi^{(i,j,k)}=&\Psi^{(i,j,k)}X_{ra}, \quad \text{if $r< i$ or $r>k$}, \nonumber \\
X_{i a} \Psi^{(i,j,k)}=&q^{2}\Psi^{(i,j,k)} X_{i a}+\frac{1-q^{14}}{\mathrm{dim}_qV}\Phi^{(i,i)} \left({\upgamma(v_a)}\right)^{(j,k)},\nonumber\\
X_{k a}\Psi^{(i,j,k)}=&q^{-2}\Psi^{(i,j,k)}X_{k a}-q^{-2}\frac{1-q^{14}}{\mathrm{dim}_qV}\Phi^{(k,k)} \left({\upgamma(v_a)}\right)^{(i,j)},\nonumber\end{align}
\begin{align}
X_{j a} \Psi^{(i,j,k)}=&\Psi^{(i,j,k)}{X_{ja}}{+}{(q^{-2}{-}1)}
 \left[\Phi^{(i,j)}{\left({\upgamma(v_a)}\right)^{(j,k)}}\right.\label{xPsi-1}\\
&{-}\left.q^2\Phi^{(j,k)}{\left({\upgamma(v_a)}\right)^{(i,j)}}{-}\frac{1{-}q^{14}}{\dim_q{V}}\Phi^{(j,j)}{\left({\upgamma(v_a)}\right)^{(i,k)}}\right], 
\nonumber\end{align}
\begin{align}
&\left(X_{ra}\Psi^{(i,j,k)}{+}q^2X_{ja}\Psi^{(i,r,k)}\right)-\left(q^2\Psi^{(i,j,k)}X_{ra}{+}\Psi^{(i,r,k)}X_{ja}\right) \label{xPsi-2}\\
=&
(q^2{-}1)\left[\Psi^{(r,j,k)}X_{ia}{-}\Psi^{(i,r,j)}X_{ka}\right.\nonumber\\
&-\left.(q^2{+}q^{-2})\left(\Phi^{(i,r)}(\upgamma(v_{a}))^{(j,k)}{-}\Phi^{(j,k)}(\upgamma(v_{a}))^{(i,r)}\right)\right],  \  \mbox{if~}  i{<}r{<}j{<}k.\nonumber
\end{align}
\end{lemma}

	\begin{proof} 
The first relation  are clear.
The second and third relations can be proved in the same way as for Lemma \ref{ppp}. 
They can also be obtained from the fourth and fifth relations in Lemma \ref{ppp}. 
Thus we will focus on \eqref{xPsi-1} and \eqref{xPsi-2}.

Let us manipulate \eqref{checkD} in two different ways. 
One way is to pull the bottom right note to the top right using $\cups$, then compose both sides of the resulting relation with $\ids\  \ \ids \ \, \splits$. 
The other is to pull the bottom left note to the top left using $\cups$, then compose both sides of the resulting relation with $\splits\   \,  \ids\  \ \ids$. Then we obtain two equations in $\End_{\UG}(V, V^{\ot 4})$. Take the difference of these equations, we obtain 
		\begin{equation}\label{3-skein}
\baln
	&		q^{-1}\left(
			\begin{tikzpicture}[baseline=0pt,scale = 0.56,color=\clr]
				\draw[-,line width=1pt](0.6,0.9)parabola bend (1.4,0) (2,0.9);\draw[-,line width=1pt](1.4,0)--(1.4,0.9);\draw[-,line width=1pt](1.1,0.3)--(1.1,0.9);\draw[-,line width=1pt](1.1,0)--(1.1,-0.5);
			\end{tikzpicture}{-}
			\begin{tikzpicture}[baseline=0pt,scale = 0.56,color=\clr]
				\draw[-,line width=1pt](2.6,0.9)parabola bend (3.2,0) (3.5,0.3);	\draw[-,line width=1pt](3.6,-0.5)--(3.6,0.9);	
				\draw[-,line width=1pt](3.7,0.4)--(4,0.9);
				\draw[-,line width=1pt](3.2,0)--(3.2,0.9);
			\end{tikzpicture}\right){-}q\left(
			\begin{tikzpicture}[baseline=0pt,scale = 0.56,color=\clr]
				\draw[-,line width=1pt](6,0.45)--(5.7,0.9);\draw[-,line width=1pt](6.1,-0.5)--(6.1,0.9); \draw[-,line width=1pt](6.2,0.4) parabolabend (6.5,0) (7.1,0.9);\draw[-,line width=1pt](6.5,0)--(6.5,0.9);
			\end{tikzpicture}{-}
			\begin{tikzpicture}[baseline=0pt,scale = 0.56,color=\clr]
				\draw[-,line width=1pt](7.7,0.9)parabola bend (8.3,0) (9.1,0.9);\draw[-,line width=1pt](8.3,0)--(8.3,0.9);\draw[-,line width=1pt](8.6,0.2)--(8.6,0.9);\draw[-,line width=1pt](8.6,0)--(8.6,-0.5);
			\end{tikzpicture}\right)\\
&{=}(q{-}q^{-1})
			\left[\left(\begin{tikzpicture}[baseline=3pt,scale = 0.62,color=\clr]	
				\draw[-,line width=1pt](13.1,-0.2)--(13.1,0.9);\draw[-,line width=1pt](13.3,0.9)parabola bend (13.7,-0.2) (14.1,0.9);	\draw[-,line width=1pt](13.7,-0.2)--(13.7,0.9);
			\end{tikzpicture}{-}
			\begin{tikzpicture}[baseline=3pt,scale = 0.62,color=\clr]
				\draw[-,line width=1pt](14.5,0.9)parabola bend (14.9,-0.2) (15.3,0.9);	\draw[-,line width=1pt](14.9,-0.2)--(14.9,0.9);\draw[-,line width=1pt](15.5,-0.2)--(15.5,0.9);\end{tikzpicture}\right){-}(q^2{+}q^{{-}2})
			\left(\begin{tikzpicture}[baseline=3pt,scale = 0.62,color=\clr]
				\draw[-,line width=1pt](19,0.9)parabola bend (19.3,0) (19.6,0.9);
				\draw[-,line width=1pt](19.7,0.9)--(19.9,0.3);\draw[-,line width=1pt](20.2,0.9)--(19.9,0.3);\draw[-,line width=1pt](19.9,0.3)--(19.9,-0.2);
			\end{tikzpicture}{-}
			\begin{tikzpicture}[baseline=3pt,scale = 0.62,color=\clr]
				\draw[-,line width=1pt](20.6,0.9)--(20.8,0.3);\draw[-,line width=1pt](21,0.9)--(20.8,0.3);\draw[-,line width=1pt](20.8,0.3)--(20.8,-0.2);
				\draw[-,line width=1pt](21.2,0.9)parabola bend (21.5,0) (21.8,0.9);\end{tikzpicture}\right)\right].
\ealn
		\end{equation}
Denote
		$$
		D_{0}^{l}:=\begin{tikzpicture}[baseline=0pt,scale=0.5,color=\clr]	
			\draw[-,line width=1pt](1,-0.2)--(1,0.9);\draw[-,line width=1pt](1.2,0.9)parabola bend (1.6,-0.2) (2,0.9);	
			\draw[-,line width=1pt](1.6,-0.2)--(1.6,0.9);
		\end{tikzpicture},\quad
		D_{0}^{r}:=	\begin{tikzpicture}[baseline=0pt,scale=0.5,color=\clr]
			\draw[-,line width=1pt](1,0.9)parabola bend (1.4,-0.2) (1.8,0.9);	\draw[-,line width=1pt](1.4,-0.2)--(1.4,0.9);\draw[-,line width=1pt](2,-0.2)--(2,0.9);
		\end{tikzpicture},
\quad		D_{+}^{2}:=\begin{tikzpicture}[baseline=0pt,scale=0.5,color=\clr]
			\draw[-,line width=1pt](1,0.45)parabola (0.7,0.9);\draw[-,line width=1pt](1.1,-0.5)--(1.1,0.9); \draw[-,line width=1pt](1.2,0.4) parabolabend (1.5,0) (2.1,0.9);\draw[-,line width=1pt](1.5,0)--(1.5,0.9);
		\end{tikzpicture}~,
		\quad
		D_{-}^{2}:=\begin{tikzpicture}[baseline=0pt,scale=0.5,color=\clr]
			\draw[-,line width=1pt](0.6,0.9)parabola bend (1.4,0) (2,0.9);\draw[-,line width=1pt](1.4,0)--(1.4,0.9);\draw[-,line width=1pt](1.1,0.3)--(1.1,0.9);\draw[-,line width=1pt](1.1,0)--(1.1,-0.5);
		\end{tikzpicture},\quad
		D_{+}^{3}:=\begin{tikzpicture}[baseline=0pt,scale=0.5,color=\clr]
			\draw[-,line width=1pt](1,0.9)parabola bend (1.6,0) (1.9,0.3);	\draw[-,line width=1pt](2,-0.5)--(2,0.9);	\draw[-,line width=1pt](2.1,0.4) parabola (2.4,0.9);
			\draw[-,line width=1pt](1.6,0)--(1.6,0.9);
		\end{tikzpicture},
		\quad 
		D_{-}^{3}:=\begin{tikzpicture}[baseline=0pt,scale=0.5,color=\clr]
			\draw[-,line width=1pt](1,0.9)parabola bend (1.6,0) (2.4,0.9);\draw[-,line width=1pt](1.6,0)--(1.6,0.9);\draw[-,line width=1pt](1.9,0.2)--(1.9,0.9);\draw[-,line width=1pt](1.9,0)--(1.9,-0.5);
		\end{tikzpicture}
		$$
Applying these maps to a basis element $v_a$ of $V$, we obtain 
\begin{align*}
&D_{0}^{l}(v_a)=\sum_{r,s,t}c_{rst}v_a\otimes v_r\otimes v_s\otimes v_t,\hspace{9mm}D_{0}^{r}(v_a)=\sum_{r,s,t}c_{rst}v_r\otimes v_s\otimes v_t\otimes v_a,\\
&D_{-}^{2}(v_a){=}\sum_{r,s,t}c_{rst}\check{R}(v_a\otimes v_r)\otimes v_s\otimes v_t,
\quad D_{+}^{3}(v_a){=}\sum_{r,s,t}c_{rst}v_r\otimes v_s\otimes \check{R}(v_t\otimes v_a)\\
&D_{+}^{2}(v_a){=}\sum_{r,s,t}c_{rst}(\id\otimes\check{R}\otimes\id)\left(v_r\otimes v_s\otimes \check{R}(v_t\otimes v_a)\right),\\
&D_{-}^{3}(v_a){=}\sum_{r,s,t}c_{rst}(\id\otimes\check{R}\otimes\id)\left(\check{R}(v_a\otimes v_r)\otimes v_s\otimes v_t\right).
\end{align*}

Let us now prove the relation \eqref{xPsi-2}. It suffices to consider the case $m=2$ and $i=1, j=3$ and $k=4$. We have   
		$$
		\begin{array}{ll}
			\tau_{(1,1,1,1)}\circ D_{-}^2\left(v_a\right)=X_{2 a} \Psi^{(1,3,4)}, & \tau_{(1,1,1,1)} \circ D_{+}^2\left(v_a\right)=\Psi^{(1,3,4)} X_{2 a}, \\
			\tau_{(1,1,1,1)} \circ D_{-}^3\left(v_a\right)=X_{3 a} \Psi^{(1,2,4)}, & \tau_{(1,1,1,1)} \circ D_{+}^3\left(v_a\right)=\Psi^{(1,2,4)} X_{3 a} ,\\
			\tau_{(1,1,1,1)} \circ D_{0}^l\left(v_a\right)=X_{1 a} \Psi^{(2,3,4)}, & \tau_{(1,1,1,1)} \circ D_{0}^r\left(v_a\right)=\Psi^{(1,2,3)} X_{4 a}.
		\end{array}
		$$
		Evaluating both sides of \eqref{3-skein} at $v_a$,
		then applying the mapping $\tau_{(1,1,1,1)}$ to the 
resulting elements of $V^{\otimes 4}$, 
		we obtain the last relation for 
		$i=1,m=2,j=3$ and $k=4$.

		Let us now prove the relation \eqref{xPsi-1}.
		We only need to prove the case $i=1, j=2$ and $k=3$.
We observe that the map $\tau_{(1, 2, 1)}: V\ot  (V\ot V) \ot V\lra S_q(V)_1\otimes S_q(V)_2\otimes S_q(V)_1$ coincides with 
$\id_V\ot(P[2\lambda_1]+P[0])\ot\id_V$.
		Denote  by $E_2$ the map
		\begin{tikzpicture}[baseline=-5pt,scale=0.9,color=\clr]
			\draw[-,line width=1pt](0.05,0.3)parabola bend (0.2,0.05) (0.35,0.3);
			\draw[-,line width=1pt](-0.1,0.3)--(0.2,-0.2);\draw[-,line width=1pt](0.5,0.3)--(0.2,-0.2);\draw[-,line width=1pt](0.2,-0.2)--(0.2,-0.5);
		\end{tikzpicture}:$V\longrightarrow V^{\otimes 4}$, 
		and let $F^i_{\pm}:=\left({\id}\otimes\left(P[2\lambda_1]+P[0]\right) \otimes {\id}\right) D^i_{\pm}$, for $i=2,3$.
	We have 
	\begin{equation*}
	\begin{aligned}
			F_+^2 =&\left({\id}\otimes P[2\lambda_1]  \otimes {\id}\right) D^2_++\left({\id}\otimes P[0] \otimes {\id}\right) D^2_+\\
			=&q^2\left({\id}\otimes P[2\lambda_1] \otimes {\id}\right) D^3_+-q^{-6}\dim_qVE_2\\
			=&q^2\left({\id}\otimes\left(P[2\lambda_1]+P[0]\right) \otimes {\id}\right) D^3_++\frac{q^8-q^{-6}}{\dim_qV}E_2\\
			=&q^2F^3_++\frac{q^8-q^{-6}}{\dim_qV}E_2,
\end{aligned}
\end{equation*}
and we can similarly prove that 
\begin{equation*}
	\begin{aligned}
		F^3_- =&q^2F_-^2+\frac{q^8-q^{-6}}{\dim_qV}E_2.\end{aligned}
\end{equation*}
			Furthermore, the relation \eqref{3-skein} gives 
		\begin{equation}\label{middle}
			F_-^2-F^3_+=(q^{-2}-1)\left(\mathrm{id}\otimes\left(P[2\lambda_1]+P[0]\right) \otimes \mathrm{id}\right)\left(
			\begin{tikzpicture}[baseline=2pt,scale=0.6,color=\clr]
				\draw[-,line width=1pt](19,0.9)parabola bend (19.3,0) (19.6,0.9);
				\draw[-,line width=1pt](19.7,0.9)--(19.9,0.3);\draw[-,line width=1pt](20.2,0.9)--(19.9,0.3);\draw[-,line width=1pt](19.9,0.3)--(19.9,-0.2);\end{tikzpicture}
			-\begin{tikzpicture}[baseline=2pt,scale=0.6,color=\clr]		
				\draw[-,line width=1pt](20.6,0.9)--(20.8,0.3);\draw[-,line width=1pt](21,0.9)--(20.8,0.3);\draw[-,line width=1pt](20.8,0.3)--(20.8,-0.2);
				\draw[-,line width=1pt](21.2,0.9)parabola bend (21.5,0) (21.8,0.9);
			\end{tikzpicture}\right).
		\end{equation}
Note that
		$$
		\tau_{(1,2,1)} \circ F_-^2\left(v_a\right)=X_{2 a} \Psi^{(1,2,3)}, \quad \tau_{(1,2,1)} \circ F^3_+\left(v_a\right)=\Psi^{(1,2,3)} X_{2 a}.
		$$
Therefore, evaluating both sides of \eqref{middle} on $v_a$, and  applying $\tau_{(1,2,1)}$ to the resulting elements of $V^{\ot 3}$,  
		we obtain the following relation for $i=1,j=2$ and $k=3$:
\begin{equation}\label{xppx}
X_{j a} \Psi^{(i,j,k)}-\Psi^{(i,j,k)}X_{ja}=(q^{-2}{-}1)\left[\Phi^{(i,j)}\left({\upgamma(v_a)}\right)^{(j,k)}{-}\left({\upgamma(v_a)}\right)^{(i,j)}\Phi^{(j,k)}\right].
\end{equation}	
Now $\left({\upgamma(v_a)}\right)^{(i,j)}=\sum_{b,c}\gamma_{a b c}X_{ib}X_{jc}$.
By using Lemma \ref{ppp}, we obtain
\begin{align*}
\left({\upgamma(v_a)}\right)^{(i,j)}\Phi^{(j,k)}
=&\sum_{b,c}\gamma_{a b c}X_{ib}X_{jc}\Phi^{(j,k)}\\
=&\sum_{b,c}\gamma_{a b c}X_{ib}\left[q^2\Phi^{(j,k)}X_{jc}+\frac{1-q^{14}}{\dim_qV}\Phi^{(j,j)}X_{kc}\right]\\
=&q^2\Phi^{(j,k)}\left({\upgamma(v_a)}\right)^{(i,j)}+\frac{1-q^{14}}{\dim_qV}\Phi^{(j,j)}\left({\upgamma(v_a)}\right)^{(i,k)},
\end{align*}
Using this in \eqref{xppx}, we obtain the  relation \eqref{xPsi-1}.

This completes the proof of the lemma.
	\end{proof}


Finally we consider the $\UG$-invariants $\Upsilon^{(i,j,k,l)}$.
\begin{lemma}\label{xU}
Assume that $m\geq 4$. The elements $\Upsilon^{(i,j,k,l)}\in \cA_m^{\UG}$, for  $i<j<k<l$ in $[1, m]$,  satisfy the following relations for all $a\in[-3, 3]$:
\begin{align*}
&\left(q^{-1}X_{ra}\Upsilon^{(i,j,k,l)}{+}qX_{ja}\Upsilon^{(i,r,k,l)}\right)-\left(q\Upsilon^{(i,j,k,l)}X_{ra}+q^{-1}\Upsilon^{(i,r,k,l)}X_{ja}\right)\\
&=(q{-}q^{-1})\left[\Upsilon^{(r,j,k,l)}X_{ia}{-}\Psi^{(i,r,j)}(\upgamma(v_a))^{(k,l)}\right.\\
&\phantom{X}\left. {+}(q^2{+}q^{-2})\left(\Psi^{(j,k,l)}(\upgamma(v_a))^{(i,r)}{-}\Phi^{(i,r)}(\phi^{-}_{a})^{(j,k,l)}\right)\right], \ \text{if~}  i<r<j<k<l,\\
&\left(q^{-1}\Upsilon^{(i,j,k,l)}X_{ra}+q\Upsilon^{(i,j,r,l)}X_{ka}\right)-q\left(X_{ra}\Upsilon^{(i,j,k,l)}+q^{-1}X_{ka}\Upsilon^{(i,j,r,l)}\right)\\
&=(q{-}q^{-1})\left[\Upsilon^{(i,j,k,r)}X_{la}{-}\Psi^{(k,r,l)}(\upgamma(v_a))^{(i,j)}\right.\\
&\phantom{X}\left.{+} (q^2{+}q^{-2})\left(\Psi^{(i,j,k)}(\upgamma(v_a))^{(r,l)}{-}\Phi^{(r,l)}(\phi^{+}_{a})^{(i,j,k)}\right)\right], \ \text{if~}  i<j<k<r<l,
\end{align*}
\begin{align*}
X_{ra}\Upsilon^{(i,j,k,l)}=&\Upsilon^{(i,j,k,l)}X_{ra},  \qquad\text{if~}  r{<}i\mbox{~or~}r{>}k,\\
X_{ia}\Upsilon^{(i,j,k,l)}=&q^{2}\Upsilon^{(i,j,k,l)}X_{ia}+\frac{1-q^{14}}{\mathrm{dim}_qV}\Phi^{(i,i)} (\phi_a^-)^{(j,k,l)},\\
X_{la}\Upsilon^{(i,j,k,l)}=&q^{-2}\Upsilon^{(i,j,k,l)}X_{la}-q^{-2}\frac{1-q^{14}}{\mathrm{dim}_qV}\Phi^{(l,l)}(\phi_a^+)^{(i,j,k)},\\	
X_{ja}\Upsilon^{(i,j,k,l)}=&\Upsilon^{(i,j,k,l)}X_{ja}+(q^{-2}{-}1)\left[\Phi^{(i,j)}\left(\phi_a^-\right)^{(j,k,l)}\right.\\
&\left.-q^2\Psi^{(j,k,l)}\left({\upgamma(v_a)}\right)^{(i,j)}
-\frac{1-q^{14}}{\dim_qV}\Phi^{(j,j)}(\phi^{-}_a)^{(i,k,l)}\right],\\
X_{ka}\Upsilon^{(i,j,k,l)}
=&\Upsilon^{(i,j,k,l)}X_{ka}-(q^{-2}{-}1)
\left[q^2\Phi^{(k,l)}\left(\phi_a^+\right)^{(i,j,k)}\right.\\
&\left. +\frac{1-q^{14}}{\dim_qV}\Phi^{(k,k)}(\phi^{+}_{a})^{(i,j,l)}
-\Psi^{(i,j,k)}\left({\upgamma(v_a)}\right)^{(k,l)}\right].
\end{align*}
\end{lemma}
\begin{proof}[Comments on the proof]
The proof of this lemma is much the same as that for Lemma \ref{xPsi}, 
thus will be omitted. 
\end{proof}

\subsubsection{The invariants $\Upsilon^{'}{}^{(i,j,k,l)}$}\label{sect:up-prime}
Finally, we consider the invariants arising from 
$
\Upsilon'=\upgaprimeform(1)\in \left(V^{\ot 4}\right)^{\UG}$. 
We define 
\beq
\Upsilon^{'}{}^{(i,j,k,l)}=\tau_{{\bf d}_{i,  j, k, \ell}}(\Upsilon').
\eeq
To study these invariants, we need the following formulae with  $i<j\le k<l$.
\begin{align}
\hat\tau_{{\bf d}_i, j, k, \ell}(\begin{tikzpicture}[baseline=2pt,scale=0.6,color=\clr]
\draw[-,line width=1pt](1.2,0.9)parabola bend (1.8,-0.2) (2.4,0.9);	\draw[-,line width=1pt](1.4,0.9)parabola bend (1.8,0.2) (2.2,0.9);	
\end{tikzpicture})
=&\Phi^{(j,k)}\Phi^{(i,l)},\\
\hat\tau_{{\bf d}_i, j, k, \ell}(\begin{tikzpicture}[baseline=2pt,scale=0.6,color=\clr]
	\draw[-,line width=1pt](1.2,0.9)parabola bend (1.6,-0.2) (1.8,0.1);	\draw[-,line width=1pt](1.9,0.35)parabola  (2,0.9);
	\draw[-,line width=1pt](1.7,0.9)parabola bend (2.2,-0.2) (2.6,0.9);	
\end{tikzpicture})
=&\Phi^{(i,k)}\Phi^{(j,l)},\\
\hat\tau_{{\bf d}_i, j, k, \ell}(\begin{tikzpicture}[baseline=2pt,scale=0.6,color=\clr]
	\draw[-,line width=1pt](1.2,0.9)parabola bend (1.6,-0.2) (2,0.9);	
	\draw[-,line width=1pt](1.7,0.9)--(1.8,0.3);
	\draw[-,line width=1pt](1.9,0.1)parabola bend (2.2,-0.2) (2.6,0.9);	
\end{tikzpicture})
=&\Phi^{(j,l)}\Phi^{(i,k)}.
\end{align}

\begin{proof} By using Lemma \ref{ppp}, we obtain

\begin{align*}
\hat\tau_{{\bf d}_i, j, k, \ell}(\begin{tikzpicture}[baseline=2pt,scale=0.6,color=\clr]
\draw[-,line width=1pt](1.2,0.9)parabola bend (1.8,-0.2) (2.4,0.9);	\draw[-,line width=1pt](1.4,0.9)parabola bend (1.8,0.2) (2.2,0.9);	
\end{tikzpicture})
=&\sum_{a,b}\varphi_{a b}X_{ia}\Phi^{(j,k)}X_{lb}=\sum_{a,b}\varphi_{a b} \Phi^{(j,k)}X_{ia}X_{lb}\\
=&\Phi^{(j,k)}\Phi^{(i,l)}, \\
%
\hat\tau_{{\bf d}_i, j, k, \ell}(\begin{tikzpicture}[baseline=2pt,scale=0.6,color=\clr]
	\draw[-,line width=1pt](1.2,0.9)parabola bend (1.6,-0.2) (1.8,0.1);	\draw[-,line width=1pt](1.9,0.35)parabola  (2,0.9);
	\draw[-,line width=1pt](1.7,0.9)parabola bend (2.2,-0.2) (2.6,0.9);	
\end{tikzpicture})
=&\sum_{\substack{a,b\\a',b'}}\varphi_{a b} \varphi_{a' b'} 
  X_{ia}\beta X_{ja'}\alpha X_{kb}X_{lb'}\\
=&\sum_{\substack{a,b\\a',b'}}\varphi_{a b} \varphi_{a' b'} X_{ia} X_{kb}X_{ja'} X_{lb'}\\
=&\Phi^{(i,k)}\Phi^{(j,l)},
\end{align*}
\begin{align*}
\hat\tau_{{\bf d}_i, j, k, \ell}(\begin{tikzpicture}[baseline=2pt,scale=0.6,color=\clr]
	\draw[-,line width=1pt](1.2,0.9)parabola bend (1.6,-0.2) (2,0.9);	
	\draw[-,line width=1pt](1.7,0.9)--(1.8,0.3);
	\draw[-,line width=1pt](1.9,0.1)parabola bend (2.2,-0.2) (2.6,0.9);	
\end{tikzpicture})
=&\hat\tau_{{\bf d}_i, j, k, \ell}(\begin{tikzpicture}[baseline=2pt,scale=0.6,color=\clr]
	\draw[-,line width=1pt](1.8,0.1)--(1.7,-0.1);
	\draw[-,line width=1pt](1.7,-0.1)parabola bend (2.2,-0.4) (2.9,0.9);	
	\draw[-,line width=1pt](1.9,0.35)parabola  (2,0.9);
	\draw[-,line width=1pt](1.7,0.9)parabola bend (2.2,-0.2) (2.6,0.9);	
\end{tikzpicture})
=\sum_{\substack{a,b\\a',b'}}\varphi_{a b} \varphi_{a' b'}\beta X_{ia'}\alpha X_{ja}X_{kb'}X_{lb}\\
=&\sum_{\substack{a,b\\a',b'}}\varphi_{a b} \varphi_{a' b'} X_{ja}X_{ia'}X_{kb'}X_{lb}
=\sum_{a,b}\varphi_{a b} X_{ja}\Phi^{(i,k)}X_{lb}\\
&=\sum_{a,b}\varphi_{a b} X_{ja}X_{lb}\Phi^{(i,k)}
=\Phi^{(j,l)}\Phi^{(i,k)}.
\end{align*}

This proves the formulae. 
\end{proof}

We obtain from \eqref{cross} and \eqref{crossprime} the following relation
\beq\label{up-upprime}
\begin{tikzpicture}[baseline=-15pt,scale=0.8,color=\clr]
	\draw[-,line width=1pt](0.3,-0.4)--(0.1,-0.2);\draw[-,line width=1pt](0.3,-0.4)--(0.5,-0.2);
	\draw[-,line width=1pt](1.1,-0.4)--(0.9,-0.2);\draw[-,line width=1pt](1.1,-0.4)--(1.3,-0.2);
	\draw[-,line width=1pt](0.3,-0.4)parabola bend (0.7,-0.8) (1.1,-0.4);
\end{tikzpicture}+
\begin{tikzpicture}[baseline=-15pt,scale=0.8,color=\clr]
	\draw[-,line width=1pt](0.8,-0.8)--(0.8,-0.4);\draw[-,line width=1pt](0.8,-0.4)--(0.6,-0.2);\draw[-,line width=1pt](0.8,-0.4)--(1,-0.2);
	\draw[-,line width=1pt](0.3,-0.2)parabola bend (0.8,-0.8) (1.3,-0.2);
\end{tikzpicture}=(q^2-1+q^{-2})\left(
\begin{tikzpicture}[baseline=2pt,scale=0.6,color=\clr]
\draw[-,line width=1pt](1.2,0.9)parabola bend (1.8,-0.2) (2.4,0.9);	\draw[-,line width=1pt](1.4,0.9)parabola bend (1.8,0.2) (2.2,0.9);	
\end{tikzpicture}+
\begin{tikzpicture}[baseline=2pt,scale=0.6,color=\clr]
\draw[-,line width=1pt](1.2,0.9)parabola bend (1.5,-0.2) (1.8,0.9);	\draw[-,line width=1pt](1.9,0.9)parabola bend (2.2,-0.2) (2.5,0.9);	
\end{tikzpicture}\right)-\left(\begin{tikzpicture}[baseline=2pt,scale=0.6,color=\clr]
	\draw[-,line width=1pt](1.2,0.9)parabola bend (1.6,-0.2) (2,0.9);	
	\draw[-,line width=1pt](1.7,0.9)--(1.8,0.3);
	\draw[-,line width=1pt](1.9,0.1)parabola bend (2.2,-0.2) (2.6,0.9);	
\end{tikzpicture}+
\begin{tikzpicture}[baseline=2pt,scale=0.6,color=\clr]
	\draw[-,line width=1pt](1.2,0.9)parabola bend (1.6,-0.2) (1.8,0.1);	\draw[-,line width=1pt](1.9,0.35)parabola  (2,0.9);
	\draw[-,line width=1pt](1.7,0.9)parabola bend (2.2,-0.2) (2.6,0.9);	
\end{tikzpicture}\right).
\eeq
Now apply $\hat\tau_{{\bf d}_i, j, k, \ell}$ with $i<j\leq k<l$ to both sides, and use the formulae above. We arrive at 
\beq\label{up-prime}
	\Upsilon^{(i,j,k,l)}{+}\Upsilon^{'}{}^{(i,j,k,l)}&=&(q^2-1+q^{-2})\left(\Phi^{(j,k)}\Phi^{(i,l)}{+}\Phi^{(i,j)}\Phi^{(k,l)}\right)\\
&&-\left(\Phi^{(j,l)}\Phi^{(i,k)}{+}\Phi^{(i,k)}\Phi^{(j,l)}\right). \nonumber
\eeq
It is clear from the diagrammatic definition of $\Upsilon'$ that ${\Upsilon^{'}}^{(i,j,j,l)}=0$.

\subsection{FFT of non-commutative algebraic invariant theory}\label{sect:FFT}

This section contains the key results on a non-commutative algebraic invariant theory for $\UG$.  In particular, Theorems \ref{thm:span} and \ref{FFT} 
amounts to a non-commutaive first fundamental theorem of invariant theory. 

Recall that the map $\hat\tau_{\bf d}:  \Hom_{\UG}(0, V^{|{\bf d}|})\lra A_{\bf d}$ defined by 
\eqref{eq:key-map} is surjective, thus 
$\cA_m^{\UG}=\sum_{{\bf d}\in\ZZ_+^m} {\rm im}(\hat\tau_{\bf d})$.

We use Theorem \ref{lem:basis} to construct a spanning set for $\cA_m^{\UG}$. 
For any ${\bf d}\in \ZZ_+^m$, let 
\[
 \widehat{\Acycl}_{\bf d}:=\left\{\hat\tau_{\bf d}(\zeta) \mid \zeta \in  \Acycl(0, |{\bf d}|)\right\}.
\]
Since $\Hom_{\UG}(0, V^{|{\bf d}|})$ is spanned by  $\Acycl(0, |{\bf d}|)$ by Theorem \ref{lem:basis}, ${\rm im}(\hat\tau_{\bf d})$ is spanned by $\widehat{\Acycl}_{\bf d}$. 
This immediately leads to the following result. 

\begin{theorem}\label{thm:span}
As vector space over $\CC(q)$, 
\[
\cA_m^{\UG}=\sum_{{\bf d}\in \ZZ_+^m} \CC(q)\widehat{\Acycl}_{\bf d}. 
\]
\end{theorem}
Since ${\cA_m^{\UG}}_{(1, 1, \dots, 1)}\simeq \left(V^{\ot m}\right)^{\UG}$,  the theorem immediately implies the following result.  
\begin{corollary}\label{corr:tensor-FFT}
The following relation holds: 
$\left(V^{\ot m}\right)^{\UG}\simeq \CC(q)\widehat{\Acycl}_{(1, 1, \dots, 1)}$. 
\end{corollary}

Conceptually the theorem provides a veracious account of the $\UG$-invariants in $\cA_m$ 
at the level of vector spaces. In order to better understand the invariants, 
we need to consider the  structure of $\cA_m^{\UG}$ as a non-commutative associative algebra. 
The following theorem describes a set of generators for the subalgebra 
$\cA_{m}^{\UG}$ of $\UG$-invariants in $\cA_{m}$.  

	\begin{theorem}[FFT] \label{FFT}
		As an associative algebra,  $\cA_{m}^{\UG}$ is generated by the elements 
		$\Phi^{(i,j)}$ with $1\le i\leq j\le m$, and $\Psi^{(r,s,t)}$ with $1\le r<s<t\le m$.
	\end{theorem}

\begin{remark}
Note that Theorem \ref{FFT} requires quadratic and 
qubic generators only. This is very different from the classical situation, where 
the (commutative) subalgebra of $G_2$-invariants in the symmetric algebra of $\CC^7\ot\CC^m$ is generated by quadratic, 
qubic and also quartic invariants  \cite{S} (also see \cite{HZ}).  
This is due to the non-commutativity of the generators $\Phi^{(i,j)}$  and $\Psi^{(r,s,t)}$.
We will see in Remark \ref{rmk:2-2-tensors} that the quartic  invariants $\Upsilon^{(i,j,k,l)}$ 
arise from $q$-commutators of $\Phi^{(r,s)}$. 
\end{remark}

\subsection{Proof of Theorem \ref{FFT}}

Let us make some preparations for the proof.
We will follow the general strategy describe in Section \ref{sect:elmt}, and retain notation introduced in that section. 
Also, we use 
$\ele$ to denote an unspecified element in $A$,
and use $\ele_{ d}$ (resp. $\ele_{<d}$) to indicate that it degree is $d$ (resp. lower than $d$).	
	
Given a $(0, |{\bf d}|)$ diagram $\zeta$, 
where ${\bf d}=(d_1, \dots, d_k)\in\ZZ_+^k$, 
we enumerate its end points by $1, 2, \dots,  |{\bf d}|$ from left to right. 
Divide them into $k$ {\em bands} $I_1=[1, d_1], I_2=[d_1+1, d_2], , \dots, I_k=[d_{k-1}+1, d_k]$, and 
label the $r$-th end point by $i$ if $r\in I_i$. 
We will say that the end points labelled by $i$ belong to the $i$-th band.

Now 
apply $\hat\tau_{\bf d}$ to $\zeta$.  
Then $\hat\tau_{\bf d}(\zeta)$ can be expressed as a sum of products of $X_{j a}$'s. 
If for some $i$, all $X_{i a}$'s (there are $d_i$ of them) associated with the $i$-th band 
appear next to each other,  thus we have  a sequence of the form 
$X_{ia_1}X_{ia_2}\cdots X_{ia_{d_i}}$ in the expression for $\hat\tau_{\bf d}(\zeta)$, 
we symbolically denote 
the sequence by $\XX_{[d_i]}$ to simplify notation.

If  for some $r$, the $r$-th and $(r+1)$-th end points of $\zeta$ are the two top end points of a crossing
(whose bottom points are jointed to other parts of the diagram $\zeta$),  which belong to the same band $I_i$, 
then by using Lemma \ref{double}, we can replace the crossing as follows: 
\beq\label{eq:cr-repl}
\baln
\text{over crossing $\cross$} &\rightsquigarrow& 
q^2 \begin{tikzpicture}[baseline=4pt,scale=0.5,color=\clr]
	\draw[-,line width=1pt,color=\clr](0,0)--(0,1);
	\draw[-,line width=1pt,color=\clr](0.5,0)--(0.5,1);
\end{tikzpicture}-q^{-5}\frac{q^7-q^{-7}}{\dim_q V}
\begin{tikzpicture}[baseline=4pt,scale=0.5,color=\clr]
	\draw[-,line width=1pt,color=\clr](0,1)parabola bend (0.5,0.6)(1,1);
	\draw[-,line width=1pt,color=\clr](0,0)parabola bend(0.5,0.4)(1,0);
\end{tikzpicture} ,\\
\text{under crossing $\Icross$}&\rightsquigarrow&
q^{-2} \begin{tikzpicture}[baseline=4pt,scale=0.5,color=\clr]
	\draw[-,line width=1pt,color=\clr](0,0)--(0,1);
	\draw[-,line width=1pt,color=\clr](0.5,0)--(0.5,1);
\end{tikzpicture}+q^{5}\frac{q^7-q^{-7}}{\dim_q V}
\begin{tikzpicture}[baseline=4pt,scale=0.5,color=\clr]
	\draw[-,line width=1pt,color=\clr](0,1)parabola bend (0.5,0.6)(1,1);
	\draw[-,line width=1pt,color=\clr](0,0)parabola bend(0.5,0.4)(1,0);
\end{tikzpicture}  .
\ealn
\eeq

The following results will also be needed below.  
Let  
\beq\label{eq:zetas1-4}
\zeta_1= \begin{tikzpicture}[baseline=4pt,scale=0.5,color=\clr]
	\draw[-,line width=1pt,color=\clr](0,1)parabola bend (0.3,0) (0.5,0.4);
	\draw[-,line width=1pt,color=\clr](0.7,0.6)--(0.9,1);
	\draw[-,line width=1pt,color=\clr](0.4,1)--(0.9,0);
\end{tikzpicture},\quad
\zeta_2=\begin{tikzpicture}[baseline=4pt,scale=0.5,color=\clr]
	\draw[-,line width=1pt,color=\clr](0,1)parabola bend (0.4,0) (0.8,1);
	\draw[-,line width=1pt,color=\clr](0.4,1)--(0.6,0.5);
	\draw[-,line width=1pt,color=\clr](0.75,0.35)--(0.9,0);
\end{tikzpicture},\quad
\zeta_3= \begin{tikzpicture}[baseline=4pt,scale=0.5,color=\clr]
	\draw[-,line width=1pt,color=\clr](-0.2,1)parabola bend (0.3,0) (0.5,0.4);
	\draw[-,line width=1pt,color=\clr](0.7,0.6)--(0.9,1);	
	\draw[-,line width=1pt,color=\clr](0.2,0)--(0.2,1);	\draw[-,line width=1pt,color=\clr](0.4,1)--(0.9,0);
\end{tikzpicture},\quad
\zeta_4=\begin{tikzpicture}[baseline=4pt,scale=0.5,color=\clr]
	\draw[-,line width=1pt,color=\clr](-0.2,1)parabola bend (0.3,0) (0.8,1);
	\draw[-,line width=1pt,color=\clr](0.4,1)--(0.6,0.55);\draw[-,line width=1pt,color=\clr](0.2,0)--(0.2,1);		\draw[-,line width=1pt,color=\clr](0.75,0.35)--(0.9,0);
\end{tikzpicture}\ ,
\eeq
and consider $\tau_{\bf d}(\zeta_\alpha(v_c))$. 
For $\alpha =1$, suppose that the bands of the three endpoints are 
$i_1, i_2, i_3$ with $i_1\leq i_2<i_3$,
then we have 
\[
\baln
\tau_{\bf d}(\zeta_1(v_c))&
=\sum_{a, a'}\varphi_{a a'}X_{i_1a}\beta_{t}(X_{i_2c})\alpha_{t}(X_{i_3a'})\\
&=\sum_{a, a'} \varphi_{a a'} X_{i_1a}X_{i_3a'}X_{i_2c}
=\Phi^{(i_1,i_3)}X_{i_2c}.
\ealn
\]

Using the diagramatic relation $\begin{tikzpicture}[baseline=4pt,scale=0.5,color=\clr]
	\draw[-,line width=1pt,color=\clr](0.41,0)--(0.63,0.37);\draw[-,line width=1pt,color=\clr](0.8,0.5)--(1,1);
	\draw[-,line width=1pt,color=\clr](0.6,1)parabola bend (1,0) (1.35,1);
\end{tikzpicture}=\begin{tikzpicture}[baseline=4pt,scale=0.5,color=\clr]
	\draw[-,line width=1pt,color=\clr](0,1)parabola bend (0.4,0) (0.8,1);
	\draw[-,line width=1pt,color=\clr](0.4,1)--(0.6,0.5);
	\draw[-,line width=1pt,color=\clr](0.75,0.35)--(0.9,0);
\end{tikzpicture}$, 
we can change $\zeta_2$ to a diagram involving an over crossing. Then we obtain, for $i_1< i_2\leq i_3$, 
\[
\baln
\tau_{\bf d}(\zeta_2(v_c))&=\sum_{a, a'} \varphi_{a a'}\beta_{t}(X_{i_1a})\alpha_{t}(X_{i_2c})X_{i_3a'}\\
&=\sum_{a, a'} \varphi_{a a'} X_{i_2c}X_{i_1a}X_{i_3a'}=X_{i_2c}\Phi^{(i_1,i_3)}.
\ealn
\]

By taking a ${\bf d}$ which assigns the four endpoints of each of 
$\zeta_3(v_c)$ and $\zeta_4(v_c)$ to
bands  $i, j, k, \ell$, we can prove analogous formulae. 

To summarise, we have the following result. 
\begin{lemma} Let $\zeta_\alpha$ with $\alpha=1, 2, 3, 4$ be as defined by \eqref{eq:zetas1-4}. For all $v_c$, and and appropriate ${\bf d}$, the following formulae hold. 
\beq\label{zeta1}
\tau_{\bf d}(\zeta_1(v_c))&=&\Phi^{(i_1,i_3)}X_{i_2c}, \quad \text{for $i_1\leq i_2<i_3$}.
\eeq
\beq\label{zeta2}
\tau_{\bf d}(\zeta_2(v_c))&=&X_{i_2c}\Phi^{(i_1,i_3)}, \quad \text{for $i_1< i_2\leq i_3$}. 
\eeq
\beq
\tau_{\bf d}(\zeta_3(v_c))&=&\Psi^{(i,j,l)}X_{kc},\quad \mbox{for~}k<l;\label{zeta3}\\
\tau_{\bf d}(\zeta_4(v_c))&=&\left\{
\begin{array}{rl}
q^2X_{jc}\Psi^{(i,j,l)}-q^{-5}\frac{q^7-q^{-7}}{\dim_q V}\left(\upgamma(v_c)\right)^{(i,l)}\Phi^{(j,j)},\\
\mbox{if~}i<j=k<\ell;\\
X_{kc}\Psi^{(i,j,l)}, \quad \mbox{otherwise}.\hspace{3cm}
\end{array}
\right. \label{zeta4}
\eeq
\end{lemma}

Now we prove Theorem \ref{FFT}.

\begin{proof}[Proof of Theorem \ref{FFT}]

For any  ${\bf d}\in\ZZ_+$, it follows Theorem \ref{lem:basis} that $A_{\bf d}$ 
is equal to the image of $\hat\tau_{\bf d}$ 
applied to the space spanned by the set of acyclic 
trivalent $(0,|{\bf d}|)$-graphs $\Acycl(0,|{\bf d}|)$. 
Thus in order to prove Theorem \ref{FFT}, we only need to show that if $\zeta\in \Acycl(0,|{\bf d}|)$, 
then $\hat\tau_{\bf d}(\zeta)$ belongs to the subalgebra $A'$ of $A$ generated 
by $\Phi^{(i, j)}$ and $\Psi^{(i, j, k)}$. 

We will prove this by induction on $|{\bf d}|$. 
The analysis on the elementary invariants in Section \ref{sect:elmt} has treated all cases with $|{\bf d}|\le 4$, 
which provides the initial step for the induction.  The induction hypothesis assumes that 
all $\ele_{<|{\bf d}|}$ belong to $A'$. 

\medskip
\noindent{\em a). Juxtaposed acyclic graphs}.

If $\zeta=\zeta_1\ot \zeta_2$ is the juxtaposition of two acyclic graphs $\zeta_1, \zeta_2$, then $\hat\tau_{\bf d}(\zeta)
=\hat\tau_{\bf d}(\zeta_1)\hat\tau_{\bf d}(\zeta_2)$, where both facts has degrees $< |{\bf d}|$.  
Thus invariants associated with juxtapositions of acyclic subgraphs belong to $A'$. 

\medskip
\noindent{\em b). Special nested acyclic graphs}.

Consider an acyclic graph $\zeta$ of the following form 
$$
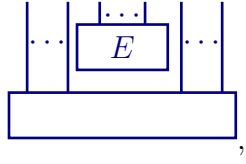
\begin{tikzpicture}[baseline=13pt,scale=0.6,color=\clr]
\draw[-,line width=1pt] (-0.1,2)--(-0.1,4);\draw[-,line width=1pt] (0.8,2)--(0.8,4);
\draw [-,line width=1pt](3.3,2)--(3.3,4);\draw [-,line width=1pt](4.2,2)--(4.2,4);
\draw[-,line width=1pt](-0.5,1)rectangle(4.5,2);	
\draw[-,line width=1pt](1,2.5)rectangle(3,3.5);\node at (2,3){$E$};
\draw[-,line width=1pt](1.5,3.5)--(1.5,4);
\draw[-,line width=1pt](2.5,3.5)--(2.5,4);
\node at (2.05,3.7){$\cdots$};
\node at (0.4,3.1){$\cdots$};
\node at (3.8,3.1){$\cdots$};
\end{tikzpicture},
$$
where $E$ is an acyclic subgraph.  We say that $\zeta$ is  an {\em acyclic nested  graph}.

Let $E$
be one of the first two elementary graphs in Figure \ref{Figure-elementary},
for example,  
\begin{tikzpicture}[baseline=4pt,scale=0.5,color=\clr]
\draw[-,line width=1pt,color=\clr](1.3,1)parabola bend (2,0) (2.7,1);
\end{tikzpicture}.
Assume that the end points of 
$E$
are labelled by $i, j$,  with $j=i$ or $i+1$.
Then 
$\hat\tau_{\bf d}(\zeta)$ is a linear combination of elements of the form (for a fixed $s$)
\beq\label{eq:symbolic}
\XX_{[d_1]}\cdots \XX_{[d_{s-1}]}\XX_{[d_s]}\Phi^{(i, j)}\cdots X_{|{\bf d}|a_{|{\bf d}|}}.  
\eeq
If $i=j$,  we easily 
obtain $
\hat\tau_{\bf d}(\zeta)={\Phi^{(i,i)}}\ele_{|{\bf d}|-2}$, 
since $\Phi^{(i, i)}$ is central. 
Now assume that $i\ne j$. We use Lemma \ref{ppp} to move $\Phi^{(i,j)}$ in each term \eqref{eq:symbolic} of $\hat\tau_{\bf d}(\zeta)$ to the left if  $s\le i$.  We obtain 
\[
\baln
\hat\tau_{\bf d}(\zeta)
&=\left\{
\begin{array}{ll}
{\Phi^{(i,j)}}\ele_{|{\bf d}|-2}\,,&\mbox{if~}s<i;\\
{\Phi^{(i,j)}}\ele^{\,\prime}_{|{\bf d}|-2}+{\Phi^{(i,i)}}\ele^{\,\prime\prime}_{|{\bf d}|-2}\,,&\mbox{if~}s=i.
\end{array}
\right.
\ealn
\]
We can similarly treat the case with $E$
being the second graph in Figure \ref{Figure-elementary}, to show that the corresponding 
invariant lies in $A'$.

Thus the invariants associated with acyclic nested trivalent graphs belong to $A'$,  
 if $E$ is one of the first two graphs in Figure \ref{Figure-elementary}. 

\medskip
\noindent{\em c). General acyclic graphs}.

Consider an acyclic graph $\zeta\in \Acycl(0,|{\bf d}|)$ with $|{\bf d}|\geq 5$, which is not a juxtaposition of acyclic sub-graphs. 
Then $\zeta$ is  either of one of the following  three types, 
$$
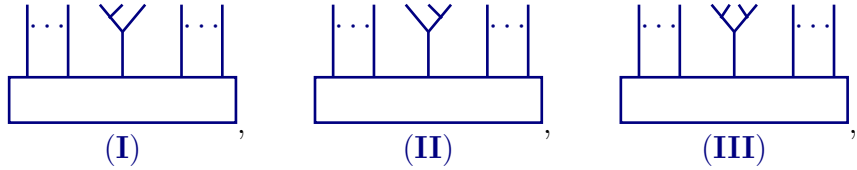
\begin{tikzpicture}[baseline=13pt,scale=0.6,color=\clr]
\draw[-,line width=1pt] (-0.1,2)--(-0.1,3.6);\draw[-,line width=1pt] (0.8,2)--(0.8,3.6);
\draw [-,line width=1pt](3.3,2)--(3.3,3.6);\draw [-,line width=1pt](4.2,2)--(4.2,3.6);
\draw[-,line width=1pt](-0.5,1)rectangle(4.5,2);	
\draw[-,line width=1pt] (2,2)--(2,3);
\draw[-,line width=1pt] (2,3)--(1.5,3.6);\draw[-,line width=1pt] (2,3)--(2.5,3.6);
\draw[-,line width=1pt] (1.7,3.3)--(2,3.6);
\node at (0.4,3.1){$\cdots$};
\node at (3.8,3.1){$\cdots$};
\node at (2,0.4){${\bf (I)}$};
\end{tikzpicture},\qquad
\begin{tikzpicture}[baseline=13pt,scale=0.6,color=\clr]
\draw[-,line width=1pt] (-0.1,2)--(-0.1,3.6);\draw[-,line width=1pt] (0.8,2)--(0.8,3.6);
\draw [-,line width=1pt](3.3,2)--(3.3,3.6);\draw [-,line width=1pt](4.2,2)--(4.2,3.6);
\draw[-,line width=1pt](-0.5,1)rectangle(4.5,2);	
\draw[-,line width=1pt] (2,2)--(2,3);
\draw[-,line width=1pt] (2,3)--(1.5,3.6);\draw[-,line width=1pt] (2,3)--(2.5,3.6);
\draw[-,line width=1pt] (2.3,3.3)--(2,3.6);
\node at (0.4,3.1){$\cdots$};
\node at (3.8,3.1){$\cdots$};
\node at (2,0.4){${\bf (II)}$};
\end{tikzpicture},\qquad
\begin{tikzpicture}[baseline=13pt,scale=0.6,color=\clr]
\draw[-,line width=1pt] (-0.1,2)--(-0.1,3.6);\draw[-,line width=1pt] (0.8,2)--(0.8,3.6);
\draw [-,line width=1pt](3.3,2)--(3.3,3.6);\draw [-,line width=1pt](4.2,2)--(4.2,3.6);
\draw[-,line width=1pt](-0.5,1)rectangle(4.5,2);	
\draw[-,line width=1pt] (2,2)--(2,3);
\draw[-,line width=1pt] (2,3)--(1.5,3.6);\draw[-,line width=1pt] (2,3)--(2.5,3.6);
\draw[-,line width=1pt] (1.7,3.3)--(1.9,3.6);\draw[-,line width=1pt] (2.3,3.3)--(2.1,3.6);
\node at (0.4,3.1){$\cdots$};
\node at (3.8,3.1){$\cdots$};
\node at (2,0.4){${\bf (III)}$};
\end{tikzpicture}, 
$$
or a nested graph, each of whose connected components is one of the first two graphs in Figure \ref{Figure-elementary}.  

If $\zeta$ is such a nested graph, we have treated it already. Now we consider the three types of graphs ${\bf(I)}$, ${\bf(II)}$ and ${\bf(III)}$. We will treat ${\bf(II)}$ and ${\bf(III)}$ in detail;  ${\bf(I)}$ can be treated in much the same way as ${\bf(II)}$.

We have the following relations 
\[
\baln
{\bf(II)}&\qquad \begin{tikzpicture}[baseline=40pt,scale=0.6,color=\clr]
\draw[-,line width=1pt] (-0.1,2)--(-0.1,3.6);\draw[-,line width=1pt] (0.8,2)--(0.8,3.6);
\draw [-,line width=1pt](3.3,2)--(3.3,3.6);\draw [-,line width=1pt](4.2,2)--(4.2,3.6);
\draw[-,line width=1pt](-0.5,1)rectangle(4.5,2);	
\draw[-,line width=1pt] (2,2)--(2,3);
\draw[-,line width=1pt] (2,3)--(1.5,3.6);\draw[-,line width=1pt] (2,3)--(2.5,3.6);
\draw[-,line width=1pt] (2.3,3.3)--(2,3.6);
\node at (0.4,3.1){$\cdots$};
\node at (3.8,3.1){$\cdots$};
\end{tikzpicture}
\  =\  \begin{tikzpicture}[baseline=40pt,scale=0.6,color=\clr]
\draw[-,line width=1pt] (-0.1,2)--(-0.1,3.6);\draw[-,line width=1pt] (0.8,2)--(0.8,3.6);
\draw [-,line width=1pt](3.3,2)--(3.3,3.6);\draw [-,line width=1pt](4.2,2)--(4.2,3.6);
\draw[-,line width=1pt](-0.5,1)rectangle(4.5,2);	
\draw[-,line width=1pt] (2.6,2.5)--(2.6,3);
\draw[-,line width=1pt] (2.6,3)--(2.1,3.6);\draw[-,line width=1pt] (2.6,3)--(3.1,3.6);
\draw[-,line width=1pt] (2.9,3.3)--(2.6,3.6);
\draw[-,line width=1pt,color=\clr](2.6,2.5)parabola bend (2.3,2.3) (1.9,2.5);
\draw[-,line width=1pt,color=\clr](1.9,2.5)--(1.9,3.1);
\draw[-,line width=1pt,color=\clr](1.9,3.1)parabola bend (1.6,3.6) (1.3,3.1);
\draw[-,line width=1pt,color=\clr](1.3,3.1)--(1.3,2);
\node at (0.4,3.1){$\cdots$};
\node at (3.8,3.1){$\cdots$};
\end{tikzpicture},\\
{  }\\
{\bf(III)}&\qquad
\begin{tikzpicture}[baseline=40pt,scale=0.6,color=\clr]
\draw[-,line width=1pt] (-0.1,2)--(-0.1,3.6);\draw[-,line width=1pt] (0.8,2)--(0.8,3.6);
\draw [-,line width=1pt](3.3,2)--(3.3,3.6);\draw [-,line width=1pt](4.2,2)--(4.2,3.6);
\draw[-,line width=1pt](-0.5,1)rectangle(4.5,2);	
\draw[-,line width=1pt] (2,2)--(2,3);
\draw[-,line width=1pt] (2,3)--(1.5,3.6);\draw[-,line width=1pt] (2,3)--(2.5,3.6);
\draw[-,line width=1pt] (1.7,3.3)--(1.9,3.6);\draw[-,line width=1pt] (2.3,3.3)--(2.1,3.6);
\node at (0.4,3.1){$\cdots$};
\node at (3.8,3.1){$\cdots$};
\end{tikzpicture}\  = \
\begin{tikzpicture}[baseline=40pt,scale=0.6,color=\clr]
\draw[-,line width=1pt] (-0.1,2)--(-0.1,3.6);\draw[-,line width=1pt] (0.8,2)--(0.8,3.6);
\draw [-,line width=1pt](3.3,2)--(3.3,3.6);\draw [-,line width=1pt](4.2,2)--(4.2,3.6);
\draw[-,line width=1pt](-0.5,1)rectangle(4.5,2);	
\draw[-,line width=1pt] (2.6,2.5)--(2.6,3);
\draw[-,line width=1pt] (2.6,3)--(2.1,3.6);\draw[-,line width=1pt] (2.6,3)--(3.1,3.6);
\draw[-,line width=1pt] (2.3,3.3)--(2.5,3.6);\draw[-,line width=1pt] (2.9,3.3)--(2.6,3.6);
\draw[-,line width=1pt,color=\clr](2.6,2.5)parabola bend (2.3,2.3) (1.9,2.5);
\draw[-,line width=1pt,color=\clr](1.9,2.5)--(1.9,3.1);
\draw[-,line width=1pt,color=\clr](1.9,3.1)parabola bend (1.6,3.6) (1.3,3.1);
\draw[-,line width=1pt,color=\clr](1.3,3.1)--(1.3,2);
\node at (0.4,3.1){$\cdots$};
\node at (3.8,3.1){$\cdots$};
\end{tikzpicture}.
\ealn
\]
Now apply \eqref{checkD} to the right hand sides to reduce the number of their vertices. 
The first two terms on the right-hand side of \eqref{checkD} produces diagrams with nested sud-diagrams,
which were treated already. 
Thus we only need to
 consider the two diagrams arising from the crossings on the left-hand side of \eqref{checkD}.
They are  
$$\zeta_{{\bf (II)}}^{+}:=\begin{tikzpicture}[baseline=40pt,scale=0.6,color=\clr]
\draw[-,line width=1pt] (-0.1,2)--(-0.1,3.6);\draw[-,line width=1pt] (0.8,2)--(0.8,3.6);
\draw [-,line width=1pt](3.3,2)--(3.3,3.6);\draw [-,line width=1pt](4.2,2)--(4.2,3.6);
\draw[-,line width=1pt](-0.5,1)rectangle(4.5,2);	
\draw[-,line width=1pt] (2.4,3)--(2.6,3.6);
\draw[-,line width=1pt,color=\clr](2.1,3.6)parabola bend (2.6,2.5) (3.1,3.6);
\draw[-,line width=1pt,color=\clr](1.9,2.3)parabola bend (1.6,3.6) (1.3,3.1);
\draw[-,line width=1pt,color=\clr](1.3,3.1)--(1.3,2);
\draw[-,line width=1pt,color=\clr](1.9,2.3)parabola bend (2.1,2.2)(2.25,2.7);
\node at (0.4,3.1){$\cdots$};
\node at (3.8,3.1){$\cdots$};
\node at (2.1,3.9){$i$};
\node at (2.6,3.9){$j$};
\node at (3.1,3.9){$k$};
\end{tikzpicture} \quad\mbox{and} \quad
\zeta_{{\bf (II)}}^{-}:=\begin{tikzpicture}[baseline=40pt,scale=0.6,color=\clr]
\draw[-,line width=1pt] (-0.1,2)--(-0.1,3.6);\draw[-,line width=1pt] (0.8,2)--(0.8,3.6);
\draw [-,line width=1pt](3.3,2)--(3.3,3.6);\draw [-,line width=1pt](4.2,2)--(4.2,3.6);
\draw[-,line width=1pt](-0.5,1)rectangle(4.5,2);	
\draw[-,line width=1pt] (2.1,3.6)--(2.3,2.9);
\draw[-,line width=1pt,color=\clr](2.5,2.75)parabola bend (2.7,2.5) (3.1,3.6);
\draw[-,line width=1pt,color=\clr](1.9,2.3)parabola bend (1.6,3.6) (1.3,3.1);
\draw[-,line width=1pt,color=\clr](1.3,3.1)--(1.3,2);
\draw[-,line width=1pt,color=\clr](1.9,2.3)parabola bend (2.1,2.3)(2.6,3.6);
\node at (0.4,3.1){$\cdots$};
\node at (3.8,3.1){$\cdots$};
\node at (2.1,3.9){$i$};
\node at (2.6,3.9){$j$};
\node at (3.1,3.9){$k$};\end{tikzpicture}$$ 
in case ${\bf (II)}$, and  
$$\zeta_{{\bf (III)}}^{+}:=\begin{tikzpicture}[baseline=40pt,scale=0.6,color=\clr]
\draw[-,line width=1pt] (-0.1,2)--(-0.1,3.6);\draw[-,line width=1pt] (0.8,2)--(0.8,3.6);
\draw [-,line width=1pt](3.3,2)--(3.3,3.6);\draw [-,line width=1pt](4.2,2)--(4.2,3.6);
\draw[-,line width=1pt](-0.5,1)rectangle(4.5,2);	
\draw[-,line width=1pt,color=\clr](1.9,3.1)parabola bend (1.6,3.6) (1.3,3.1);
\draw[-,line width=1pt,color=\clr](1.3,3.1)--(1.3,2);
\draw[-,line width=1pt,color=\clr](1.9,3.1)parabola bend (2.2,2.4)(2.4,2.7);
\draw[-,line width=1pt,color=\clr](2.6,2.9)--(2.8,3.6);
\draw[-,line width=1pt,color=\clr](2.6,2.7)parabola bend (2.8,2.4)(3.1,3.6);
\draw[-,line width=1pt,color=\clr](2.6,2.7)--(2.3,3.3);
\draw[-,line width=1pt,color=\clr](2.3,3.3)--(2.1,3.6);
\draw[-,line width=1pt,color=\clr](2.3,3.3)--(2.5,3.6);
\node at (2.1,3.95){$i$};
\node at (2.5,3.95){$j$};
\node at (2.85,3.95){$k$};
\node at (3.2,3.95){$l$};
\node at (0.4,3.1){$\cdots$};
\node at (3.8,3.1){$\cdots$};
\end{tikzpicture}\quad\mbox{and}\quad \zeta_{{\bf (III)}}^{-}:=\begin{tikzpicture}[baseline=40pt,scale=0.6,color=\clr]
\draw[-,line width=1pt] (-0.1,2)--(-0.1,3.6);\draw[-,line width=1pt] (0.8,2)--(0.8,3.6);
\draw [-,line width=1pt](3.3,2)--(3.3,3.6);\draw [-,line width=1pt](4.2,2)--(4.2,3.6);
\draw[-,line width=1pt](-0.5,1)rectangle(4.5,2);	
\draw[-,line width=1pt,color=\clr](1.9,3.1)parabola bend (1.6,3.6) (1.3,3.1);
\draw[-,line width=1pt,color=\clr](1.3,3.1)--(1.3,2);
\draw[-,line width=1pt,color=\clr](1.9,3.1)parabola bend (2.2,2.4)(2.8,3.6);
\draw[-,line width=1pt,color=\clr](2.6,2.7)parabola bend (2.8,2.4)(3.1,3.6);
\draw[-,line width=1pt,color=\clr](2.5,2.9)--(2.3,3.3);
\draw[-,line width=1pt,color=\clr](2.3,3.3)--(2.1,3.6);
\draw[-,line width=1pt,color=\clr](2.3,3.3)--(2.5,3.6);
\node at (0.4,3.1){$\cdots$};
\node at (3.8,3.1){$\cdots$};
\node at (2.1,3.95){$i$};
\node at (2.5,3.95){$j$};
\node at (2.85,3.95){$k$};
\node at (3.2,3.95){$l$};
\end{tikzpicture}$$ 
in case ${\bf (III)}$, 
where $i, j, k$ and $\ell$ indicate the bands which the relevant endpoints belong to 
when $\hat\tau_{\bf d}$ is applied to the diagrams.

Consider $\hat\tau_{\bf d}(\zeta_{{\bf (II)}}^{\pm})$.  If $i=j$ or $j=k$, the relations \eqref{eq:cr-repl} reduce both diagrams to sums of nested graphs, which were treated already. Thus we assume $i<j<k$.  

It follows from \eqref{zeta2} that the relevant endpoints lead to 
$\Phi^{(i,k)}X_{jb}$  $(i< j<k)$ in  $\hat\tau_{\bf d}(\zeta_{{\bf (II)}}^{-})$, thus  $\hat\tau_{\bf d}(\zeta_{{\bf (II)}}^{-})$ is a linear combination of elements  of the form $X_{[d_1]}\cdots \XX_{[d_s}]\Phi^{(i,k)}X_{jb}\XX_{[d_{s+1}]}\cdots \XX_{[d_k]}$ (for some fixed $s$).  The same arguments as in the case of nested graphs lead to 
\beq\label{II-}
\hat\tau_{\bf d}(\zeta_{{\bf (II)}}^{-})=\left\{
\begin{array}{ll}
\Phi^{(i,k)}\ele_{|{\bf d}|-2}\,,&\mbox{if~}s<i;\\
\Phi^{(i,k)}\ele^{\,\prime}_{|{\bf d}|-2}+\Phi^{(i,i)}\ele^{\,\prime\prime}_{|{\bf d}|-2},&\mbox{if~}s=i.
\end{array}\right.
\eeq
Also, by using \eqref{zeta1} for $i<j< k$ and Lemma \ref{ppp},
we similarly obtain
\beq\label{II+}
\hat\tau_{\bf d}(\zeta_{{\bf (II)}}^{+})=\ele_{|{\bf d}|-2}\Phi^{(i,k)}+c\,  \ele^{\,\prime}_{|{\bf d}|-2}\Phi^{(i,i)},\eeq
for some scalar $c$, which is zero if the endpoint label by $k$ and 
the one on its immediate  right  are not in the same band.

Now we consider $\zeta_{{\bf (III)}}^{\pm}$. Note that if $i=j$, the defining property of a trivalent vertex leads to $\zeta_{{\bf (III)}}^{\pm}=0$. If $k=\ell$, equation \eqref{eq:cr-repl} enables us to reduce them to the acyclic nested graph case. Hence we only need to consider the case with $i<j$ and $k<\ell$. 

By applysing  \eqref{zeta4}, we conclude that if  $ i<j=k<l$, then $\hat\tau_{\bf d}(\zeta_{{\bf (III)}}^{+})$ is a linear combination of 
\begin{align*}
&q^2\XX_{[d_1]}\cdots \XX_{[d_s]}X_{jc}\Psi^{(i,j,l)}\XX_{[d_{s+1}]}\cdots \XX_{[d_k]}\\
&-q^{-5}\frac{q^7-q^{-7}}{\dim_q V}\XX_{[d_1]}\cdots \XX_{[d_s]}\left(\upgamma(v_c)\right)^{(i,l)}\Phi^{(j,j)}\XX_{[d_{s+1}]}\cdots \XX_{[d_k]},  
\end{align*}
for some fixed $s$; and if $ i<j<k<l$, it is a  linear combination of elements of the form
\[ 
\XX_{[d_1]}\cdots \XX_{[d_s]}X_{kc}\Psi^{(i,j,l)}\XX_{[d_{s+1}]}\cdots \XX_{[d_k]},
\]
for some fixed $s$. 
To simply the expression of $\hat\tau_{\bf d}(\zeta_{{\bf (III)}}^{+})$, we use  the commutative relations between $\Psi^{(i,j,l)}$ and  $X_{r a}$ given by Lemma \ref{xPsi}. We obtained 
\beq\label{III+}
\hat\tau_{\bf d}(\zeta_{{\bf (III)}}^{+}){=}\left\{
\begin{array}{rl}
\ele_{|{\bf d}|-3}\Psi^{(i,j,l)}{+}\ele_{|{\bf d}|-2}\Phi^{(j,j)}{+}\ele^{\,\prime}_{|{\bf d}|-2}\Phi^{(l,l)},\\
\mbox{if~}i{<}j{=}k{<}l;\\
\ele^{\,\prime}_{|{\bf d}|-3}\Psi^{(i,j,l)}+ c\ \ele^{\,\prime\prime}_{|{\bf d}|-2}\Phi^{(l,l)},\quad \mbox{otherwise},
\end{array}\right.
\eeq
for some scalar $c$, which is zero if the endpoint labelled by $\ell$ and 
the one on its immediate  right  are not in the same band.

We can show in a similar way by using \eqref{zeta3}
and Lemma \ref{xPsi} that 
\beq\label{III-}
\hat\tau_{\bf d}(\zeta_{{\bf (III)}}^{-})=\Psi^{(i,j,l)}\ele_{|{\bf d}|-3}+ c \Phi^{(i,i)}\ele_{|{\bf d}|-2},
\eeq
for some scalar $c$, which is zero if the endpoint labelled by $i$ and 
the one on its immediate left  are not in the same band.
	
By induction hypothesis $\ele_d\in A'$ for all $d<|{\bf d}|$. Hence it follows from parts $a)$ and $b)$ and 
\eqref{II-}--\eqref{III-} in part $c)$ that $A=A'$. 

This completes the proof.
\end{proof}

\subsection{Further structural properties of the subalgebra of invariants}\label{sect:struct}
To study the algebraic structure of $\cA_m^{\UG}$,
we investigate commutation relations among its generators.

A new feature of the present case is that $q$-commutation relations of the generators 
give rise to invariants of higher degrees, which are associated with well defined elements of $\cup_{r\ge 0}\Acycl(0, r)$, 
see Remark \ref{rmk:2-2-tensors} and Lemma \ref{lem:Theta}. 
Some of these higher degree invariants are indispensable in the sense that they can not be expressed as linear combinations of ordered monomials of the generators, as discussed in Section \ref{comments}.

 \subsubsection{Commutation relations among generators}

The following relations among the invariants $\Phi^{(i, j)}$ 
are easy consequences of Lemma \ref{ppp}.
\begin{lemma}\label{2-2-tensors}
		These $\UG$-invariants $\Phi^{(i,j)}(i\leq j)$ satisfy the following relations:
		\begin{equation*}\label{phicross}
			\begin{aligned}
				\Phi^{(i, i)} \Phi^{(j, k)}-\Phi^{(j, k)} \Phi^{(i, i)} & =0, \quad \text { for all } i, j, k, \\
				\Phi^{(i, k)} \Phi^{(i, j)}-q^{2} \Phi^{(i, j)} \Phi^{(i, k)} & =\frac{1-q^{14}}{\textit{dim}_q(V)}\Phi^{(i, i)}\Phi^{(j, k)}, \quad k \neq i, j ; i<j, \\
				\Phi^{(i, j)} \Phi^{(j, k)}-q^2\Phi^{(j, k)} \Phi^{(i, j)}  & =\frac{1-q^{14}}{\textit{dim}_q(V)}\Phi^{(j, j)} \Phi^{(i, k)}, \quad k \neq i, j ; i<j, \\
				\Phi^{(i, j)} \Phi^{(k, l)}-\Phi^{(k, l)} \Phi^{(i, j)} & =0, \quad j<k,\mbox{ or }k<i<j<l,\mbox{ or } l<i,\\			
q^{-1} \Phi^{(k, l)}\Phi^{(i, j)}{-}q\Phi^{(i, j)}\Phi^{(k, l)} 
&{=}(q{-}q^{-1})\left[\Phi^{(i, l)} \Phi^{(k,j)}{-}(q^2{+}q^{-2})\Phi^{(i, k)} \Phi^{(j,l)}{+}\Upsilon^{(i,k,j,l)}\right],   \\
&\hspace{5cm} i{<}k{<}j{<}l.
\end{aligned} 
\end{equation*}
\end{lemma}

\begin{remark}	\label{rmk:2-2-tensors}
Particularly noteworthy is the last relation, which produces the quartic invariants $\Upsilon^{(i,k,j,l)}$ as $q$-commutators of quadratic invariants.  
\end{remark}	
\begin{remark}	
The last relation of Lemma \ref{2-2-tensors} has the following more conceptual proof. 
We can easily obtain  the following relation from \eqref{checkD}.
		$$
		q^{-1}\begin{tikzpicture}[baseline=2pt,scale=0.6,color=\clr]
			\draw[-,line width=1pt](1.2,0.9)parabola bend (1.6,-0.2) (2,0.9);	
			\draw[-,line width=1pt](1.7,0.9)--(1.8,0.3);
			\draw[-,line width=1pt](1.9,0.1)parabola bend (2.2,-0.2) (2.6,0.9);	
		\end{tikzpicture}-q
		\begin{tikzpicture}[baseline=2pt,scale=0.6,color=\clr]
			\draw[-,line width=1pt](1.2,0.9)parabola bend (1.6,-0.2) (1.8,0.1);	\draw[-,line width=1pt](1.9,0.35)parabola  (2,0.9);
			\draw[-,line width=1pt](1.7,0.9)parabola bend (2.2,-0.2) (2.6,0.9);	
		\end{tikzpicture}=(q-q^{-1})
		\begin{tikzpicture}[baseline=2pt,scale=0.6,color=\clr]
			\draw[-,line width=1pt](1.2,0.9)parabola bend (1.8,-0.2) (2.4,0.9);	\draw[-,line width=1pt](1.4,0.9)parabola bend (1.8,0.2) (2.2,0.9);	
		\end{tikzpicture}-(q{-}q^{-1})(q^2+q^{-2})
		\begin{tikzpicture}[baseline=2pt,scale=0.6,color=\clr]
			\draw[-,line width=1pt](1.2,0.9)parabola bend (1.5,-0.2) (1.8,0.9);	\draw[-,line width=1pt](1.9,0.9)parabola bend (2.2,-0.2) (2.5,0.9);	
		\end{tikzpicture}+(q-q^{-1})	\begin{tikzpicture}[baseline=2pt,scale=0.6,color=\clr]
	\draw[-,line width=1pt] (1.2,0.4)--(0.9,0.8); \draw[-,line width=1pt] (1.2,0.4)--(1.5,0.8); 
	\draw[-,line width=1pt] (0.2,0.4)--(0.5,0.8); \draw[-,line width=1pt] (0.2,0.4)--(-0.1,0.8); 	
	\draw[-,line width=1pt](1.2,0.4)parabola bend(0.7,-0.3) (0.2,0.4);
		\end{tikzpicture}.
		$$
 Applying  the map $\hat\tau_{{\bf d}_{i, k, j, \ell}}$ with $i{<}k{<}j{<}l$ to both sides of the above relation, we obtain 
the last relation of  Lemma \ref{2-2-tensors}. 
\end{remark}

We now consider commutation relations between $\Phi^{(i, j)}$ and $\Psi^{(r,s,t)}$. 
When some of the indices coincide, the commutation relations actually become simpler.  We have the following result. 
\begin{lemma}\label{lem:phipsi-overlap}
Assume that $i \leq j$ and $s{<}t$. Then the following relations hold. 

\begin{equation}\label{ijij}
\begin{aligned}
& \Psi^{(i,j,k)}\Phi^{(i,j)}{=}\Phi^{(i,j)}\Psi^{(i,j,k)}, \\
& \Psi^{(i,j,k)}\Phi^{(j,k)}{=}\Phi^{(j,k)}\Psi^{(i,j,k)},\\
&\Psi^{(i,j,k)}\Phi^{(i,k)}{=}\Phi^{(i,k)}\Psi^{(i,j,k)};
\end{aligned}
\end{equation}
\begin{equation}\label{ijist}
			\begin{aligned}
				\Psi^{(s,t,i)}\Phi^{(i,j)}&=q^{2}\Phi^{(i,j)}\Psi^{(s,t,i)}+\frac{1-q^{14}}{\textit{dim}_q(V)}\Phi^{(i,i)}\Psi^{(s,t,j)}, \  \text { for } j>t,\\	
				\Psi^{(s,i,t)}\Phi^{(i,j)}&=q^{2}\Phi^{(i,j)}\Psi^{(s,i,t)}+\frac{1-q^{14}}{\textit{dim}_q(V)}\Phi^{(i,i)}\Psi^{(s,j,t)}, \ \text { for } s<j<t,\\
				\Psi^{(s,i,t)}\Phi^{(i,j)}&=\Phi^{(i,j)}\Psi^{(s,i,t)}+(1{-}q^{-2})\left[\Phi^{(s,i)}\Psi^{(i,t,j)}{-}\Phi^{(i,t)}\Psi^{(s,i,j)}\right], \ \text { for } j>t,\\
				\Psi^{(i,s,t)}\Phi^{(i,j)}&=q^{2}\Phi^{(i,j)}\Psi^{(i,s,t)}+\frac{1-q^{14}}{\textit{dim}_q(V)}\Phi^{(i,i)}\Psi^{(j,s,t)}, \  \text { for } j<s,\\
				\Psi^{(i,s,t)}\Phi^{(i,j)}&=q^{2}\Phi^{(i,j)}\Psi^{(i,s,t)}+\frac{1-q^{14}}{\textit{dim}_q(V)}\Phi^{(i,i)}\Psi^{(s,t,j)}, \ \text { for } j>t,\\
				\Psi^{(i,s,t)}\Phi^{(i,j)}&=\Phi^{(i,j)}\Psi^{(i,s,t)} 
+ (1{-}q^{-2})\Phi^{(i,t)}\Psi^{(i,s,j)}\\
& \phantom{X}-(1{-}q^{-2})
				\left(q^2\Phi^{(i,s)}\Psi^{(i,j,t)}{+}\frac{1{-}q^{14}}{\dim_qV}\Phi^{(i,i)}\Psi^{(s,j,t)}\right), \text { for }s< j<t; 
\end{aligned}
\end{equation}
\begin{equation}\label{ijjst}
	\begin{aligned}
	q^2\Psi^{(j,s,t)}\Phi^{(i,j)}&=\Phi^{(i,j)}\Psi^{(j,s,t)}-\frac{1-q^{14}}{\textit{dim}_q(V)}\Phi^{(j,j)}\Psi^{(i,s,t)}, \  \text { for }i<s,\\
				q^2\Psi^{(s,j,t)}\Phi^{(i,j)}&=\Phi^{(i,j)}\Psi^{(s,j,t)}-\frac{1-q^{14}}{\textit{dim}_q(V)}\Phi^{(j,j)}\Psi^{(s,i,t)}, \  \text { for }s<i<t,\\
				\Psi^{(s,j,t)\Phi^{(i,j)}}&=\Phi^{(i,j)}\Psi^{(s,j,t)}-(1{-}q^{-2})\left[\Phi^{(j,t)}\Psi^{(i,s,j)}{-}\Phi^{(s,j)}\Psi^{(i,j,t)}\right], \  \text { for }i<s,\\
				\Psi^{(s,t,j)}\Phi^{(i,j)}&=q^2\Phi^{(i,j)}\Psi^{(s,t,j)}+\frac{1-q^{14}}{\textit{dim}_q(V)}\Phi^{(j,j)}\Psi^{(i,s,t)}, \  \text { for }i<s,\\
				\Psi^{(s,t,j)}\Phi^{(i,j)}&=q^2\Phi^{(i,j)}\Psi^{(s,t,j)}+\frac{1-q^{14}}{\textit{dim}_q(V)}\Phi^{(j,j)}\Psi^{(s,t,i)}, \  \text { for }i>t,\\
\Psi^{(s,t,j)}\Phi^{(i,j)}&=\Phi^{(i,j)}\Psi^{(s,t,j)} -(1{-}q^{-2})\Phi^{(s,j)}\Psi^{(i,t,j)}\\
&\phantom{X}-(1{-}q^{-2})\left(\frac{1-q^{14}}{\textit{dim}_q(V)}\Phi^{(j,j)}\Psi^{(s,i,t)}{-}\Phi^{(t,j)}\Psi^{(s,i,j)}\right), \text { for }s<i<t.
\end{aligned}\end{equation}
\end{lemma}
	\begin{proof}
The proof of \eqref{ijjst} is the same as that for \eqref{ijist}, the we consider only latter only.
The first five relations of \eqref{ijist} can be obtained by using Lemma \ref{ppp} and Lemma \ref{xPsi},
but the sixth relation requires a different proof. 
By applying the relation \eqref{checkD} twice, we obtain
\begin{equation}\label{phipsiev}
\begin{aligned}
&\left[q^{-1}\begin{tikzpicture}[baseline=2pt,scale=0.6,color=\clr]
			\draw[-,line width=1pt](1.2,0.9)parabola bend (1.9,-0.2) (2.6,0.9);
			\draw[-,line width=1pt](1.6,0.9)parabola bend (1.9,0.2) (2.2,0.9);	\draw[-,line width=1pt](1.45,0.2)--(1.6,0.55);\draw[-,line width=1pt](1.75,0.65)--(1.85,0.9);
		\end{tikzpicture}-q\begin{tikzpicture}[baseline=2pt,scale=0.6,color=\clr]
			\draw[-,line width=1pt](1.2,0.9)parabola bend (1.9,-0.2) (2.6,0.9);
			\draw[-,line width=1pt](1.5,0.9)--(1.6,0.7); \draw[-,line width=1pt](1.7,0.5) parabola bend (1.9,0.2) (2.2,0.9);	
			\draw[-,line width=1pt](1.45,0.2)--(1.85,0.9);
		\end{tikzpicture}~\right]-\left[q^{-1}\begin{tikzpicture}[baseline=2pt,scale=0.6,color=\clr]
			\draw[-,line width=1pt](1.2,0.9)parabola bend (1.6,-0.2) (1.8,0.1);	\draw[-,line width=1pt](1.9,0.3)parabola  (2,0.9);
			\draw[-,line width=1pt](1.7,0.6)parabola bend (2.2,-0.2) (2.6,0.9);	\draw[-,line width=1pt](1.7,0.6)--(1.5,0.9);\draw[-,line width=1pt](1.7,0.6)--(1.9,0.9);
		\end{tikzpicture}-q\begin{tikzpicture}[baseline=2pt,scale=0.6,color=\clr]
			\draw[-,line width=1pt](1.2,0.9)parabola bend (1.6,-0.2) (2,0.9);	
			\draw[-,line width=1pt](1.7,0.6)--(1.8,0.3);\draw[-,line width=1pt](1.7,0.6)--(1.5,0.9);\draw[-,line width=1pt](1.7,0.6)--(1.9,0.9);
			\draw[-,line width=1pt](1.9,0.1)parabola bend (2.2,-0.2) (2.6,0.9);	
		\end{tikzpicture}~\right]\\
=&(q-q^{-1})\left[\begin{tikzpicture}[baseline=2pt,scale=0.6,color=\clr]
			\draw[-,line width=1pt](1.2,0.9)parabola bend (1.9,-0.2) (2.6,0.9);
			\draw[-,line width=1pt](1.9,0.9) parabola bend (2.1,0.2) (2.3,0.9);	
			\draw[-,line width=1pt](1.37,0.33)--(1.65,0.9);
		\end{tikzpicture}-\begin{tikzpicture}[baseline=2pt,scale=0.6,color=\clr]
			\draw[-,line width=1pt](1.1,0.9)parabola bend (1.4,-0.2) (1.7,0.9);\draw[-,line width=1pt](1.4,-0.2)--(1.4,0.9);			
			\draw[-,line width=1pt](1.9,0.9) parabola bend (2.1,-0.2) (2.3,0.9);	
		\end{tikzpicture}+(q^2+q^{-2})\left(\begin{tikzpicture}[baseline=2pt,scale=0.55,color=\clr]
			\draw[-,line width=1pt](0.3,0.9)parabola bend (0.8,-0.2) (1.3,0.9);
			\draw[-,line width=1pt](0.5,0.9)parabola bend (0.8,0) (1.1,0.9);	\draw[-,line width=1pt](0.8,0)--(0.8,0.9);
		\end{tikzpicture}-\begin{tikzpicture}[baseline=2pt,scale=0.55,color=\clr]
			\draw[-,line width=1pt](1.8,0.9)parabola bend (2,0.2) (2.2,0.9);	\draw[-,line width=1pt](1.6,0.9)parabola bend (2.2,-0.2) (2.6,0.6);	\draw[-,line width=1pt](2.6,0.6)--(2.4,0.9);\draw[-,line width=1pt](2.6,0.6)--(2.8,0.9);
		\end{tikzpicture}\right)				
		\right].
\end{aligned}\end{equation}

Introduce
		$F=\left(P[2\lambda_1]+P[0]\right) \otimes\id\otimes\id\otimes \id$.
We have 
\begin{align*}
		F\left(~\begin{tikzpicture}[baseline=4pt,scale=0.6,color=\clr]
			\draw[-,line width=1pt](1.2,0.9)parabola bend (1.6,-0.2) (1.8,0.1);	\draw[-,line width=1pt](1.9,0.3)parabola  (2,0.9);
			\draw[-,line width=1pt](1.7,0.6)parabola bend (2.2,-0.2) (2.6,0.9);	\draw[-,line width=1pt](1.7,0.6)--(1.5,0.9);\draw[-,line width=1pt](1.7,0.6)--(1.9,0.9);
		\end{tikzpicture}(1)\right)&=q^{-2}F\left(\begin{tikzpicture}[baseline=4pt,scale=0.6,color=\clr]
			\draw[-,line width=1pt](1.2,0.9)parabola bend (1.9,-0.2) (2.6,0.9);
			\draw[-,line width=1pt](1.5,0.9)--(1.6,0.7); \draw[-,line width=1pt](1.7,0.5) parabola bend (1.9,0.2) (2.2,0.9);	
			\draw[-,line width=1pt](1.45,0.2)--(1.85,0.9);
		\end{tikzpicture} (1)\right)+\frac{q^{-8}-q^{6}}{\mathrm{dim}_qV}\begin{tikzpicture}[baseline=-1pt,scale=1,color=\clr]
			\draw[-,line width=1pt](1.2,0.4)parabola bend (1.4,-0.2) (1.6,0.4);	\draw[-,line width=1pt](1.7,0.4)parabola bend (2,-0.2) (2.3,0.4);\draw[-,line width=1pt](2,-0.2)--(2,0.4);
		\end{tikzpicture}(1),\\
		F\left(~\begin{tikzpicture}[baseline=4pt,scale=0.6,color=\clr]
			\draw[-,line width=1pt](1.2,0.9)parabola bend (1.9,-0.2) (2.6,0.9);
			\draw[-,line width=1pt](1.6,0.9)parabola bend (1.9,0.2) (2.2,0.9);	\draw[-,line width=1pt](1.45,0.2)--(1.6,0.55);\draw[-,line width=1pt](1.75,0.65)--(1.85,0.9);
		\end{tikzpicture}(1)\right)&=q^{-2}\left(\begin{tikzpicture}[baseline=4pt,scale=0.6,color=\clr]
			\draw[-,line width=1pt](1.2,0.9)parabola bend (1.6,-0.2) (2,0.9);	
			\draw[-,line width=1pt](1.7,0.6)--(1.8,0.3);\draw[-,line width=1pt](1.7,0.6)--(1.5,0.9);\draw[-,line width=1pt](1.7,0.6)--(1.9,0.9);
			\draw[-,line width=1pt](1.9,0.1)parabola bend (2.2,-0.2) (2.6,0.9);	
		\end{tikzpicture}(1)\right)+\frac{q^{-8}-q^{6}}{\mathrm{dim}_qV}\begin{tikzpicture}[baseline=-1pt,scale=1,color=\clr]
			\draw[-,line width=1pt](1.2,0.4)parabola bend (1.4,-0.2) (1.6,0.4);	\draw[-,line width=1pt](1.7,0.4)parabola bend (2,-0.2) (2.3,0.4);\draw[-,line width=1pt](2,-0.2)--(2,0.4);
		\end{tikzpicture}(1).
\end{align*}
Let ${\bf d}={\bf d}_{r,i, s, j, t}$, and apply 
$\hat\tau_{{\bf d}}\circ F$ to both sides of \eqref{phipsiev},
we obtain
		$$\Psi^{(i,s,t)}\Phi^{(i,j)}=\Phi^{(i,j)}\Psi^{(i,s,t)}+(1{-}q^{-2})\left[\Phi^{(i,t)}\Psi^{(i,s,j)}{-}\Psi^{(i,j,t)}\Phi^{(i,s)}\right].$$
This together with the fourth relation of \eqref{ijist} implies the sixth relation of \eqref{ijist}.	
		
Now we prove  \eqref{ijij}. 
By using the fourth and fifth relations in Lemma \ref{ppp}, we can show that  
\[
\begin{aligned}
q^2\Psi^{(i,j,k)}\Phi^{(i,j)}
=&q^2\sum_{a,b,c} c_{a b c} X_{ia}X_{jb}\Phi^{(i,j)}X_{kc}\\
=&\sum_{a,b,c}c_{a b c} X_{ia}\left[\Phi^{(i,j)}X_{jb}-\frac{1-q^{14}}{\dim_q(V)}X_{ib}\Phi^{(j,j)}\right]X_{kc}. 
\end{aligned}
\]
Note that the second term on the right does not contribute. To see this, recall that $\Phi^{(j,j)}$ is central,  thus 
$\sum_{a,b,c}c_{a b c} X_{ia}X_{ib}\Phi^{(j,j)}X_{kc} = \Phi^{(j,j)}\sum_{a,b,c}c_{a b c} X_{ia}X_{ib}X_{kc}$, which vanishes since
$\sum_{a,b,c}c_{a b c} X_{ia}X_{ib}X_{kc}=0$. 
Therefore, 
\[
\begin{aligned}
q^2\Psi^{(i,j,k)}\Phi^{(i,j)}
=&\sum_{a,b,c}c_{a b c} X_{ia}\Phi^{(i,j)}X_{jb}X_{kc}\\
=&\sum_{a,b,c}c_{a b c} \left[q^{2}\Phi^{(i,j)}X_{ia}+\frac{1-q^{14}}{\dim_q(V)}X_{ja}\Phi^{(i,i)}\right]X_{jb}X_{kc}\\
=&q^{2}\Phi^{(i,j)}\Psi^{(i,j,k)},
\end{aligned}
\]
where the last equality is obtained by the same argument given above.  
This proves the first relation of \eqref{ijij},
and the other relations can be proved in the same way.
\end{proof}

\begin{lemma}\label{lem:phipsi}
Assume that $i \leq j$ and $r{<}s{<}t$. 
The invariants $\Phi^{(i, j)}$ and $\Psi^{(r,s,t)}$ satisfy the following relations.
\begin{enumerate}
\item 
If $i{\neq}r,s,t$, 
\begin{align*}
\Phi^{(i, i)} \Psi^{(r,s,t)}=\Psi^{(r,s,t)} \Phi^{(i, i)}.
\end{align*}
\item If $j{<}r$, or $r{<}i,j{<}s$, 
or $s{<}i,j{<}t$, 
or $t{<}i$, or $i{<}r, t{<}j$,
\begin{align*}
\Phi^{(i, j)} \Psi^{(r,s,t)}=\Psi^{(r,s,t)}\Phi^{(i, j)}.
\end{align*}
\item If $i{<}j{<}r{<}s{<}t$, 	
\begin{align*}
&q^{-1}\Psi^{(i,j,s)}\Phi^{(r,t)}-q\Phi^{(r,t)}\Psi^{(i,j,s)}+q\Psi^{(i,r,s)}\Phi^{(j,t)}-q^{-1}\Phi^{(j,t)}\Psi^{(i,r,s)}\\
=&
					(q{-}q^{-1})\left[\Phi^{(s,t)}\Psi^{(i,j,r)}{-}\Phi^{(i,t)}\Psi^{(j,r,s)}\right]\\
&\phantom{X}+(q{-}q^{-1})(q^2{+}q^{-2})\left[\Phi^{(i,j)}\Psi^{(r,s,t)}{-}\Phi^{(r,s)}\Psi^{(i,j,t)}\right],
\end{align*}		
\begin{align*}			
&q^{-1}\Psi^{(i,r,s)}\Phi^{(j,t)}-q\Phi^{(j,t)}\Psi^{(i,r,s)} +q\Psi^{(i,r,t)}\Phi^{(j,s)}-q^{-1}\Phi^{(j,s)}\Psi^{(i,r,t)}\\
=&
(q{-}q^{-1})\left[\Phi^{(i,j)}\Psi^{(r,s,t)}-\Phi^{(j,r)}\Psi^{(i,s,t)}\right]\\
&\phantom{X}+(q{-}q^{-1})(q^2{+}q^{-2})\left[\Phi^{(r,s)}\Psi^{(i,j,t)}-\Phi^{(i, t)}\Psi^{(j,r,s)}\right],
\end{align*}		
\begin{align*}	
&q^{-1}\Psi^{(i,r,t)}\Phi^{(j,s)}-q\Phi^{(j,s)}\Psi^{(i,r,t)}+q\Psi^{(j,r,t)}\Phi^{(i,s)}-q^{-1}\Phi^{(i,s)}\Psi^{(j,r,t)}\\
=&
(q{-}q^{-1})\left[\Phi^{(r,s)}\Psi^{(i,j,t)}{-}\Phi^{(s,t)}\Psi^{(i,j,r)}\right]\\
&\phantom{X}{+}(q{-}q^{-1})(q^2{+}q^{-2})\left[\Phi^{(i, t)}\Psi^{(j,r,s)}{-}\Phi^{(j,r)}\Psi^{(i,s,t)}\right],
\end{align*}		
\begin{align*}				
&q^{-1}\Psi^{(r,j,t)}\Phi^{(i,s)}-q\Phi^{(i,s)}\Psi^{(r,j,t)} +q\Psi^{(j,s,t)}\Phi^{(i,r)}-q^{-1}\Phi^{(i, r)}\Psi^{(j,s,t)}\\
=&
(q{-}q^{-1})\left[\Phi^{(i, t)}\Psi^{(r,j,s)}-\Phi^{(i, j)}\Psi^{(r,s,t)}\right]\\
&\phantom{X}{+}(q{-}q^{-1})(q^2{+}q^{-2})\left[\Phi^{(j,r)}\Psi^{(i,s,t)}-\Phi^{(s,t)}\Psi^{(i,r,j)}\right].										
\end{align*} 
\end{enumerate}

\end{lemma}
\begin{proof}
Part (1) and (2) are  consequences of relations in Lemma \ref{xPsi} and \ref{xU}.  
Part (3) follows from the last relations in Lemma \ref{ppp} and Lemma \ref{xPsi}.
It can also be proved diagrammatically. 
Take the last relation of part (3) as an example.
By applying the relation \eqref{checkD} twice, we obtain
\[
\baln
q^{-1}
		\begin{tikzpicture}[baseline=2pt,scale=0.55,color=\clr]
			\draw[-,line width=1pt](1.2,0.9)parabola bend (1.6,-0.2) (1.8,0.1);	\draw[-,line width=1pt](1.9,0.35)parabola  (2,0.9);
			\draw[-,line width=1pt](1.7,0.9)parabola bend (2.2,-0.2) (2.6,0.6);	\draw[-,line width=1pt](2.6,0.6)--(2.4,0.9);\draw[-,line width=1pt](2.6,0.6)--(2.8,0.9);
		\end{tikzpicture}{-}q
		\begin{tikzpicture}[baseline=2pt,scale=0.55,color=\clr]
			\draw[-,line width=1pt](1.2,0.9)parabola bend (1.6,-0.2) (2,0.9);	
			\draw[-,line width=1pt](1.7,0.9)--(1.8,0.3);
			\draw[-,line width=1pt](1.9,0.1)parabola bend (2.2,-0.2) (2.6,0.6);	\draw[-,line width=1pt](2.6,0.6)--(2.4,0.9);\draw[-,line width=1pt](2.6,0.6)--(2.8,0.9);
		\end{tikzpicture}
		&{=}q^{-1}
		\begin{tikzpicture}[baseline=2pt,scale=0.55,color=\clr]	
			\draw[-,line width=1pt](1.2,0.9)parabola bend (1.6,-0.2) (2,0.9);	
			\draw[-,line width=1pt](1.7,0.6)--(1.8,0.3);\draw[-,line width=1pt](1.7,0.6)--(1.5,0.9);\draw[-,line width=1pt](1.7,0.6)--(1.9,0.9);
			\draw[-,line width=1pt](1.9,0.1)parabola bend (2.2,-0.2) (2.6,0.9);	
		\end{tikzpicture}{-}q
		\begin{tikzpicture}[baseline=2pt,scale=0.55,color=\clr]
			\draw[-,line width=1pt](1.2,0.9)parabola bend (1.6,-0.2) (1.8,0.1);	\draw[-,line width=1pt](1.9,0.3)parabola  (2,0.9);
			\draw[-,line width=1pt](1.7,0.6)parabola bend (2.2,-0.2) (2.6,0.9);	\draw[-,line width=1pt](1.7,0.6)--(1.5,0.9);\draw[-,line width=1pt](1.7,0.6)--(1.9,0.9);
		\end{tikzpicture}{+}(q{-}q^{-1})
		\left(
		\begin{tikzpicture}[baseline=2pt,scale=0.55,color=\clr]
			\draw[-,line width=1pt](0.8,0.9)parabola bend (1.1,-0.2) (1.4,0.9);
			\draw[-,line width=1pt](1.7,0.9)parabola bend (2.2,-0.2) (2.6,0.9);	\draw[-,line width=1pt](2.2,-0.2)--(2.2,0.9);
		\end{tikzpicture}{-}
		\begin{tikzpicture}[baseline=2pt,scale=0.55,color=\clr]
			\draw[-,line width=1pt](0.3,0.9)parabola bend (0.8,-0.2) (1.3,0.9);
			\draw[-,line width=1pt](0.5,0.9)parabola bend (0.8,0) (1.1,0.9);	\draw[-,line width=1pt](0.8,0)--(0.8,0.9);
		\end{tikzpicture}
		\right)\\
		&{+}(q{-}q^{-1})(q^2{+}q^{-2})\left(
		\begin{tikzpicture}[baseline=2pt,scale=0.55,color=\clr]
			\draw[-,line width=1pt](0.8,0.9)parabola bend (1.1,-0.2) (1.4,0.9);\draw[-,line width=1pt](1.1,-0.2)--(1.1,0.9);
			\draw[-,line width=1pt](1.7,0.9)parabola bend (1.9,-0.2) (2.1,0.9);
		\end{tikzpicture}{-}
		\begin{tikzpicture}[baseline=2pt,scale=0.55,color=\clr]
			\draw[-,line width=1pt](1.8,0.9)parabola bend (2,0.2) (2.2,0.9);	\draw[-,line width=1pt](1.6,0.9)parabola bend (2.2,-0.2) (2.6,0.6);	\draw[-,line width=1pt](2.6,0.6)--(2.4,0.9);\draw[-,line width=1pt](2.6,0.6)--(2.8,0.9);
		\end{tikzpicture}\right).
\ealn
\]
Now applying   
$\hat\tau_{{\bf d}_{i, j, r, s, t}}$ to the above equation, and then re-arranging the terms,  
we obtain the last relation  of part (3). 
\end{proof}

Finally, we consider commutation relations 
between $\Psi^{(j,s,t)}$ and $\Phi^{(i,r)}$ for $i{<}j{<}r{<}s{<}t$. 
Let $\Theta
:=\begin{tikzpicture}[baseline=-13pt,scale=0.7,color=\clr]
		\draw[-,line width=1pt](1.4,-0.4)--(1.2,-0.2);\draw[-,line width=1pt](1.4,-0.4)--(1.6,-0.2);
		\draw[-,line width=1pt](0.9,-0.4)--(0.9,-0.8);		
		\draw[-,line width=1pt](0.9,-0.4)--(0.7,-0.2);\draw[-,line width=1pt](0.9,-0.4)--(1.1,-0.2);
		\draw[-,line width=1pt](0.3,-0.2)parabola bend (0.9,-0.8) (1.4,-0.4);
	\end{tikzpicture}
$, and   
denote $\Theta^{(i,j,r,s,t)} = \hat\tau_{{\bf d}_{i, j, r, s, t}}(\Theta)$. 
We can derive the following relation from equation \eqref{checkD}. 
\beq\label{grapht=3}
(q-q^{-1})\begin{tikzpicture}[baseline=-15pt,scale=0.8,color=\clr]
		\draw[-,line width=1pt](1.4,-0.4)--(1.2,-0.2);\draw[-,line width=1pt](1.4,-0.4)--(1.6,-0.2);
		\draw[-,line width=1pt](0.9,-0.4)--(0.9,-0.8);		
		\draw[-,line width=1pt](0.9,-0.4)--(0.7,-0.2);\draw[-,line width=1pt](0.9,-0.4)--(1.1,-0.2);
		\draw[-,line width=1pt](0.3,-0.2)parabola bend (0.9,-0.8) (1.4,-0.4);
	\end{tikzpicture}=q^{-1}
		\begin{tikzpicture}[baseline=2pt,scale=0.55,color=\clr]
			\draw[-,line width=1pt](1.2,0.9)parabola bend (1.6,-0.2) (1.8,0.1);	\draw[-,line width=1pt](1.9,0.35)parabola  (2,0.9);
			\draw[-,line width=1pt](1.7,0.9)parabola bend (2.2,-0.2) (2.6,0.6);	\draw[-,line width=1pt](2.6,0.6)--(2.4,0.9);\draw[-,line width=1pt](2.6,0.6)--(2.8,0.9);
		\end{tikzpicture}{-}q
		\begin{tikzpicture}[baseline=2pt,scale=0.55,color=\clr]
			\draw[-,line width=1pt](1.2,0.9)parabola bend (1.6,-0.2) (2,0.9);	
			\draw[-,line width=1pt](1.7,0.9)--(1.8,0.3);
			\draw[-,line width=1pt](1.9,0.1)parabola bend (2.2,-0.2) (2.6,0.6);	\draw[-,line width=1pt](2.6,0.6)--(2.4,0.9);\draw[-,line width=1pt](2.6,0.6)--(2.8,0.9);
		\end{tikzpicture}-(q{-}q^{-1})\left(
\begin{tikzpicture}[baseline=22pt,scale=0.6,color=\clr]		
	\draw[-,line width=1pt](0.5,2)parabola bend (0.7,1) (0.9,2);
	\draw[-,line width=1pt](1.1,2)parabola bend (1.4,1) (1.7,2);\draw[-,line width=1pt](1.4,1)--(1.4,2);
\end{tikzpicture}-(q^2{+}q^{-2})
\begin{tikzpicture}[baseline=22pt,scale=0.6,color=\clr]		
	\draw [-,line width=1pt](0.5,2)parabola bend (1,1)(1.4,1.4);
	\draw [-,line width=1pt](1.4,1.4)--(1.3,2);
	\draw [-,line width=1pt](1.4,1.4)--(1.6,2);
	\draw [-,line width=1pt](0.8,2)parabola bend (1,1.5)(1.2,2);
\end{tikzpicture}\right).
\eeq
Applying  $\hat\tau_{{\bf d}_{i, j, r, s, t}}$  to \eqref{grapht=3}, and re-arranging the terms, we obtain
the following result. 
\begin{lemma} \label{lem:Theta} 
Assume that $i{<}j{<}r{<}s{<}t$. Then 
\begin{align}
q\Psi^{(j,s,t)}\Phi^{(i,r)}&=q^{-1}\Phi^{(i,r)}\Psi^{(j,s,t)} -(q{-}q^{-1})
\Phi^{(i,j)}\Psi^{(r,s,t)} \\
&\phantom{X}-(q{-}q^{-1})
\left(
{-}(q^2{+}q^{-2})\Phi^{(j,r)}\Psi^{(i,s,t)}{+}\Theta^{(i,j,r,s,t)}
\right).\nonumber
\end{align}
\end{lemma}
Note that  that $q$-commutators of $\Psi^{(j,s,t)}$ and $\Phi^{(i,r)}$ produce the
invariants $\Theta^{(i,j,r,s,t)}$.

\subsubsection{Comments}\label{comments}

We have also worked out all commutation relations among the elements $\Psi^{(r,s,t)}$,  by using a similar diagrammatic method as that used in the proofs of the previous lemmas. The final results are rather messy,  thus we choose not to include them here. However, the work reveals the important fact  that $q$-commutators of $\Psi^{(r,s,t)}$'s give rise to some 
invariants in  $\cup_{|{\bf d}|=6} \widehat{\Acycl}_{\bf d}$. 
The same invariants also arise from $q$-commutators of 
$\Phi^{(a,b)}$ and $\Upsilon^{(i,k,j,l)}$. 
We have checked that  for some indices, $\Psi^{(r,s,t)}$ and $\Upsilon^{(i,k,j,l)}$ do not commute or $q$-commute, thus they can potentially generate new invariants of degree $7$. 

Let us call an ordered set of homogeneous invariants a PBW set if ordered monomials 
of the elements span $\cA_m^\UG$, and no proper subset has this property. 
Then a PBW set necessarily contains the elements $\Phi^{(a,b)}$, $\Psi^{(r,s,t)}$ and $\Upsilon^{(i,k,j,l)}$, as one can show by considering the $m=4$ case. 
It appears that for bigger $m$, a PBW set should also include $\Theta^{(i,j,r,s,t)}$, and possibly 
the degree $6$ invariants mentioned above if $m\ge 6$. 

It is intriguing that $q$-commutators of some lower degree invariants can 
produce higher degree ones (e.g., $\Upsilon^{(i,k,j,l)}$), which are indispensable elements of PBW sets. 
This is a new phenomenon not observed before in the context of invaraint theory, even for quantum groups. 
On the other hand, this is not that surprising as the subalgebra $\cA_m^\UG$ of invariants is non-commutative. 

However, we feel that the degrees of the elements of PBW sets should be bounded by some finite integer, which does not depend on $m$ for large $m$.   We hope to further investigate this in a future work.

\bigskip
\noindent{\bf Declarations}. 
The authors have no conflict of interest to declare

\bigskip
	

\begin{thebibliography}{9999}
	\footnotesize\itemsep=0pt

\bibitem{APW}
Andersen, H. H; Polo, P.; Wen, K. X.:
Representations of quantum algebras.
{\sl Invent. Math. \bf 104} (1991), no. 1, 1--59.
	
\bibitem{benkart}
Benkart, G.;  Elduque A.: Cross products, invariants, and centralizers.
{\sl J. Algebra \bf 500} (2018), 69--102.
	
\bibitem{bz:bs}
Berenstein, A.; Zwicknagl, S.: Braided symmetric and exterior algebras. 
{\sl Trans. Amer. Math. Soc. \bf 360} (2008), 3429--3472.

\bibitem{b:dcb}
Brundan, J.: Dual canonical bases and Kazhdan-Lusztig polynomials. 
{\sl J. Algebra \bf 306} (2006), 17--46.

\bibitem{BSW}
Brundan, J.;  Savage, A.;  Webster, B.:
On the definition of quantum Heisenberg category.
{\sl Algebra and  Number theory, \bf 14} (2020), no. 2, 275--321.
		
\bibitem{CW-queer} Chang, Z. H.; Wang, Y. J:
Howe duality for quantum queer superalgebras. 
{\sl J. Algebra \bf 547} (2020), 358--378.

\bibitem{fy:bc}Freyd, P. J.; Yetter, D. N.:  
Braided compact closed categories with applications to low-dimensional topology. 
{\sl Adv. Math. \bf 77} (1989)  (2), 156--182.


\bibitem{GZ}  
Gover, A. R.; Zhang, R. B.:
Geometry of quantum homogeneous vector bundles and representation theory of quantum groups. I. 
{\sl Rev. Math. Phys. \bf 11} (1999), no. 5, 533--552.

\bibitem{HZ}
 Huang, J.-S.; Zhu, C.-B. : Weyl’s construction and tensor power decomposition for $G_2$. 
 {\sl Proc. Amer. Math.  Soc. \bf 127}  (1999), 925--934.
			
\bibitem{jan:qg} Jantzen, J. C.: 
Lectures on quantum groups. Graduate Studies in Mathematics, \textbf{6}, Providence, RI: Amer. Math. Soc., 1996.

\bibitem{J}
 Jones, V. F. R.: Hecke algebra representations of braid groups and link polynomials. 
 {\sl Ann. of Math. (2) \bf 126} (1987), no. 2, 335--388.
	
\bibitem{kuper:q} Kuperberg G.: 
The quantum $G_2$ link invariant.
{\sl Internat. J. Math. \bf 5} (1994), 61--85.
	
\bibitem{kuper:s}Kuperberg G.:
Spiders for Rank $2$ Lie algebras. 
{\sl Comm. Math. Phys. \bf 180} (1996), 109--151.
	
\bibitem{Lai-Z}  Lai, K. F.;  Zhang, R. B.
Multiplicity Free Actions of Quantum Groups
and Generalized Howe Duality.
{\sl Lett. Math. Physics \bf 64} (2003) 255--272.

\bibitem{lz:sm}Lehrer, G. I.; Zhang, R. B.: 
Strongly multiplicity free modules for Lie algebras and quantum groups. 
{\sl J. Algebra \bf 306} (2006), no. 1, 138--174.
	
\bibitem{lzz:ft}
Lehrer, G. I.; Zhang, H. C.; Zhang, R. B.:
A quantum analogue of the first fundamental theorem of classical invariant theory. 
{\sl Comm. Math. Phys. \bf 17} (2011), 131--174.	

\bibitem{LZZ-2}Lehrer, G. I.; Zhang, H. C.; Zhang, R. B.:
 First fundamental theorems of invariant theory for quantum supergroups. 
{\sl Eur. J. Math. \bf 6} (2020), no. 3, 928--976.
	
\bibitem{LZ-bc}Lehrer, G. I.;  Zhang, R. B.:
 The Brauer category and invariant theory. 
 {\sl J. Eur. Math. Soc. \bf 17} (2015), 2311--2351.

\bibitem{LZ-rev}Lehrer, G. I.;  Zhang, R. B.:
 Diagram categories and invariant theory for classical groups and supergroups. 
{\sl Proceedings of the International Consortium of Chinese Mathematicians 2019}, August 2024. 

\bibitem{LSS}Letzter, G.; Sahi, S.; Salmasian, H.: 
Quantized Weyl algebras, the double centralizer property, and a new first fundamental theorem for ${\rm U}_q(\mathfrak{gl}_n)$. 
{\sl J. Phys. \bf A 57} (2024), no. 19, Paper No. 195304, 67 pp.

\bibitem{MW}
Mikhaylov, V.; Witten, E.: 
Branes and Supergroups.
{\sl Commun. Math. Phys. \bf 340} (2015),  699--832.

\bibitem{m:ha}Montgomery, S.:
Hopf algebras and their actions on rings. 
CBMS Regional Conference Series in Mathematics, 82. Published for the Conference Board of the Mathematical Sciences, Washington, DC; by the American Mathematical Society, Providence, RI, 1993.

\bibitem{Ms}
Morrison, S.: The braid group surjects onto  $G_2$  tensor space.
{\sl Pacific J. Math.\bf 249} (2011), no. 1, 189--198.

	
\bibitem{rt:rg} Reshetikhin, N. Yu.; Turaev, V. G.:
Ribbon graphs and their invariants derived from quantum groups.
{\sl Comm. Math. Phys. \bf 127} (1990), 1--26.

\bibitem{rt:in}Reshetikhin, N.; Turaev, V. G.: 
Invariants of 3-manifolds via link polynomials and quantum groups. 
{\sl Invent. Math. \bf 103} (1991), 547--597. 
	
\bibitem{r:ar}Rossi-Doria, O.:
A ${\rm U}_q(sl(2))$-representation with no quantum symmetric algebra.
{\sl Rend. Mat. Acc. Linceis. \bf 10} (1999), 5--9.

\bibitem{SW} Savage, A.; Westbury, B. W.: 
Quantum diagrammatics for $F_4$. {\sl J. Pure  Applied Algebra \bf 228} (2024), no. 11, 107731.
	
\bibitem{S}	 Schwarz, G. W.:
Invariant theory of $G_2$ and ${\rm Spin}_7$. 
{\sl Comment. Math. Helv. \bf 63} (1988) 624--663.
		
\bibitem{Wh} Wenzl, H,:
On tensor categories of Lie Type $E_N$, $N\ne 9$. 
{\sl Adv. Math. \bf 177} (2003),  66--104.

\bibitem{W-cl}Weyl, H.: 
The classical groups. Their invariants and representations. Fifteenth printing. Princeton Landmarks in Mathematics. Princeton University Press, Princeton, NJ, 1997.		
		
\bibitem{Witten}Witten, E.:  
Quantum field theory and the Jones polynomial. 
{\sl Comm. Math. Phys. \bf 121} (1989), no. 3, 351--399.

\bibitem{WZ-hd}Wu, Y. Z.; Zhang, R. B.:
Unitary highest weight representations of quantum general linear superalgebra. 
{\sl J. Algebra \bf 321} (2009), no. 11, 3568--3593.

\bibitem{z:hd} Zhang, R. B.:
 Howe duality and the quantum general linear group.
 {\sl Proc. Amer. Math. Soc. \bf131} (2003), 2681--2692. 
	

\bibitem{Zr}  
Zhang, R. B.:
Topological invariants for Lens spaces and exceptional quantum groups. 
{\sl Lett. Math. Phys. \bf 41} (1997), no. 1, 1--11.


\bibitem{Zy}  Zhang, Y.:	
The first and second fundamental theorems of invariant theory for the quantum general linear supergroup. 
{\sl J Pure Applied Alg. \bf 224} (2020), 106411.

\bibitem{ZS} 
Zubkov, A. N.;  Shestakov, I. P.:
Invariants of $G_2$ and Spin(7) in positive characteristic.
{\sl Transformation Groups \bf 23} (2018), 555--588. 


\bibitem{z:pa}Zwicknagl, S.:
$R$-matrix Poisson algebras and their deformations.
{\sl Adv. Math. \bf 220} (2009), 1--58.
\end{thebibliography}


\end{document}